\documentclass[12pt,oneside,english]{amsart}
\usepackage[T1]{fontenc}
\usepackage[latin9]{inputenc}
\usepackage{color}
\usepackage{fancybox}
\usepackage{calc}
\usepackage{mathrsfs}
\usepackage{amsthm}
\usepackage{amsbsy}
\usepackage{amssymb}

\usepackage{tikz}
\usetikzlibrary{arrows}
\usetikzlibrary{decorations.pathreplacing}

\allowdisplaybreaks

\makeatletter
\numberwithin{equation}{section}
\numberwithin{figure}{section}

  \theoremstyle{plain}
  \newtheorem{thm}{\protect\theoremname}[section]
	\theoremstyle{remark}
  \newtheorem{rem}[thm]{\protect\remarkname}
  \theoremstyle{plain}
  \newtheorem{lem}[thm]{\protect\lemmaname}
  \theoremstyle{definition}
  \newtheorem{defn}[thm]{\protect\definitionname}
  \theoremstyle{plain}
  \newtheorem{prop}[thm]{\protect\propositionname}
  \theoremstyle{plain}
  \newtheorem{cor}[thm]{\protect\corollaryname}

 \theoremstyle{remark}
  \newtheorem*{rem*}{\protect\remarkname}
  \theoremstyle{plain}
  \newtheorem*{lem*}{\protect\lemmaname}
 \theoremstyle{definition}
 \newtheorem*{defn*}{\protect\definitionname}
  \theoremstyle{plain}
  \newtheorem*{prop*}{\protect\propositionname}
  \theoremstyle{plain}
  \newtheorem*{thm*}{\protect\theoremname}

\makeatother

\usepackage{babel}
  \providecommand{\definitionname}{Definition}
  \providecommand{\lemmaname}{Lemma}
  \providecommand{\propositionname}{Proposition}
  \providecommand{\corollaryname}{Corollary}
  \providecommand{\remarkname}{Remark}
  \providecommand{\theoremname}{Theorem}

\begin{document}
\global\long\def\cE{\mathcal{E}}
\global\long\def\cN{\mathcal{N}}
\global\long\def\cR{\mathcal{R}}

\global\long\def\al{\alpha}
\global\long\def\be{\beta}
\global\long\def\ga{\gamma}
\global\long\def\de{\delta}
\global\long\def\ep{\epsilon}
\global\long\def\la{\lambda}

\global\long\def\ka{\mbox{\large{$\kappa$}}}

\global\long\def\alb{\bar{\alpha}}
\global\long\def\beb{\bar{\beta}}
\global\long\def\gab{\bar{\gamma}}
\global\long\def\deb{\bar{\delta}}
\global\long\def\epb{\bar{\epsilon}}
\global\long\def\mub{\bar{\mu}}
\global\long\def\nub{\bar{\nu}}
\global\long\def\lab{\bar{\lambda}}
\global\long\def\rhob{\bar{\rho}}

\global\long\def\Ab{\bar{A}}
\global\long\def\Bb{\bar{B}}
\global\long\def\Cb{\bar{C}}
\global\long\def\Db{\bar{D}}
\global\long\def\Eb{\bar{E}}

\global\long\def\Ib{\bar{I}}
\global\long\def\Ibp{\bar{I\:\!\!'}}

\global\long\def\thetah{\hat{\theta}}
\global\long\def\nablah{\hat{\nabla}}
\global\long\def\cnabla{\check{\nabla}}
\global\long\def\bL{\mathbb{L}}

\global\long\def\btheta{\boldsymbol{\theta}}
\global\long\def\bh{\boldsymbol{h}}

\global\long\def\La{\Lambda}
\global\long\def\Si{\Sigma}
\global\long\def\si{\sigma}
\global\long\def\vsig{\varsigma}
\global\long\def\ups{\upsilon}
\global\long\def\Ups{\Upsilon}
\global\long\def\II{I\! I}
\global\long\def\rmN{\mathrm{N}}
\global\long\def\sfH{\mathsf{H}}
\global\long\def\sfS{\mathsf{S}}
\global\long\def\sfM{\mathsf{M}}
\global\long\def\bbD{\mathbb{D}}

\global\long\def\bbR{\mathbb{R}}
\global\long\def\bbC{\mathbb{C}}
\global\long\def\bbZ{\mathbb{Z}}
\global\long\def\scrK{\mathscr{K}}
\global\long\def\id{\mathrm{id}}
\global\long\def\fX{\mathfrak{X}}
\global\long\def\d{\mathrm{d}}

\global\long\def\cT{\mathcal{T}}
\global\long\def\cA{\mathcal{A}}
\global\long\def\cG{\mathcal{G}}
\global\long\def\cH{\mathcal{H}}
\global\long\def\cU{\mathcal{U}}
\global\long\def\SU{\mathrm{SU}}
\global\long\def\End{\mathrm{End}}

\global\long\def\fg{\mathfrak{g}}
\global\long\def\fp{\mathfrak{p}}
\global\long\def\fh{\mathfrak{h}}

\global\long\def\rarrow{\rightarrow}
\global\long\def\hrarrow{\hookrightarrow}
\global\long\def\hook{\lrcorner}

\newcommand{\lpl}{
  \mbox{$
  \begin{picture}(12.7,8)(-.5,-1)
  \put(2,0.2){$+$}
  \put(6.2,2.8){\oval(8,8)[l]}
  \end{picture}$}}

\def\sideremark#1{\ifvmode\leavevmode\fi\vadjust{\vbox to0pt{\vss
 \hbox to 0pt{\hskip\hsize\hskip1em
 \vbox{\hsize3cm\tiny\raggedright\pretolerance10000
  \noindent #1\hfill}\hss}\vbox to8pt{\vfil}\vss}}}%

                                                   %
\newcommand{\edz}[1]{\sideremark{#1}}
\def\idx#1{{\em #1\/}}

\title{CR Embedded Submanifolds of CR Manifolds}

\author{Sean N.\ Curry and A.\ Rod Gover}
\address{S.N.C. \& A.R.G.:Department of Mathematics\\
  The University of Auckland\\
  Private Bag 92019\\
  Auckland 1142\\
  New Zealand} 
\email{sean.curry@auckland.ac.nz}
\email{r.gover@auckland.ac.nz}

\begin{abstract}
We develop a complete local theory for CR embedded submanifolds of CR manifolds in a way which parallels
the Ricci calculus for Riemannian submanifold theory. We define a
normal tractor bundle in the ambient standard tractor bundle along the
submanifold and show that the orthogonal complement of this bundle is
not canonically isomorphic to the standard tractor bundle of the
submanifold. By determining the subtle relationship between
submanifold and ambient CR density bundles we are able to invariantly
relate these two tractor bundles, and hence to invariantly relate the
normal Cartan connections of the submanifold and ambient manifold by a
tractor analogue of the Gauss formula. This leads also to CR analogues 
of the Gauss, Codazzi, and Ricci equations. The tractor Gauss formula
includes two basic invariants of a CR embedding which, along with the
submanifold and ambient curvatures, capture the jet data of the
structure of a CR embedding. These objects therefore form the basic
building blocks for the construction of local invariants of the
embedding. From this basis we develop a broad calculus for the
construction of the invariants and invariant differential operators of
CR embedded submanifolds.

The CR invariant tractor calculus of CR embeddings is developed
concretely in terms of the Tanaka-Webster calculus of an arbitrary
(suitably adapted) ambient contact form. This enables straightforward
and explicit calculation of the pseudohermitian invariants of the
embedding which are also CR invariant. These are extremely difficult
to find and compute by more na\"ive methods. We conclude by
establishing a CR analogue of the classical Bonnet theorem in
Riemannian submanifold theory.
\end{abstract}

\thanks{ARG acknowledges support from the Royal
  Society of New Zealand via Marsden Grant 13-UOA-018}

\maketitle

\section{Introduction}

Hypersurface type CR geometry is motivated by the biholomorphic
equivalence problem for complex domains, and is rooted in the result
of Poincar\'e that the analogue of the Riemann mapping theorem fails
for domains of complex dimension greater than one \cite{Poincare}.
 On the side of geometry key pioneering work was
developed by Cartan, Tanaka, and Chern-Moser in which it was seen that the
structure is invariantly captured by a prolonged system now known as a
Cartan connection \cite{Cartan-RealHypInC2,Chern-Moser,Tanaka-pseudoConfConnection}. The fundamental role of CR geometry in analysis was significantly strengthened by the result
of Fefferman that any biholomorphic map between smoothly bounded strictly pseudoconvex domains in
$\bbC^{n+1}$ extends smoothly to the boundary, and so induces a CR
diffeomorphism between the boundaries \cite{FeffermanBihol}; so
Poincar\'e's result may be recovered by a simple counting of
invariants argument (that was in fact proposed in \cite{Poincare}). This
brought to the fore the role of CR invariants as tools for
distinguishing domains. Hypersurface type CR geometry is an important
example in a class of structures known as parabolic geometries that
also includes conformal geometry, projective differential geometry,
and many other structures. Seeking to determine the asymptotic
expansion of the Bergman kernel, Fefferman initiated a programme for
the explicit construction of CR, and more widely parabolic, invariants
\cite{FeffermanParabolic}. There has subsequently been much interest and
progress on this \cite{BaileyEastwoodGraham,Hirachi,Gover-Advances}.

The study of CR embeddings and immersions (in CR manifolds) is also
closely connected with the study of holomorphic mappings between
domains. Although open questions remain about when proper holomorphic
mappings between domains in $\bbC^{m+1}$ and $\bbC^{n+1}$ extend
smoothly \cite{Bedford, Webster-MappingBalls}, if a holomorphic map between smoothly
bounded domains does extend in this way then it induces a CR map
between the boundaries.  So again CR invariants of the boundaries play
a fundamental role.  The Chern-Moser moving frames approach to the CR
Cartan connection has been effectively applied to the study of CR
embeddings and immersions in the important work of Webster
\cite{Webster-Rigidity} on CR rigidity for real codimension two
embeddings. This theme is significantly extended in the article
\cite{EbenfeltHuangZaitsev-Rigidity} of Ebenfelt, Huang and Zaitsev
where rigidity is established when the codimension is not too
large. These works have strong applications to the study of proper
holomorphic maps between balls and to the study of Milnor links of
isolated singularities of analytic varieties
\cite{Webster-MappingBalls,EbenfeltHuangZaitsev-Rigidity}. The Chern-Moser 
approach has also been applied in related work generalising the Schwarz 
reflection principle to several complex variables, where invariant 
nondegeneracy conditions on CR maps play a key role 
\cite{Faran-Reflection,Lamel-Reflection}.

Despite the strong specific results mentioned, and geometric studies by several authors \cite{Dragomir,DragomirMinor,DragomirTomassini,Markowitz-CR,Minor,Takemura}, a significant gap has
remained in the general theory for CR embeddings and immersions. A
basic general theory should enable the straightforward construction of
local CR invariants, but in fact to this point very few invariants are
known. In particular using existing approaches there has been no scope
for a general theory of invariant construction, as the first step in a
Fefferman-type invariants programme cf.\ \cite{FeffermanParabolic}. Closely related is the need to
construct CR invariant differential operators required for geometric
analysis. Again no general theory for their construction has been
previously advanced. The aim of this article is to close this gap. We
develop a general CR invariant treatment that on the one hand is
conceptual and on the other provides a practical and constructive
approach to treating the problems mentioned. The final package may be
viewed as, in some sense, an analogue of the usual Ricci calculus
approach to Riemannian submanifold theory, which is in part based
around the Gauss formula. Our hope is that this may be easily used by
analysts or geometers not already strongly familiar with CR geometry;
for this reason we have attempted to make the treatment largely self
contained. The theory and tools developed here may also be viewed as
providing a template for the general problem of treating parabolic
submanifolds in parabolic geometries. This is reasonably well
understood in the conformal setting
\cite{BaileyEastwoodGover-Thomas'sStrBundle,BurstallCalderbank-Conformal,Gover-AlmostEinstein,Grant,Stafford,Vyatkin-Thesis} but little is known
in the general case. The CR case treated here is considerably more
subtle than the conformal analogue as it involves dealing with a
non-maximal parabolic.

\subsection{CR Embeddings}
Abstractly, a nondegenerate hypersurface-type CR manifold is a smooth manifold $M^{2n+1}$ equipped with a contact distribution $H$ on which there is a formally integrable complex structure $J:H\rightarrow H$. We refer to such manifolds simply as \emph{CR manifolds}. A \emph{CR mapping} between two CR manifolds is a smooth mapping whose tangent map restricts to a complex linear bundle map between the respective contact distributions. A \emph{CR embedding} is a CR mapping which is also an embedding. 

Typically in studying CR embeddings one works with an arbitrary choice of contact form for the contact distribution in the ambient manifold (ambient \emph{pseudohermitian structure}). A CR embedding is said to be \emph{transversal} if at every point in the submanifold there is a submanifold tangent vector which is transverse to the ambient contact distribution (this is automatic if the ambient manifold is strictly pseudoconvex). Assuming transversality the ambient contact form then pulls back to a pseudohermitian contact form on the submanifold. Associated with these contact forms are their respective Tanaka-Webster connections, and these can be used to construct pseudohermitian invariants of the embedding. The task of finding some, let alone all, pseudohermitian invariants which are in fact CR invariants (not depending on the additional choice of ambient contact form) is very difficult, unless one can find a manifestly invariant approach. We give such an approach. Our approach uses the natural invariant calculus on CR manifolds, the \emph{CR tractor calculus}. In the CR tractor calculus the \emph{standard tractor bundle} and \emph{normal tractor (or Cartan) connection} play the role analogous to the (holomorphic) tangent bundle and Tanaka-Webster connection in pseudohermitian geometry.

\subsection{Invariant Calculus on CR Manifolds}

Due to the work of Cartan, Tanaka, and Chern-Moser we may view a CR
manifold $(M,H,J)$ as a Cartan geometry of type $(G,P)$ with $G$ a
pseudo-special unitary group and $P$ a parabolic subgroup of $G$. The
\emph{tractor bundles} are the associated vector bundles on $M$
corresponding to representations of $G$, the standard tractor bundle
corresponding to the standard representation. The normal Cartan
connection then induces a linear connection on each tractor bundle
\cite{CapGover-TracCalc}. In order to relate the CR tractor calculus
to the Tanaka-Webster calculus of a choice of pseudohermitian contact
form we work with the direct construction of the CR standard tractor
bundle and connection given in \cite{GoverGraham-CRInvtPowers}. This
avoids the need to first construct the Cartan bundle.

 To fully treat CR submanifolds one
needs to work with \emph{CR density line bundles}, and their
Tanaka-Webster calculus. From the Cartan geometric point of view the
CR density bundles $\cE(w,w')$ on $M^{2n+1}$ are the complex line
bundles associated to one dimensional complex representations of $P$,
and include the \emph{canonical bundle} $\scrK$ as $\cE(-n-2,0)$. The
bundle $\cE(1,0)$ is the dual of an $(n+2)^{th}$ root of $\scrK$ and
\begin{equation*}
\cE(w,w')=\cE(1,0)^w\otimes \overline{\cE(1,0)}^{w'}
\end{equation*}
where $w-w'\in \mathbb{Z}$ (and $w,w'$ may be complex). Since the Tanaka-Webster connection acts on the canonical bundle it acts on all the density bundles.

\subsection{Invariant Calculus on Submanifolds and Main Results}

We seek to extend the CR tractor calculus to the setting of transversally CR embedded submanifolds of CR manifolds in order to deal with the problem of invariants. Our approach parallels the usual approach to Riemannian submanifold geometry; of central importance in the Riemannian theory of submanifolds is the second fundamental form.

\subsubsection{Normal tractors and the tractor second fundamental form} One way to understand the Riemannian second fundamental form is in terms of the turning of normal fields (i.e. as the shape operator). To define a tractor analogue of the shape operator one needs a tractor analogue of the normal bundle for a CR embedding $\iota:\Sigma^{2m+1}\rightarrow M^{2n+1}$. 

In Section \ref{sec:NormalTrac} we use a CR invariant differential splitting operator to give a CR analogue of the normal tractor of \cite{BaileyEastwoodGover-Thomas'sStrBundle} associated to a weighted unit normal field in conformal submanifold geometry. It turns out that this \emph{a priori} differential splitting gives a canonical bundle isomorphism between the Levi-orthogonal complement of $T^{1,0}\Si$ in $T^{1,0}M|_{\Si}$, tensored with the appropriate ambient density bundle, and a subbundle $\cN$ of the ambient standard tractor bundle along $\Si$ (Proposition \ref{NormalTractorConnectionProp}). The ambient standard tractor bundle carries a parallel Hermitian metric of indefinite signature and the \emph{normal tractor bundle} $\cN$ is nondegenerate since $\Si$ and $M$ are required to be nondegenerate. Thus the ambient tractor connection induces connections  $\nabla^{\cN}$ and $\nabla^{\cN^{\perp}}$ on $\cN$ and $\cN^{\perp}$ respectively. We therefore obtain (Section \ref{TractorGCR}, see also Sections \ref{sec:NormalTrac}, \ref{TractorGaussFormula}, and \ref{RelatingTractorsHC}):
\begin{prop}\label{NormalTractorsAndTractorSFFResult}
The ambient standard tractor bundle $\cT M$ splits along $\Si$ as $\cN^{\perp}\oplus\cN$, and the ambient tractor connection $\nabla$ splits as 
\begin{equation*}
\iota^*\nabla = \left(\begin{array}{cc} 
\nabla^{\cN^{\perp}}  &  -\bL^{\dagger}   \\
 \bL  &   \nabla^{\cN}  \\ 
\end{array}\right) 
\qquad \mathrm{on} \qquad \cT M|_{\Si} = \begin{array}{c} 
\phantom{i}\cN^{\perp}  \\
\oplus \\
\cN
\end{array}
\end{equation*}
where $\bL^{\dagger}(X)$ is the Hermitian adjoint of $\bL(X)$ for any $X\in\mathfrak{X}(\Si)$.
\end{prop}
The $\mathrm{Hom}(\cN,\cN^{\perp})$ valued 1-form $\bL^{\dagger}$ on
$\Si$ is the CR tractor analogue of the shape operator, and we term
$\bL$ the \emph{CR tractor second fundamental form}. The ambient
standard tractor bundle can be decomposed with respect to a choice of
contact form. Here it is sensible to choose an ambient contact form
whose Reeb vector field is tangent to the submanifold (called
\emph{admissible} \cite{EbenfeltHuangZaitsev-Rigidity}). We give the
components of $\bL$ with respect to an admissible ambient contact form
in Proposition \ref{tractor-sff-prop-HC} (see also Proposition
\ref{tractor-sff-prop}). The principal component of $\bL$ is the
\emph{CR second fundamental form} $\II_{\mu\nu}{^{\ga}}$ of $\Si$ in
$M$, which appears, for example, in
\cite{EbenfeltHuangZaitsev-Rigidity}.

\subsubsection{Relating submanifold and ambient densities and tractors}
Another way to understand the Riemannian second fundamental form is in terms of the normal part of the ambient covariant derivative of a submanifold vector field in tangential directions. This is achieved via the Gauss formula. In the Riemannian Gauss formula a submanifold vector field is regarded as an ambient vector field along the submanifold using the pushforward of the embedding, which relies on the tangent map. In order to give a CR tractor analogue of the Gauss formula one needs to be able to pushforward submanifold tractors to give ambient tractors along the submanifold -- one looks for a CR `standard tractor map'. One might hope for a canonical isomorphism
\begin{equation*}
\cT \Si \rightarrow \cN^{\perp}
\end{equation*}
between the submanifold standard tractor bundle and the orthogonal complement of the normal tractor bundle (these having the same rank). In the conformal case there is such a canonical isomorphism \cite{BurstallCalderbank,Gover-AlmostEinstein,Grant}, however in the CR case it turns out that there is no natural `standard tractor map' $\cT\Si \rightarrow \cT M$ in general.

The problem has to do with the necessity of relating corresponding submanifold and ambient CR density bundles. It turns out that these are not isomorphic along the submanifold, but are related by the top exterior power of the normal tractor bundle $\cN$. Rather than seeking to identify these bundles we therefore define the ratio bundles of densities
\begin{equation*}
\cR(w,w')=\cE(w,w')|_{\Si}\otimes \cE_{\Si}(w,w')^*
\end{equation*}
where $\cE(w,w')|_{\Si}$ is a bundle of ambient CR densities along
$\Si$ and $\cE_{\Si}(w,w')^*$ is dual to the corresponding submanifold
intrinsic density bundle. We obtain (in Section
\ref{RelatingDensities-HC}, see also Section
\ref{CanonicalConnectionOnRatio}):
\begin{prop}\label{RatioAndNormalBundlesResult}
Given a transversal CR embedding $\iota:\Si^{2m+1}\rightarrow M^{2n+1}$ we have a canonical isomorphism of complex line bundles
\begin{equation*}
\cR(m+2,0) \cong \Lambda^{d}\cN
\end{equation*} 
where $d=n-m$. The complex line bundles $\cR(w,w')$ therefore carry a canonical CR invariant connection $\nabla^{\cR}$ induced by $\nabla^{\cN}$.
\end{prop}
The diagonal bundles $\cR(w,w)$ are canonically trivial and the connection
$\nabla^{\cR}$ on these is flat. All the ratio bundles $\cR(w,w')$ are therefore
normed, and $\nabla^{\cR}$ is a $\mathrm{U}(1)$-connection. Using the
pseudohermitian Gauss and Ricci equations (Sections
\ref{PseudohermitianGaussSect} and \ref{PseudohermitianRicciSect}) we
calculate the curvature of $\nabla^{\cR}$ (Section
\ref{RelatingDensities-HC}, see also Section
\ref{CanonicalConnectionOnRatioCurvature}, and in particular Lemma
\ref{SchoutenDifferenceLemma}) and see that this connection is not
flat in general when $w\neq w'$. Thus rather than identifying
corresponding density bundles we should keep the ratio bundles
$\cR(w,w')$ in the picture.

We are then able to show (from Theorem \ref{DiagonalTractorMap} combined with Definitions \ref{TwistedTractorMapDefn}, \ref{difference-tractor-defn}, \ref{tractor-sff-defn} and Section \ref{RelatingTractorsHC}):
\begin{thm}\label{TwistedTractorMapResult}
Let $\iota:\Sigma^{2m+1} \rightarrow M^{2n+1}$ be a transversal CR embedding.  Then there is a canonical, metric and filtration preserving, bundle map
\begin{equation*}
\cT^{\cR}\iota: \cT\Si  \rightarrow \cT M|_{\Si} \otimes \cR(1,0)
\end{equation*}
over $\iota$, which gives an isomorphism of $\cT\Si$ with $\cN^{\perp}\otimes\cR(1,0)$. Moreover, the submanifold and ambient tractor connections are related by the tractor Gauss formula
\begin{equation*}
\nabla_X \iota_* u \: = \: \iota_*(D_X u + \sfS(X) u) + \bL(X) \iota_* u
\end{equation*}
for all $u\in\Gamma(\cT\Si)$ and $X\in \mathfrak{X}(\Si)$, where $\sfS$ is an $\mathrm{End}(\cT\Si)$ valued 1-form on $\Si$, $D$ is the submanifold tractor connection, $\nabla$ is the (pulled back) ambient tractor connection coupled with $\nabla^{\cR}$, and the pushforward map $\iota_*$ is defined using using $\cT^{\cR}\iota$.
\end{thm}
By Proposition \ref{NormalTractorsAndTractorSFFResult} the tractor Gauss formula implies
\begin{equation*}
\nabla^{\cN^{\perp}}_X \iota_* u \: = \: \iota_*(D_X u + \sfS(X) u)
\end{equation*}
for all $u\in\Gamma(\cT\Si)$ and $X\in \mathfrak{X}(\Si)$, where
$\nabla^{\cN^{\perp}}$ is coupled with $\nabla^{\cR}$. The
\emph{difference tractor} $\sfS$ measures the failure of the ambient
tractor (or normal Cartan) connection to induce the submanifold
one. The components of $\sfS$ with respect to an admissible ambient
contact form are given in \eqref{hDifferenceTractor},
\eqref{aDifferenceTractor}, and \eqref{0DifferenceTractor} (in Section
\ref{RelatingTractorsHC} it is noted that these formulae hold in
arbitrary codimension and signature). The principal component of
$\sfS$ is the difference between ambient and submanifold
pseudohermitian Schouten tensors $P_{\mu\nub}-p_{\mu\nub}$ for a
pair of \emph{compatible} contact forms (Definition
\ref{CompatibleScalesDef}); using the pseudohermitian Gauss equation
(Section \ref{PseudohermitianGaussSect}) one can give a manifestly
invariant expression for this tensor involving the ambient Chern-Moser
tensor and the CR second fundamental form (see Lemma
\ref{SchoutenDifferenceLemma} for the case $m=n-1$).

\subsubsection{Constructing invariants}

In Section \ref{InvariantsOfSubmanifolds} we develop both the
theoretical and practical aspects of constructing invariants of CR
embeddings. We deal with the geometric part of the invariant theory
problem, using the results stated above. In particular, in Section
\ref{JetsOfStructure} we demonstrate that the tractor second
fundamental form $\bL$, the difference tractor $\sfS$, and the
submanifold and ambient tractor (or Cartan) curvatures are the basic
invariants of the CR embedding, in that they determine the higher jets
of the structure (Proposition \ref{BasicInvariantsProp}). By applying
natural differential operators to these objects and making suitable
contractions, one can start to proliferate local invariants of a CR
embedding. In practice a more refined construction is useful. The
algebraic problem of showing that one can make all invariants of a CR
embedding, suitably polynomial in the jets of the structure, is beyond
the scope of this article; despite much progress on the analogous
problems for CR or conformal manifolds, these are still far from being
completely solved (see, e.g., \cite{BaileyEastwoodGraham,Hirachi}). We
therefore turn in Section \ref{PracticalConstructions} to considering
practical constructions of invariants. In Sections
\ref{AlternativeTractorExp} and \ref{SubmanifoldInvtOps} we develop a
richer calculus of invariants than that presented for theoretical
purposes in Sections \ref{JetsOfStructure} and \ref{PackagingJets}. In
Section \ref{ComputingInvariants} we illustrate this calculus with an
example of an invariant section $\mathcal{I}$ of $\cE_{\Si}(-2,-2)$
given by a manifestly invariant tractor expression
\eqref{ExampleInvariant} involving $\bL\otimes \overline{\bL}$; we
show how to calculate $\mathcal{I}$ in terms of the pseudohermitian
calculus of a pair of compatible contact forms, yielding the
expression \eqref{ExampleInvariantComputed}.

\subsubsection{A CR Bonnet theorem}

With the setup of Proposition \ref{NormalTractorsAndTractorSFFResult} and Theorem \ref{TwistedTractorMapResult} established it is straightforward to give CR tractor analogues of the Gauss, Codazzi and Ricci equations from Riemannian submanifold theory. These are given in Section \ref{TractorGCR}. Just as in the Riemannian theory, if we specialise to the ambient flat case the (tractor) Gauss, Codazzi and Ricci equations give the integrability conditions for a Bonnet theorem or fundamental theorem of embeddings. We have (Theorem \ref{BonnetTheorem}):
\begin{thm}
Let $(\Sigma^{2m+1},H,J)$ be a signature $(p,q)$ CR manifold and suppose we have a complex rank $d$ vector bundle $\mathcal{N}$ on $\Sigma$ equipped with a signature $(p',q')$ Hermitian bundle metric $h^{\mathcal{N}}$ and metric connection $\nabla^{\mathcal{N}}$. Fix an $(m+2)^{th}$ root $\mathcal{R}$ of $\Lambda^d\mathcal{N}$, and let $\nabla^{\mathcal{R}}$ denote the connection induced by $\nabla^{\mathcal{N}}$. Suppose we have a $\mathcal{N}\otimes\mathcal{T}^*\Sigma\otimes\mathcal{R}$ valued 1-form $\bL$ which annihilates the canonical tractor of $\Si$ and an $\mathcal{A}^0\Sigma$ valued 1-form $\mathsf{S}$ on $\Sigma$ such that the connection
\begin{equation*}
\nabla := \left(\begin{array}{cc} 
D\otimes\nabla^{\mathcal{R}}+ \mathsf{S}  &  -\mathbb{L}^{\dagger}   \\
 \mathbb{L}  &   \nabla^{\mathcal{N}}  \\ 
\end{array}\right) 
\qquad \mathrm{on} \quad \begin{array}{c} 
\mathcal{T}\Sigma \otimes \mathcal{R}^*  \\
\oplus \\
\mathcal{N}\\
\end{array}
\end{equation*}
is flat (where $D$ is the submanifold tractor connection), then (locally) there exists a transversal CR embedding of $\Sigma$ into the model $(p+p',q+q')$ hyperquadric $\mathcal{H}$, unique up to automorphisms of the target, realising the specified extrinsic data as the induced data.
\end{thm}
The bundle $\cA^0\Si$ here is the bundle of skew-Hermitian endomorphisms of $\cT\Si$ which also preserve the natural filtration of $\cT\Si$, see Section \ref{AdjointTractorBundle}.

\subsection{Geometric Intuition}
In the case where $M$ is the standard CR sphere $\mathbb{S}^{2n+1}$ we can give a clear geometric interpretation of the normal tractor bundle $\cN$ of a CR embedded submanifold, or rather of its orthogonal complement $\cN^{\perp}$. In the conformal case a similar characterisation of the normal tractor bundle may be given via the notion of a central sphere congruence (see \cite{BurstallCalderbank}).

One may explicitly realise the standard tractor bundle of $\mathbb{S}^{2n+1}$ by considering the sphere as the space of isotropic lines in the projectivisation of $\mathbb{C}^{n+1,1}$; if $\ell$ is a complex isotropic line then a standard tractor at the point $\ell\in\mathbb{S}^{2n+1}$ is a constant vector field along $\ell$ in the ambient space $\mathbb{C}^{n+1,1}$. The tractor parallel transport on $\mathbb{S}^{2n+1}$ then comes from the affine structure of $\mathbb{C}^{n+1,1}$ and the standard tractor bundle is flat. 
\begin{figure}[ht]
\begin{center}
\definecolor{dblue}{RGB}{17,46,158}
	\begin{tikzpicture}[xscale=0.5,yscale=0.5]
	\begin{scope}[rotate=-120]
			\draw  (0,0) circle (3.5);
			\draw[dblue,dashed,xscale=1,yscale=1,domain=0:3.141,smooth,variable=\t] plot ({0.6*sin(\t r)-0.8},{3.4*cos(\t r)});
			\draw[dblue,xscale=1,yscale=1,domain=3.141:6.283,smooth,variable=\t] plot ({0.6*sin(\t r)-0.8},{3.4*cos(\t r)});
			\draw[xscale=1,yscale=1,domain=-1.9:3,smooth,variable=\t] plot ({-1.6*sin((\t) r)+0.26},{\t-0.3});
			\draw  [fill] (-1.34,1.29) ellipse (0.05 and 0.05);
		\end{scope}
		\node at (0.53,-4.15) {$\mathbb{P}\mathcal{C} \cong \mathbb{S}^{2n+1}$} ;
		\node at (-2.12,-0.2) {$\Sigma$} ;
		\node at (2.05,1) {$\ell$} ;
		\node at (-2.7,4.5) {$\mathbb{CP}^{n+1}$} ;
		\node at (-0.1,2.3) {$\mathbb{S}_{\ell}$} ;
		\draw  (13,2.5) ellipse (4 and 1);
		\draw (9.05,2.34) -- (13,-3.5) -- (16.96,2.36);
		\node at (8.5,1.81) {$\mathcal{C}$} ;
		\begin{scope}[cm={1 ,-0.45 ,0.15,1,(13.3 cm,0.2 cm)}]
			\draw[dblue]  (-2,4.1) rectangle (2,-3.6);
		\end{scope}
		\node at (15.85,-2.85) {$\mathcal{N}_{\ell}^{\perp}$} ;
		\draw (15.947,2.3) -- (13,-3.5);
		\node at (14,-0.4) {$\ell$} ;
		\node at (8.5,4.5) {$\mathbb{C}^{n+1,1}$} ;
		\draw[-latex, thick] (8.5,-0.5)--(5,-0.5);
		\node at (6.7,0) {$\mathbb{P}$} ;
	\end{tikzpicture}
	\label{NormalBundlePic}
	\caption{The orthogonal complement $\cN^{\perp}$ of the normal tractor bundle when $M=\mathbb{S}^{2n+1}$. The subspace $\cN_{\ell}^{\perp}$ intersects the cone $\mathcal{C}$ of isotropic lines in $\mathbb{C}^{n+1,1}$ in a subcone corresponding to the subsphere $\mathbb{S}_{\ell}$ tangent to $\Si$ at $\ell$.}%
\end{center}
\end{figure}
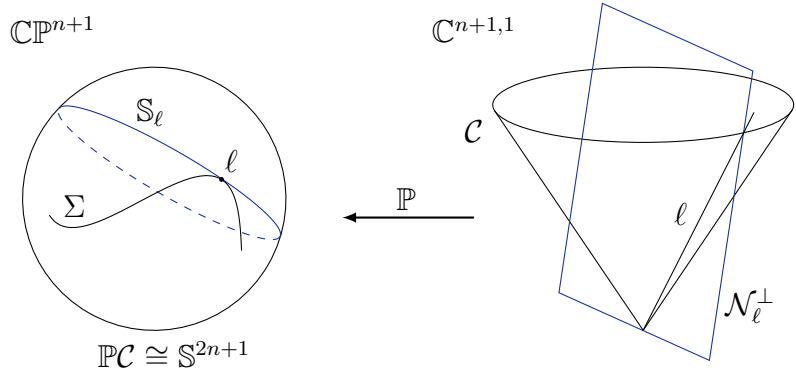
Given a point $x$ in our CR embedded submanifold $\Si^{2m+1}\subset\mathbb{S}^{2n+1}$ there is a unique totally chain CR subsphere $\mathbb{S}_{x}$ of dimension $2m+1$ which osculates $\Si$ to first order at $x$. If we view $\mathbb{S}^{2n+1}$ as the unit sphere in $\bbC^{n+1}$ then $\mathbb{S}_{x}$ is the intersection of $\mathbb{S}^{2n+1}$ with the $(m+1)$-dimensional complex affine subspace of $\bbC^{n+1}$ generated by the tangent space to $\Si$ at $x$. Viewing $\mathbb{S}^{2n+1}$ instead as a projective hyperquadric the sphere $\mathbb{S}_{\ell}$ with $x=\ell$ is the image under the projectivisation map of the intersection of the cone $\mathcal{C}$ of isotropic lines in $\mathbb{C}^{n+1,1}$ with a nondegenerate complex $(m+2)$-dimensional subspace $\cN_{\ell}^{\perp}$. 

In this case the rank $d=n-m$ normal tractor bundle $\cN$ may be viewed as giving a $\mathrm{Gr}(d,\mathbb{C}^{n+1,1})$ valued CR analogue of the Gauss map of an embedded Riemannian submanifold in Euclidean space.

\subsection{Structure of the Article}

We aim to produce a calculus of invariants for CR embeddings which is
both simple and practical, and yields a machinery for constructing
local CR invariants with formulae in terms the psuedohermitian (Tanaka-Webster)
calculus. We thus emphasise heavily the connection between the CR
tractor calculus and the pseudohermitian calculus of a fixed contact
form. Although our final results have a simple interpretation in terms
of tractor calculus, they are often established though explicit
calculation using pseudohermitian calculus. For this reason we have
devoted the first part of the article to giving a detailed exposition
of the Tanaka-Webster calculus associated to a choice of
pseudohermitian contact form (Section \ref{Tanaka-Webster-Calc}) and
an explicit description of the CR tractor calculus in terms of this
pseudohermitian calculus (Section \ref{Tractor-Calc}). Although the
results of Section \ref{Tanaka-Webster-Calc} may largely be found
elsewhere in the literature, proofs are often merely indicated;
collecting these results, and establishing them by proof, provides the
essential reference for verifying the CR invariance of our later
constructions. These results are immediately applied in Section
\ref{Tractor-Calc} where we present the CR tractor calculus, using the
explicit description of the standard tractor bundle and normal
connection given in \cite{GoverGraham-CRInvtPowers}. For the purpose
of invariant theory we introduce some CR analogues of parts of the
conformal tractor calculus not yet developed in the CR case.

In Section \ref{Submanifolds-and-Contact-Forms} we discuss the pseudohermitian geometry of CR embeddings, working in particular with pairs of compatible ambient and submanifold contact forms (see Definition \ref{CompatibleScalesDef}). We also discuss in this section the relationship between the submanifold and ambient CR density bundles. For simplicity we initially treat the minimal codimension strictly pseudoconvex case, generalising to nondegenerate transversal CR embeddings of arbitrary codimension between CR manifolds of any signature in Section \ref{HigherCodimension}.

In Section \ref{Submanifolds-and-Tractors} we develop a manifestly CR invariant approach to studying CR submanifolds using tractor calculus. Again we restrict initially to the minimal codimension strictly pseudoconvex case, generalising in Section \ref{HigherCodimension}. In Section \ref{InvariantsOfSubmanifolds} we apply this calculus to the basic geometric problems of invariant theory for CR embeddings, addressing practical constructions of invariants in Section \ref{PracticalConstructions}. In Section \ref{CRBonnet} we prove a CR analogue of the Bonnet theorem (Theorem \ref{BonnetTheorem}).

\section{Weighted Tanaka-Webster Calculus} \label{Tanaka-Webster-Calc}

\subsection{CR Geometry}

A CR manifold of hypersurface type is a triple $(M^{2n+1},H,J)$ where
$M$ is a real $(2n+1)$-dimensional manifold, $H$ is a corank one
distribution in $TM$, and $J$ is an almost complex structure on $H$
satisfying the integrability condition
\begin{gather*}
[JX,Y]+[X,JY] \in \Gamma(H)\\*
[X,Y]-[JX,JY]+J\left([JX,Y]+[X,JY]\right)=0
\end{gather*}
for any two vector fields $X,Y\in\Gamma(H)$. The almost complex structure
$J$ extends by complex linearity to act on $H\otimes\mathbb{C}$,
and since $J^{2}=-\id$ the eigenvalues of $J$ must be $\pm i$.
It is easy to see that $J$ acts by $i$ on the bundle
\begin{equation*}
T^{1,0}M:=\left\{ X-iJX\::\: X\in H\right\} \subseteq H\otimes\mathbb{C}
\end{equation*}
and by $-i$ on the bundle $T^{0,1}M=\overline{T^{1,0}M}$.
Moreover one has that $T^{1,0}M\cap T^{0,1}M= \{ 0 \}$ and 
\begin{equation*}
H\otimes\mathbb{C}=T^{1,0}M\oplus T^{0,1}M.
\end{equation*}
From the integrability condition imposed on $J$ it follows that $T^{1,0}M$
is \emph{formally integrable}, that is
\begin{equation*}
[T^{1,0}M,T^{1,0}M]\subseteq T^{1,0}M
\end{equation*}
where here we have used the same notation for the bundle $T^{1,0}M$
and its space of sections.

To simplify our discussion we assume that $M$ is orientable. Since
$H$ carries an almost complex structure it must be an orientable
vector bundle, thus the annihilator line bundle $H^{\perp}\subset T^*M$ must
also be orientable (so there exists a global section of $H^{\perp}$
which is nowhere zero). We say that the CR manifold of hypersurface
type $(M^{2n+1},H,J)$ is \emph{nondegenerate} if $H$ is a contact
distribution for $M$, that is, for any global section $\theta$ of
$H^{\perp}$ which is nowhere zero the $(2n+1)$-form $\theta\wedge \d\theta^{n}$
is nowhere zero (this is equivalent to the antisymmetric bilinear form $\d\theta$ being nondegenerate at each point when restricted to elements of $H$). If $H$ is a contact distribution then a global section $\theta$
of $H^{\perp}$ which is nowhere zero is called a contact form. We
assume that the line bundle $H^{\perp}$ has a fixed orientation so that we can
talk about positive and negative elements and sections. We also assume
that $(M^{2n+1},H,J)$ is nondegenerate. 

By the integrability condition on $J$ the bilinear form $\d\theta(\cdot,J\cdot)$ on $H$ is symmetric. The signature $(p,q)$ of this nondegenerate bilinear form on $H$ does not depend on the choice of positive contact form and is called the \emph{signature} of $(M,H,J)$.

Given a choice of contact form $\theta$ for $(M,H,J)$ we refer to
the quadruple $(M,H,J,\theta)$ as a \emph{(nondegenerate) pseudohermitian structure}.
Clearly for any two positive contact forms $\theta$
and $\thetah$ there is a smooth function $\Ups\in C^{\infty}(M)$
such that $\thetah=e^{\Ups}\theta$. One can therefore think of the
CR manifold $(M,H,J)$ as an equivalence class of pseudohermitian
structures much as we may think of a conformal manifold as an equivalence
class of Riemannian structures. In order to make calculations in CR
geometry it is often convenient to fix a choice of contact form $\theta$,
calculate, and then observe how things change if we rescale $\theta$.
We will take this approach in the following, working primarily in
terms of the pseudohermitian calculus associated with the Tanaka-Webster
connection of the chosen contact form $\theta$. In order to make real progress however we will need to make use of the CR invariant tractor calculus \cite{GoverGraham-CRInvtPowers} as a tool to produce CR invariants and invariant operators which can be expressed in terms of the Tanaka-Webster calculus.

\subsection{CR Densities}

On a CR manifold $(M^{2n+1},H,J)$ we denote the
annihilator subbundle of $T^{1,0}M$ by $\La^{0,1}M\subseteq\bbC T^{*}M$
(where by $\bbC T^{*}M$ we mean the complexified cotangent bundle).
Similarly we denote the annihilator subbundle of $T^{0,1}M$ by $\La^{1,0}M\subseteq\bbC T^{*}M$. The bundle $\La^{1,0}M$ has complex rank $n+1$ and hence $\scrK=\La^{n+1}(\La^{1,0}M)$
is a complex line bundle on $M$. The line bundle $\scrK$ is simply the bundle of $(n+1,0)$-forms on $M$, that is
\begin{equation*}
\scrK=\Lambda^{n+1,0}M:=\left\{ \omega\in\Lambda^{n+1}M\,:\, V\hook\omega=0\;\mbox{for all}\; V\in T^{1,0}M\right\},
\end{equation*}
and is known as the \emph{canonical bundle}. 
We assume that $\scrK$ admits an $(n+2)^{\mathrm{th}}$ root $\cE(-1,0)$
and we define $\cE(1,0)$ to be $\cE(-1,0)^{*}$. We then define the
\emph{CR density bundles} $\cE(w,w')$ to be $\cE(1,0)^{w}\otimes\overline{\cE(1,0)}^{w'}$
where $w, w' \in \bbC$ with $w-w'\in\bbZ$.
\begin{rem} 
The assumption that $\scrK$ admits an $(n+2)^{\mathrm{th}}$ root is equivalent to saying that the Chern class $c_1(\scrK)$ is divisible by $n+2$ in $H^2(M,\mathbb{Z})$. Note that if $M$ is a real hypersurface in $\bbC^{n+1}$ then $\scrK$ is trivial and therefore admits such a root.
\end{rem}
Note that the bundles $\cE(w,w')$ and $\cE(w,w')$ are complex conjugates of one another. In particular, each diagonal density bundle $\cE(w,w)$ is fixed under conjugation. We denote by $\cE(w,w)_{\mathbb{R}}$ the real line subbundle of $\cE(w,w)$ consisting of elements fixed by conjugation.

\subsection{Abstract Index Notation}

We freely use abstract index notation for the holomorphic tangent bundle
$T^{1,0}M$, denoting it by $\cE^{\al}$, allowing the use of lower case Greek abstract
indices from the start of the alphabet: $\al$, $\be$, $\ga$, $\de$, $\ep$, $\al'$, $\be'$,
and so on. Similarly we use the abstract index notation $\cE^{\alb}$
for $T^{0,1}M$. We denote the dual bundle of $\cE^{\al}$ by $\cE_{\al}$
and the dual bundle of $\cE^{\alb}$ by $\cE_{\alb}$. Tensor powers
of these bundles are denoted by attaching appropriate indices to the
$\cE$, so, for example, we denote $\cE^{\al}\otimes\cE_{\beta}$ by $\cE^{\al}{_{\be}}$
and $\cE_{\al}\otimes\cE_{\beb}\otimes\cE_{\ga}$ by $\cE_{\al\beb\ga}$.
We attach abstract indices to the elements or sections of our bundles
to show which bundle they belong to, so a section $V$ of $T^{1,0}M$
will be written as $V^{\al}$ and a section $\varpi$ of $(T^{1,0}M)^{*}\otimes T^{0,1}M$
will be denoted by $\varpi_{\al}{^{\beb}}$. The tensor product of
$V^{\al}$ and $\varpi_{\ga}{^{\beb}}$ is written as $V^{\al}\varpi_{\ga}{^{\beb}}$, and repeated indices denote contraction, so $\varpi(V)$ is written as $V^{\al}\varpi_{\al}{^{\beb}}$. Skew-symmetrisation over a collection of indices is indicated by enclosing them in square brackets. Correspondingly we denote the $k^{th}$ exterior power of $\cE_{\al}$ by $\cE_{[\al_1\cdots\al_k]}$. We indicate a tensor product of some (unweighted) complex vector bundle $\mathcal{V}\rightarrow M$ with the density bundle $\cE(w,w')$ by appending $(w,w')$, i.e. $\mathcal{V}(w,w')=\mathcal{V}\otimes\cE(w,w')$. 

We may conjugate elements (or sections) of $\mathcal{E}^{\alpha}$ to get
elements (or sections) of $\mathcal{E}^{\bar{\alpha}}$: we write
\begin{equation*}
V^{\bar{\alpha}}:=\overline{V^{\alpha}}
\end{equation*}
to say that $V^{\bar{\alpha}}$ is the conjugate of $V^{\alpha}$.
This extends in the obvious way to (weighted) tensor product bundles;
note that the complex conjugate bundle of $\cE_{\al}{^{\beb}}(w,w')$
is $\cE_{\alb}{^{\be}}(w',w)$.

We will occasionally use abstract index notation for the tangent bundle, denoting it by $\cE^a$ and allowing lower case Latin abstract indices from the start of the alphabet.

\subsection{The Reeb Vector Field}

Given a choice of contact form $\theta$ for $(M,H,J)$
there is a unique vector field $T\in\fX(M)$ determined by the conditions
that $\theta(T)=1$ and $T\hook\d\theta=0$; this $T$ is called the
\emph{Reeb vector field} of $\theta$. The Reeb vector field gives
us a direct sum decomposition of the tangent bundle
\begin{equation*}
TM=H\oplus\mathbb{R}T
\end{equation*}
and of the complexified tangent bundle
\begin{equation}\label{CTM-Decomp}
\bbC TM=T^{1,0}M\oplus T^{0,1}M\oplus\bbC T,
\end{equation}
where $\mathbb{R}T$ (resp. $\mathbb{C}T$) denotes the real (resp. complex) line bundle spanned by $T$. Dually, given $\theta$ we have
\begin{equation}\label{CT*M-Decomp}
\bbC T^{*}M\cong(T^{1,0}M)^{*}\oplus(T^{0,1}M)^{*}\oplus\bbC\theta.
\end{equation}

\subsection{Densities and Scales}\label{DensitiesAndScales}

\begin{defn}[\cite{Lee-FeffermanMetric}]
Given a contact form $\theta$ for $H$ we say that a section $\zeta$
of $\mathscr{K}$ is \emph{volume normalised} if it satisfies 
\begin{equation}
\theta\wedge(\d\theta)^{n}=i^{n^{2}}n!(-1)^{q}\theta\wedge(T\hook\zeta)\wedge(T\hook\bar{\zeta}).
\end{equation}
\end{defn}
 Given $\zeta$
volume normalised for $\theta$ clearly $\zeta'=e^{i\varphi}\zeta$
is also volume normalised for $\theta$ for any real valued smooth
function $\varphi$ on $M$, so that such a $\zeta$ is determined
only up to phase at each point. Note however that $\zeta\otimes\bar{\zeta}$
does not depend on the choice of volume normalised $\zeta$. Let
us fix a real $(n+2)^{th}$ root $\vsig$ of $\zeta\otimes\bar{\zeta}$
in $\mathcal{E}(-1,-1)$. If $\hat{\theta}=f\theta$ and $\hat{\zeta}$ is volume normalised for $\hat{\theta}$
then $f\vsig$ is an $(n+2)^{th}$ root of $\hat{\zeta}\otimes\bar{\hat{\zeta}}$. The map taking $\theta$ to $\vsig$ and $f\theta$ to $f\vsig$ determines an isomorphism from $H^{\perp}$ to $\cE(-1,-1)_{\mathbb{R}}$. Fixing this isomorphism simply corresponds to fixing an orientation of $\cE(-1,-1)_{\mathbb{R}}$, and we henceforth assume this is fixed. The isomorphism 
\begin{equation}\label{RealDensityIso}
H^{\perp}\cong\cE(-1,-1)_{\mathbb{R}}
\end{equation}
defines a tautological $\cE(1,1)_{\mathbb{R}}$ valued $1$-form:
\begin{defn}
The \emph{CR contact form} is the $\cE(1,1)_{\mathbb{R}}$-valued $1$-form $\btheta$ which is given by $\vsig^{-1}\theta$ where $\theta$ is any pseudohermitian contact form and $\vsig$ is the corresponding positive section of $\cE(-1,-1)_{\mathbb{R}}$.
\end{defn}

\subsection{The Levi Form\label{sub:The-Levi-Form}}

The Levi form of a pseudohermitian contact form $\theta$ is the
Hermitian form $h:T^{1,0}M\otimes T^{0,1}M\rarrow\bbC$ defined by
\begin{equation*}
(U,\overline{V})\mapsto-2i\d\theta(U,\overline{V})=2i\theta([U,\overline{V}])
\end{equation*}
for $U,V\in\Gamma(T^{1,0}M)$. The Levi form of $\theta$ may be thought
of as a section of $\cE_{\al\beb}$, which we write as $h_{\al\beb}$;
there is also an inverse of the Levi form $h^{\al\beb}$ determined
by the condition that $h^{\al\beb}h_{\ga\beb}=\delta^{\al}{_{\ga}}$ (where $\delta^{\al}{_{\ga}}$
is the identity endomorphism of $\cE^{\al}$). Note that if $\theta$
is replaced by $\thetah=e^{\Ups}\theta$ then $\hat{h}_{\al\beb}=e^{\Ups}h_{\al\beb}$
and consequently $\hat{h}^{\al\beb}=e^{-\Ups}h^{\al\beb}$, moreover
it is clear that $\hat{\vsig}=e^{\Ups}\vsig$ (where $\theta=\vsig\btheta$ and $\thetah=\hat{\vsig}\btheta$).
This allows us to define a canonical weighted Levi form:
\begin{defn}
The \emph{CR Levi form} is the $\cE(1,1)$-valued Hermitian form $\bh_{\al\beb}\in \Gamma\left(\cE_{\al\beb}(1,1)\right)$ given by $\vsig^{-1}h_{\al\beb}$ for any pseudohermitian contact form $\theta=\vsig\btheta$.
\end{defn}
In the following we use the CR Levi form $\bh_{\al\beb}$ and its inverse
$\bh^{\al\beb}$ to raise and lower indices. Note that lowering indices
with $\bh_{\al\beb}$ identifies $\cE^{\al}$ with $\cE_{\beb}(1,1)$
so that weights generally change when indices are raised and lowered.

The CR Levi form could also have been defined by the map
\begin{equation*}
(U,\overline{V})\mapsto 2i\btheta([U,\overline{V}]).
\end{equation*}
By complexifying and dualising the isomorphism \eqref{RealDensityIso} we obtain an isomorphism of $\mathcal{E}(1,1)$ with $(\bbC H^{\perp})^{*}=\bbC TM/\bbC H$. This allows us to identify $\bh$, up to a constant factor, with the usual $\bbC TM/\bbC H$-valued Levi form in CR geometry.

\begin{rem}\label{WebsterMetric}
Given a contact form $\theta$ one may also define a pseudo-Riemannian
metric $g_{\theta}$ on the tangent bundle of $M$ by taking the direct
sum of the bilinear form $\d\theta(\cdot,J\cdot)$ on $H$ (which is precisely the real part of the Levi form $h$ of $\theta$) with $\theta\otimes\theta$ on $\bbR T$. This metric is called the \emph{Webster metric}.
\end{rem}

\subsection{Decomposing Tensors}\label{DecomposingTensors}

Using the direct sum decomposition of $\bbC TM$ given by a choice
of contact form $\theta$ a real tangent vector $X$ may be represented
by the triple 
\begin{equation*}
(X^{\al},X^{\alb},X^{0})
\end{equation*}
where $X^{\al}$ is the holomorphic part of $X$, $X^{\alb}$ is the
antiholomorphic part, and $X^{0}=\btheta(X)$. Note that $X^{0}$
is a $(1,1)$ density and $X^{\alb}=\overline{X^{\al}}$. (We follow \cite{GoverGraham-CRInvtPowers}
in using $\btheta$ rather than $\theta$ in defining $X^{0}$, this
simplifies later conformal transformation laws.) Similarly we may represent
a real covector $\omega$ by the triple
\begin{equation*}
(\omega_{\al},\omega_{\alb},\omega_{0})
\end{equation*}
where $\omega_{\al}$ is the restriction of $\omega$ to holomorphic
directions, $\omega_{\alb}=\overline{\omega_{\al}}$ is the restriction of $\omega$ to antiholomorphic
directions, and the $(-1,-1)$ density $\omega_{0}$ is the $\btheta$-component
of $\omega$ (i.e. $\vsig\omega(T)$ where $\theta=\vsig\btheta$).
It is easy to see that the above decompositions extend to arbitrary
tensors or tensor fields. For instance we can represent a real
covariant $2$-tensor $T$ by the $9$-tuple
\begin{equation*}
(T_{\al\be},T_{\al\beb},T_{\alb\be},T_{\alb\beb},T_{\al0},T_{\alb0},T_{0\be},T_{0\beb},T_{00});
\end{equation*}
moreover, by reality it is enough to specify the 5-tuple
\begin{equation*}
(T_{\al\be},T_{\al\beb},T_{\al0},T_{0\be},T_{00})
\end{equation*}
since $T_{\alb\be}=\overline{T_{\al\beb}}$, $T_{\alb0}=\overline{T_{\al0}}$,
and $T_{0\beb}=\overline{T_{0\be}}$.

\subsection{The Tanaka-Webster Connection}\label{sub:The-Tanaka-Webster-Connection}

Since a choice of contact form $\theta$ for $(M,H,J)$ gives rise to a pseudo-Riemannian metric $g_{\theta}$ on $M$ (Remark \ref{WebsterMetric}) one also obtains the Levi-Civita connection $\nabla^{g_{\theta}}$ of $g_{\theta}$. Calculating with this connection is highly inconvenient however, since it does not preserve the direct sum decomposition \eqref{CTM-Decomp} of $\bbC TM$ induced by $\theta$. We instead look for a connection $\nabla$ on $M$ which still satisfies 
\begin{equation*}
\nabla g_{\theta}=0,
\end{equation*}
but whose parallel transport also preserves $H$ and (as a connection on $H$) preserves $J$; such a connection cannot be torsion free, since by the contact condition there exist $X,Y\in\Gamma(H)$ with $[X,Y]\notin\Gamma(H)$ and hence
\begin{equation*}
T^{\nabla}(X,Y)=\nabla_{X}Y-\nabla_{Y}X-[X,Y]
\end{equation*}
cannot be zero since $\nabla_{X}Y-\nabla_{Y}X\in\Gamma(H)$. It turns out that these conditions do not determine a connection on $M$ uniquely, but we can determine $\nabla$ uniquely by imposing
the following additional conditions on the torsion of $\nabla$,
\begin{equation*}
T_{\al\beb}^{\nabla}{^{\ga}}=0,\: T_{\al\beb}^{\nabla}{^{\gab}}=0,\: T_{\al\beb}^{\nabla}{^{0}}=i\bh_{\al\beb},
\end{equation*}
\begin{equation*}
T_{\al\be}^{\nabla}{^{\ga}}=0,\: T_{\al\be}^{\nabla}{^{\gab}}=0,\: T_{\al\be}^{\nabla}{^{0}}=0,
\end{equation*}
\begin{equation*}
T_{\al0}^{\nabla}{^{\ga}}=0,\: T_{\al0}^{\nabla}{^{\gab}}=-A^{\bar{\gamma}}{_{\alpha}},\:\mathrm{and}\,\; T_{\al0}^{\nabla}{^{0}}=0
\end{equation*}
for some $A^{\bar{\gamma}}{_{\alpha}}\in\Gamma(\cE^{\gab}{_{\al}}(-1,-1))$
with $A_{\al\be}$ symmetric (see \cite{Tanaka-book}, Proposition 3.1). The connection $\nabla$ determined
uniquely by these conditions is called the \emph{Tanaka-Webster connection}
of $\theta$ (it was discovered independently by Tanaka and Webster in \cite{Tanaka-pseudoConfConnection,Webster-PseudoHStrs}), and $A_{\al\be}$ is known as
the \emph{pseudohermitian torsion tensor}.

Since the Tanaka-Webster connection preserves $H$ and $g_{\theta}$ it also preserves the $g_{\theta}$-orthogonal complement of $H$, which is spanned by the Reeb vector field $T$. Since $g_{\theta}(T,T)=1$ this implies that $\nabla T=0$. Thus also 
\begin{equation*}
\nabla \theta=0,
\end{equation*}
since $\theta(\:\cdot\:)=g_{\theta}(\:\cdot\:,T)$. By definition the Tanaka-Webster connection $\nabla$ preserves the direct sum decomposition \eqref{CTM-Decomp} of $\bbC TM$ induced by $\theta$. So, by definition $\nabla$ induces a linear connection on $H$ and on $T^{1,0}M$. It therefore makes sense to take the Tanaka-Webster covariant derivative of the Levi form $h$ of $\theta$, and it is easily seen that $\nabla h=0$.

\subsubsection{Interpreting the torsion conditions}
The conditions on the torsion tensor may be alternatively
phrased by saying that for any function $f\in C^{\infty}(M)$ we have
\begin{equation}\label{torsion1}
\nabla_{\alpha}\nabla_{\bar{\beta}}f-\nabla_{\bar{\beta}}\nabla_{\alpha}f=-i\bh_{\alpha\bar{\beta}}\nabla_{0}f,
\end{equation}
\begin{equation}\label{torsion2}
\nabla_{\alpha}\nabla_{\beta}f-\nabla_{\beta}\nabla_{\alpha}f=0,
\end{equation}
and
\begin{equation}\label{PHtorsion}
\nabla_{\alpha}\nabla_{0}f-\nabla_{0}\nabla_{\alpha}f=A^{\bar{\gamma}}{_{\alpha}}\nabla_{\bar{\gamma}}f
\end{equation}
where $A_{\al\be}$ is the symmetric pseudohermitian torsion tensor.
We note that since the Tanaka-Webster connection preserves the direct sum decomposition \eqref{CT*M-Decomp} of $\bbC T^*M$ induced by $\theta$
there is no ambiguity in the notation used in the above displays; for instance
one can equivalently think of $\nabla_{\al}\nabla_{\be}f$ as the
`$\al$-component' of $\nabla$ acting on $\nabla_{\be}f$ or
as the `$\al\be$-component' of $\nabla\nabla f$. 

\subsection{The Tanaka-Webster Connection on Densities}
The Tanaka-Webster connection of a contact form acts on sections of any density bundle since it acts on sections of $\cE(-1,0)^{n+2}=\scrK$. In equation \eqref{PHtorsion} above we are already implicitly using the action of the connection on the density bundle $\cE(-1,-1)$ in the expression $\nabla_{\alpha}\nabla_{0}f$. It does not matter whether or not we think of $\nabla_0 f$ as density valued in such equations because of the following lemma.
\begin{lem}\label{TWonDiagonalDensitiesLemma}
The Tanaka-Webster connection $\nabla$ of $\theta$ on $\cE(-1,-1)$ is simply the flat connection corresponding to the trivialisation induced by the contact form $\theta$, i.e. by the section $\vsig$ satisfying $\theta = \vsig \btheta$. In particular the isomorphism \eqref{RealDensityIso} is parallel for the Tanaka-Webster connection, i.e. it intertwines the actions of $\nabla$ on $H^{\perp}$ and on $\cE(-1,-1)_{\mathbb{R}}$.
\end{lem}
\begin{proof}
Suppose the section $\zeta$ of $\scrK$ is volume normalised for $\theta$. Parallel transporting $\zeta$ along any curve must preserve $\zeta$ up to phase (since the result of parallel transport will still be volume normalised, $\theta$, $\d\theta$, and $T$ being parallel). This implies that $\zeta\otimes\bar{\zeta}$ is parallel, but by definition $\vsig^{n+2}=\zeta\otimes\bar{\zeta}$ so that $\vsig^{n+2}$ and hence $\vsig$ is parallel.
\end{proof}
The lemma also tells us that for the Tanaka-Webster connection $\nabla$ of any contact form $\theta$ we have
\begin{equation}
\nabla\btheta =0 \qquad \mathrm{and} \qquad \nabla \bh =0.
\end{equation}
The advantage of raising and lowering indices with the CR Levi form $\bh_{\al \beb}$ is that these operations commute with \emph{any} Tanaka-Webster covariant derivative.

\subsection{Pseudohermitian Curvature}\label{PseudohermitianCurvature}

By equation \eqref{torsion1} the operator 
\begin{equation*}
\nabla_{\alpha}\nabla_{\bar{\beta}}-\nabla_{\bar{\beta}}\nabla_{\alpha}+i\bh_{\alpha\bar{\beta}}\nabla_{0}
\end{equation*}
annihilates smooth functions on $M$; moreover, this operator preserves $\cE^{\gab}$. By the Leibniz rule the above displayed operator commutes with multiplication by smooth functions when acting on sections of $\cE^{\gab}$. Thus there
is a tensor $R_{\al\beb}{^{\gab}}_{\deb}$ such that 
\begin{equation}\label{PHcurvatureDef}
\nabla_{\alpha}\nabla_{\bar{\beta}}V^{\gab}-\nabla_{\bar{\beta}}\nabla_{\alpha}V^{\gab}+i\bh_{\alpha\bar{\beta}}\nabla_{0}V^{\gab}=-R_{\al\beb}{^{\gab}}_{\deb}V^{\deb}
\end{equation}
for all sections $V^{\gab}$ of $\mathcal{E}^{\gab}$. Equivalently $R_{\al\beb}{^{\gab}}_{\deb}$ is characterised by
\begin{equation}
\nabla_{\alpha}\nabla_{\bar{\beta}}V_{\deb}-\nabla_{\bar{\beta}}\nabla_{\alpha}V_{\deb}+i\bh_{\alpha\bar{\beta}}\nabla_{0}V_{\deb}=R_{\al\beb}{^{\gab}}_{\deb}V_{\gab}
\end{equation}
for all sections $V_{\deb}$ of $\mathcal{E}_{\deb}$. Our conventions agree with those of \cite{GoverGraham-CRInvtPowers,Webster-PseudoHStrs}. We refer to this tensor, or to $R_{\al\beb\ga\deb}=\bh_{\ga\epb}R_{\al\beb}{^{\epb}}_{\deb}$,
as the \emph{pseudohermitian curvature tensor}, and it has the following
properties
\begin{equation}
R_{\al\beb\ga\deb}=R_{\ga\beb\al\deb}=\overline{R_{\be\alb\de\gab}}=R_{\al\deb\ga\beb}=R_{\ga\deb\al\beb}
\end{equation}
which we derive in section \ref{sub:Full-Tanaka-Webster-Curvature}
below. The trace
\begin{equation}
R_{\al\beb}=R_{\al\beb}{^{\gab}}_{\gab}
\end{equation}
of the pseudohermitian curvature tensor is referred to as the \emph{pseudohermitian Ricci
tensor} of $\theta$ and its trace
\begin{equation}
R=\bh^{\al\beb}R_{\al\beb}
\end{equation}
is called the \emph{pseudohermitian scalar curvature} of $\theta$. The pseudohermitian
curvature tensor can be decomposed as
\begin{equation}
R_{\al\beb\ga\deb}=S_{\al\beb\ga\deb}+P_{\al\beb}\bh_{\ga\deb}+P_{\ga\deb}\bh_{\al\beb}+P_{\al\deb}\bh_{\ga\beb}+P_{\ga\beb}\bh_{\al\deb}
\end{equation}
where $S_{\al\beb\ga\deb}$ satisfies
\begin{equation}
S_{\al\beb\ga\deb}=S_{\ga\beb\al\deb}=\overline{S_{\be\alb\de\gab}}=S_{\al\deb\ga\beb}=S_{\ga\deb\al\beb},\; S_{\al\beb}{^{\gab}}_{\gab}=0
\end{equation}
and 
\begin{equation}
P_{\al\beb}=\frac{1}{n+2}\left(R_{\al\beb}-\frac{1}{2(n+1)}R\bh_{\al\beb}\right).
\end{equation}
The tensor $S_{\al\beb\ga\deb}$ is known as the \emph{Chern-Moser
tensor} and is a CR invariant, by which we mean that if $\thetah$
is another contact form for $H$ then $\hat{S}_{\al\beb\ga\deb}=S_{\al\beb\ga\deb}$
(note that we are thinking of $S_{\al\beb\ga\deb}$ as a weighted
tensor field).

\subsection{The Full Tanaka-Webster Curvature\label{sub:Full-Tanaka-Webster-Curvature}}

The full curvature tensor of the Tanaka-Webster connection $\nabla$ of a contact form $\theta$ is defined by
\begin{equation}\label{fullTWcurv}
\nabla_a \nabla_b Y^c - \nabla_b \nabla_a Y^c + T^\nabla_{ab}{^e}\nabla_e Y^c = -R_{ab}{^c}{_d}Y^d 
\end{equation}
for any tangent vector field $Y^c$, where $T^\nabla$ is the torsion of $\nabla$ given in section \ref{sub:The-Tanaka-Webster-Connection}. The pseudohermitian curvature tensor $R_{\al\beb}{^{\gab}}_{\deb}$ is just one component of the full curvature tensor, taken with respect to the direct sum decomposition
\begin{equation}\label{CTMdsd}
\bbC TM=T^{1,0}M\oplus T^{0,1}M\oplus\bbC T
\end{equation}
and its dual. 
\begin{lem}
The full curvature tensor $R_{ab}{^c}{_d}$ of the Tanaka-Webster connection is completely determined by the components $R_{\al\beb}{^{\gab}}_{\deb}$, $R_{\al\be}{^{\gab}}_{\deb}$, and $R_{\al0}{^{\gab}}_{\deb}$.
\end{lem}
\begin{proof}
Note that the tensor $R_{ab}{^c}{_d}$ is real, so that
the component $R_{\alb\be}{^{\ga}}_{\de}$ is simply the complex conjugate
of $R_{\al\beb}{^{\gab}}_{\deb}$ and so on. Also, the symmetry $R_{ab}{^{c}}_d=-R_{ba}{^{c}}_d$ translates into $R_{\al\be}{^{\gab}}_{\deb}=-R_{\be\al}{^{\gab}}_{\deb}$, $R_{\alb0}{^{\ga}}_{\de}=-R_{0\alb}{^{\ga}}_{\de}$, etcetera. Now since the Tanaka-Webster connection preserves the splitting of sections of $\bbC TM$ according to \eqref{CTMdsd} we must have
\begin{equation*}
R_{ab}{^{\ga}}_{\deb}=0,\, R_{ab}{^{\ga}}_{0}=0,\:\mbox{and}\; R_{ab}{^{0}}_{\de}=0.
\end{equation*}
Since $T$ is parallel we also have that
\begin{equation*}
R_{ab}{^{0}}_{0}=0.
\end{equation*}
From this we see that, up to conjugation and swapping the first two indices, the only nonzero components of $R_{ab}{^c}{_d}$ are
\begin{equation*}
R_{\al\beb}{^{\ga}}_{\de} \: , \; R_{\al\be}{^{\ga}}_{\de}\: , \;  R_{\al0}{^{\ga}}_{\de} \: , \; R_{\al\beb}{^{\gab}}_{\deb} \: , \; R_{\al\be}{^{\gab}}_{\deb} \; \; \mathrm{and} \;\; R_{\al0}{^{\gab}}_{\deb}.
\end{equation*}
Our conclusion follows by observing that if we lower indices using the CR Levi form then we have that
\begin{equation*}
R_{ab\ga\deb}=-R_{ab\deb\ga},
\end{equation*}
since $\bh_{\ga\deb}$ is parallel.
\end{proof}
\begin{rem}
From the last display of the above proof we have that $R_{\al\beb\ga\deb}=-R_{\al\beb\deb\ga}$. It immediately follows that $R_{\al\beb\ga\deb}=R_{\beb\al\deb\ga}=\overline{R_{\be\alb\de\gab}}$, establishing one of the claims from Section \ref{PseudohermitianCurvature}.
\end{rem}
Using \eqref{fullTWcurv} the curvature component $R_{\al\be}{^{\gab}}_{\deb}$ may also be characterised by a Ricci-type identity
\begin{equation}
\nabla_{\al}\nabla_{\be}V^{\gab}-\nabla_{\be}\nabla_{\al}V^{\gab}=-R_{\al\be}{^{\gab}}_{\deb}V^{\deb}
\end{equation}
for any section $V^{\gab}$ of $\cE^{\gab}$. Similarly for $R_{\al0}{^{\gab}}_{\deb}$ we have
\begin{equation}
\nabla_{\al}\nabla_{0}V^{\gab}-\nabla_{0}\nabla_{\al}V^{\gab}-A_{\al}^{\ep}\nabla_{\ep}V^{\gab}=-R_{\al0}{^{\gab}}_{\deb}V^{\deb}.
\end{equation}
On a section $V_{\deb}$ of $\cE_{\deb}$ we have
\begin{equation}
\nabla_{\al}\nabla_{\be}V_{\deb}-\nabla_{\be}\nabla_{\al}V_{\deb}=R_{\al\be}{^{\gab}}_{\deb}V_{\gab}
\end{equation}
by duality, and likewise for $R_{\al0}{^{\gab}}_{\deb}$. 

\subsubsection{The Bianchi symmetry}

Recall that for a connection
$\nabla$ without torsion the Bianchi symmetry comes from observing that by torsion freeness one has
\[
(\nabla_{a}\nabla_{b}-\nabla_{b}\nabla_{a})\nabla_{d}f+(\nabla_{b}\nabla_{d}-\nabla_{d}\nabla_{b})\nabla_{a}f+(\nabla_{d}\nabla_{a}-\nabla_{a}\nabla_{d})\nabla_{b}f=0
\]
for any $f\in C^{\infty}(M)$, since the curvature tensor must then satisfy
\[
(R_{ab}{^{c}}_{d}+R_{bd}{^{c}}_{a}+R_{da}{^{c}}_{b})\nabla_{c}f=0.
\] 
This approach also works for a connection with torsion and we use it below to express the consequences of the Bianchi symmetry for the curvature of the Tanaka-Webster connection in terms of the components $R_{\al\beb}{^{\gab}}_{\deb}$, $R_{\al\be}{^{\gab}}_{\deb}$, $R_{\al0}{^{\gab}}_{\deb}$, and the pseudohermitian torsion. Because so many components of $R_{ab}{^c}_d$ already vanish one obtains expressions for $R_{\al\be}{^{\gab}}_{\deb}$ and $R_{\al0}{^{\gab}}_{\deb}$ in terms of the pseudohermitian torsion.
\begin{prop}
The Bianchi symmetry for the Tanaka-Webster curvature tensor $R_{ab}{^{c}}_{d}$ is equivalent to the following identities
\begin{equation}\label{BianchiSymm2}
R_{\al\beb}{^{\gab}}_{\deb}=R_{\al\deb}{^{\gab}}_{\beb}
\end{equation}
\begin{equation}\label{BianchiSymm3}
R_{\al\be\vphantom{\deb}}{^{\gab}}_{\deb}=i\bh_{\al\deb}A_{\be}^{\gab}-i\bh_{\be\deb}A_{\al}^{\gab}
\end{equation}
\begin{equation}\label{BianchiSymm4}
R_{\al0\vphantom{\deb}}{^{\gab}}_{\deb}=\nabla_{\deb}A_{\al}^{\gab}
\end{equation}
\begin{equation}\label{BianchiSymm5}
\nabla_{\al}A_{\de}^{\gab}=\nabla_{\de}A_{\al}^{\gab}.
\end{equation}
\end{prop}
\begin{proof}
By elementary considerations the cyclic sum of $R_{ab}{^{c}}_{d}$ with respect to the lower three indices is determined by the cyclic sums of $R_{\al\be}{^{c}}_{\de}$, $R_{\al\be}{^{c}}_{\deb}$, $R_{\al0}{^{c}}_{\de}$, and $R_{\al0}{^{c}}_{\deb}$ with respect to their lower three indices. In the case of $R_{\al\be}{^{c}}_{\de}$ we may instead cyclically permute the lower indices of $R_{\al\be}{^{\ga}}_{\de}$ since for each permutation the only nonzero part of the tensor is obtained by replacing $c$ with $\ga$. By \eqref{torsion2} $\nabla_{\al}\nabla_{\be}$ is symmetric on smooth functions so that
\begin{equation*}
(\nabla_{\al}\nabla_{\be}-\nabla_{\be}\nabla_{\al})\nabla_{\de}f+(\nabla_{\be}\nabla_{\de}-\nabla_{\de}\nabla_{\be})\nabla_{\al}f+(\nabla_{\de}\nabla_{\al}-\nabla_{\al}\nabla_{\de})\nabla_{\be}f=0
\end{equation*}
for all $f\in C^{\infty}(M)$, and hence
\begin{equation}\label{BianchiSymm1}
R_{\al\be}{^{\ga}}_{\de}+R_{\de\al}{^{\ga}}_{\be}+R_{\be\de}{^{\ga}}_{\al}=0.
\end{equation}
This expression is not listed above because it is a consequence of \eqref{BianchiSymm3}, the latter being equivalent to 
\begin{equation}\label{BianchiSymm3Conseq}
R_{\al\be}{^{\ga}}_{\de}=-i\delta_{\al\vphantom{\be}}^{\ga}A_{\vphantom{\deb}\be\de}+i\delta_{\be}^{\ga}A_{\al\de}
\end{equation}
since $R_{\al\be\ga\deb}=-R_{\al\be\deb\ga}$.

Now let $f$ be a smooth function on $M$. We similarly compute
\begin{equation*}
\left(R_{\al\be\vphantom{\deb}}{^{c}}_{\deb}+R_{\deb\al}{^{c}}_{\be\vphantom{\deb}}+R_{\be\deb}{^{\:c}}_{\al\vphantom{\deb}}\right)\nabla_{c}f.
\end{equation*}
Noting that $R_{\al\be\vphantom{\deb}}{^{c}}_{\deb}\nabla_{c}f=R_{\al\be\vphantom{\deb}}{^{\gab}}_{\deb}\nabla_{\gab}f$, and so on, we get
\begin{align*} 
&
R_{\al\be\vphantom{\deb}}{^{\gab}}_{\deb}\nabla_{\gab}f+R_{\deb\al}{^{\ga}}_{\be\vphantom{\deb}}\nabla_{\ga}f+R_{\be\deb}{^{\ga}}_{\al\vphantom{\deb}}\nabla_{\ga}f\\
 & =  (\nabla_{\al}\nabla_{\be}-\nabla_{\be}\nabla_{\al})\nabla_{\deb}f+(\nabla_{\deb}\nabla_{\al}-\nabla_{\al}\nabla_{\deb}-i\bh_{\al\deb}\nabla_{0})\nabla_{\be}f\\*
 &  \quad +(\nabla_{\be}\nabla_{\deb}-\nabla_{\deb}\nabla_{\be}+i\bh_{\be\deb}\nabla_{0})\nabla_{\al}f\\
 & =  \nabla_{\al}(\nabla_{\be}\nabla_{\deb}f-\nabla_{\deb}\nabla_{\be}f)+\nabla_{\be}(\nabla_{\deb}\nabla_{\al}f-\nabla_{\al}\nabla_{\deb}f)\\*
 & \quad  +\nabla_{\deb}(\nabla_{\al}\nabla_{\be}f-\nabla_{\be}\nabla_{\al}f)-i\bh_{\al\deb}\nabla_{0}\nabla_{\be}f+i\bh_{\be\deb}\nabla_{0}\nabla_{\al}f\\
 & =  \nabla_{\al}(-i\bh_{\be\deb}\nabla_{0}f)+\nabla_{\be}(i\bh_{\al\deb}\nabla_{0}f)\\*
 &  \quad  -i\bh_{\al\deb}\nabla_{0}\nabla_{\be}f+i\bh_{\be\deb}\nabla_{0}\nabla_{\al}f\\
 & =  i\bh_{\al\deb}(\nabla_{\be}\nabla_{0}f-\nabla_{0}\nabla_{\be}f)+i\bh_{\be\deb}(\nabla_{0}\nabla_{\al}f-\nabla_{\al}\nabla_{0}f)\\
 & =  i\bh_{\al\deb}A_{\be}^{\epb}\nabla_{\epb}f-i\bh_{\be\deb}A_{\al}^{\epb}\nabla_{\epb}f
\end{align*}
Since $f$ was arbitrary the above display holds at any point for
all functions $f$ with $\nabla_{\ga}f=0$ (or with $\nabla_{\gab}f=0$) at that point, and thus
we conclude that
\begin{equation*}
R_{\al\be\vphantom{\deb}}{^{\gab}}_{\deb}=i\bh_{\al\deb}A_{\be}^{\gab}-i\bh_{\be\deb}A_{\al}^{\gab}
\end{equation*}
and
\begin{equation*}
R_{\deb\al}{^{\ga}}_{\be\vphantom{\deb}}+R_{\be\deb}{^{\ga}}_{\al\vphantom{\deb}}=0.
\end{equation*}
By conjugation the last display is equivalent to \eqref{BianchiSymm2}, noting that  $R_{\be\deb}{^{\ga}}_{\al\vphantom{\deb}}=-R_{\deb\be}{^{\ga}}_{\al\vphantom{\deb}}$.

Similarly computing 
\begin{equation*}
R_{\al0\vphantom{\deb}}{^{c}}_{\deb}\nabla_{c}f+R_{0\deb}{^{c}}_{\al\vphantom{\deb}}\nabla_{c}f+ R_{\deb\al}{^{c}}_{0}\nabla_{c}f
\end{equation*}
we get (noting $R_{\deb\al}{^{c}}_{0}=0$)
\begin{equation*}
R_{\al0\vphantom{\deb}}{^{\gab}}_{\deb}\nabla_{\gab}f+R_{0\deb}{^{\ga}}_{\al\vphantom{\deb}}\nabla_{\ga}f=-(\nabla_{\al\vphantom{\deb}}A_{\deb}^{\ga})\nabla_{\ga}f+(\nabla_{\deb}A_{\al}^{\gab})\nabla_{\gab}f
\end{equation*}
so that
\begin{equation*}
R_{\al0\vphantom{\deb}}{^{\gab}}_{\deb}=\nabla_{\deb}A_{\al}^{\gab}.
\end{equation*}

Finally, computing the cyclic sum for $R_{\al0}{^{c}}_{\de}\nabla_c f$ we obtain 
\begin{equation*}
R_{\al0}{^{\ga}}_{\delta}\nabla_{\ga}f+R_{0\de}{^{\ga}}_{\al}\nabla_{\ga}f=-(\nabla_{\al}A_{\de}^{\gab})\nabla_{\gab}f+(\nabla_{\de}A_{\al}^{\gab})\nabla_{\gab}f
\end{equation*}
so that
\begin{equation*}
\nabla_{\al}A_{\de}^{\gab}=\nabla_{\de}A_{\al}^{\gab}
\end{equation*}
and
\begin{equation}\label{BianchiSymmCons}
R_{\al0}{^{\ga}}_{\delta}=R_{\de0}{^{\ga}}_{\al}.
\end{equation}
The identity \eqref{BianchiSymmCons} follows already from \eqref{BianchiSymm4} since by lowering indices we have $R_{\al0\ga\deb}=\nabla_{\deb}A_{\al\ga}$ and using that $R_{\al0\deb\ga}=-R_{\al0\ga\deb}$ we get that
\begin{equation}
R_{\al0}{^{\ga}}_{\delta}=-\nabla^{\ga}A_{\al\de}.
\end{equation}
\end{proof}

The expressions \eqref{BianchiSymm4} and \eqref{BianchiSymm3Conseq} agree with those given in section 1.4.2 of \cite{DragomirTomassini}, after adjusting \eqref{BianchiSymm3Conseq} by factor of two to account for their slightly different conventions (see (1.84) in \cite{DragomirTomassini}).

Note that \eqref{BianchiSymm2} implies that the pseudohermitian curvature
tensor satisfies
\begin{equation*}
R_{\al\beb\ga\deb}=R_{\ga\beb\al\deb}
\end{equation*}
(as was previously claimed) from which we also deduce that
\begin{equation*}
R_{\al\beb\ga\deb}=R_{\al\deb\ga\beb}=R_{\ga\deb\al\beb}.
\end{equation*}

\subsection{Curvature of the Density Bundles}

Although the Tanaka-Webster connection $\nabla$ of a contact form $\theta$ is flat on the diagonal density bundles $\cE(w,w)$ it is not flat on density bundles in general. The curvature of the density bundles was calculated in \cite[Prop. 2.2]{GoverGraham-CRInvtPowers}. We give this proposition with an alternate proof:
\begin{prop}\label{TWDensityCurvature}
Let $\theta$ be a pseudohermitian contact form and $\nabla$ its Tanaka-Webster connection. On a section $f$ of $\cE(w,w')$ we have
\begin{align}
\nabla_{\al}\nabla_{\be}f-\nabla_{\be}\nabla_{\al}f  &=0, \\
\nabla_{\al}\nabla_{\beb}f-\nabla_{\beb}\nabla_{\al}f+i\boldsymbol{h}_{\al\beb}\nabla_{0}f  &=\frac{w-w'}{n+2}R_{\al\beb}f, \\
\nabla_{\al}\nabla_{0}f-\nabla_{0}\nabla_{\al}f-A_{\al\ga}\nabla^{\ga}f  &=\frac{w-w'}{n+2}(\nabla^{\ga}A_{\ga\al})f.
\end{align}
\end{prop}
\begin{proof}
We first consider sections $f=\zeta$ of $\cE(-n-2,0)=\scrK$. The map 
\begin{equation*}
\zeta\mapsto (T\hook \zeta)|_{T^{1,0}M}
\end{equation*}
induces an isomorphism between the complex line bundles $\scrK$ and $\Lambda^n \left(T^{1,0}M\right)^*$. This isomorphism between $\scrK$ and $\cE_{[\al_1\cdots \al_n]}$ intertwines the action of the Tanaka-Webster connection $\nabla$ since the Reeb vector field $T$ is parallel and $\nabla$ preserves $T^{1,0}M$, so the two line bundles have the same curvature. The curvature of the top exterior power $\cE_{[\al_1\cdots \al_n]}$ of $\cE_{\al}$ is simply obtained by tracing the curvature of $\cE_{\al}$. 

Now let $f$ be a section of $\cE(-n-2,0)=\scrK$. By tracing the conjugate of \eqref{BianchiSymm3} and using the appropriate Ricci-type identity for $R_{\alb\beb}{^{\ga}}_{\de}$ we therefore obtain
\begin{equation*}
\nabla_{\alb}\nabla_{\beb}f-\nabla_{\beb}\nabla_{\alb}f  = 0
\end{equation*}
since  $R_{\alb\beb}{^{\ga}}_{\ga}=0$. Similarly, using that $R_{\al\be\ga\deb}=-R_{\al\be\deb\ga}$, we obtain from \eqref{BianchiSymm3} that $R_{\al\be}{^{\ga}}_{\ga}=0$ and hence
\begin{equation*}
\nabla_{\al}\nabla_{\be}f-\nabla_{\be}\nabla_{\al}f  =0.
\end{equation*}
Using that $R_{\al\beb}{^{\ga}}_{\ga}=R_{\al\beb}{_{\gab}}^{\gab}=-R_{\al\beb}{^{\gab}}_{\gab}=-R_{\al\beb}$
and the appropriate Ricci-type identity for $R_{\al\beb}{^{\ga}}_{\de}$ we get
\begin{equation*}
\nabla_{\al}\nabla_{\beb}f-\nabla_{\beb}\nabla_{\al}f+i\boldsymbol{h}_{\al\beb}\nabla_{0}f  =-R_{\al\beb}f.
\end{equation*}
Finally, by tracing the conjugate of \eqref{BianchiSymm4} we obtain 
\begin{equation*}
\nabla_{\alb}\nabla_{0}f-\nabla_{0}\nabla_{\alb}f-A_{\alb}^{\ga}\nabla_{\ga}f  =(\nabla_{\ga}A_{\alb}^{\ga})f
\end{equation*}
and using $R_{\al0\ga\deb}=-R_{\al0\deb\ga}$ we obtain from \eqref{BianchiSymm4} that $R_{\al0}{^{\ga}}_{\de}=-\nabla^{\ga}A_{\al\de}$ so that
\begin{equation*}
\nabla_{\al}\nabla_{0}f-\nabla_{0}\nabla_{\al}f-A_{\al}^{\gab}\nabla_{\gab}f  =-(\nabla^{\ga}A_{\al\ga})f.
\end{equation*}
This establishes the proposition for $(w,w')$ equal to $(-n-2,0)$, and for $(w,w')$ equal to $(0,-n-2)$ by conjugating.

Considering the action of the curvature operator(s) on $(n+2)^{th}$ powers of sections of $\cE(-1,0)$ and $\cE(0,-1)$ we obtain the result for $(w,w')$ equal to $(-1,0)$ or $(0,-1)$. Taking powers and tensor products then gives the full proposition.
\end{proof}

\subsection{Changing Contact Form}\label{ChangingContactForm}

Here we establish how the various pseudohermitian objects we have introduced transform
under conformal rescaling of the contact form. 
The first thing to consider is the Reeb vector field:
\begin{lem}\label{ReebTransform}
Under the transformation $\thetah=e^{\Ups}\theta$ of pseudohermitian contact forms, $\Ups\in C^{\infty}(M)$, the Reeb vector field transforms according to
\begin{equation}\label{hatT}
\hat{T}=e^{-\Ups}\left[T+\left((\d \Ups) |_{H}\circ J\right)^{\sharp}\right]
\end{equation}
where $\sharp$ denotes the usual isomorphism $H^*\rightarrow H$ induced by the 
bundle metric $\d\theta(\cdot,J\cdot)|_H$ on the contact distribution.
\end{lem}
\begin{proof}
Defining $\hat{T}$ by \eqref{hatT} one has $\thetah(\hat{T})=1$ and
\begin{equation*}
\hat{T}\hook\d\thetah =\hat{T}\hook\d(e^{\Ups}\theta)= \d\Ups(\hat{T})\thetah-\d\Ups+e^{\Ups}\hat{T}\hook\d\theta.
\end{equation*}
We observe that
\begin{equation*}
\d\thetah(\hat{T},JY) = -\d\Ups(JY)+\d\theta(e^{\Ups}\hat{T},JY) =0
\end{equation*}
for any $Y\in\Gamma(H)$, using that $\thetah(JY)=0$ and 
\begin{equation*}
\d\theta(e^{\Ups}\hat{T},JY) = \d\theta(\left((\d \Ups) |_{H}\circ J\right)^{\sharp},JY)=\d\Ups(JY).
\end{equation*}
Now since $\d\thetah(\hat{T},\hat{T})=0$ we have $\hat{T}\hook\d\thetah=0$.
\end{proof}
If $\eta$ is a 1-form whose restrictions to $\cE^{\al}$ and $\cE^{\alb}$ are $\eta_{\al}$ and $\eta_{\alb}$ respectively then $(\eta |_H)^{\sharp}$ is a contact vector field whose antiholomorphic component is $h^{\be\alb}\eta_{\be}$ and whose holomorphic component is $h^{\al\beb}\eta_{\beb}$. It is easy to see that the restriction of $(\d \Ups) |_{H}\circ J$ to $\cE^{\al}$ is $i\nabla_{\al} \Ups$ and the restriction to $\cE^{\alb}$ is $-i\nabla_{\alb} \Ups$. These observations imply:
\begin{lem}\label{0ComponentTransform}
If a 1-form $\omega$ has components $(\:\!\omega_{\al},\omega_{\alb},\omega_0)$ with respect to some contact form $\theta$, then the components of $\omega$ with respect to $e^{\Ups}\theta$ are
\begin{equation*}
(\:\!\omega_{\al},\:\omega_{\alb},\:\omega_0+i\Ups^{\alb}\omega_{\alb}-i\Ups^{\al}\omega_{\al})
\end{equation*}
where $\Ups^{\alb}=\bh^{\be\alb}\nabla_{\be}\Ups$ and $\Ups^{\al}=\bh^{\al\beb}\nabla_{\beb}\Ups$.
\end{lem}

\subsubsection{The Tanaka-Webster transformation laws}

We need to see how the Tanaka-Webster connection transforms under rescaling of the contact form.
\begin{prop}\label{TWtransform}
Under the transformation $\thetah=e^{\Ups}\theta$ of pseudohermitian contact forms, $\Ups\in C^{\infty}(M)$, the Tanaka-Webster connection on sections $\tau_{\beta}$ of $\mathcal{E}_{\beta}$ transforms according to
\begin{align}
\label{TWhhTransform} \nablah_{\al}\tau_{\be} & =\nabla_{\al}\tau_{\be}-\Upsilon_{\be}\tau_{\al}-\Upsilon_{\al}\tau_{\be}\\ 
\label{TWahTransform}\nablah_{\alb}\tau_{\be} & =\nabla_{\alb}\tau_{\be}+\boldsymbol{h}_{\be\alb}\Upsilon^{\ga}\tau_{\ga}\\
\label{TW0hTransform}\nablah_{0}\tau_{\be} & =\nabla_{0}\tau_{\be}+i\Upsilon^{\gab}\nabla_{\gab}\tau_{\be}-i\Upsilon^{\ga}\nabla_{\ga}\tau_{\be}-i(\Upsilon^{\ga}{_{\be}}-\Upsilon^{\ga}\Upsilon_{\be})\tau_{\ga}
\end{align}
where $\Ups_{\al}=\nabla_{\al}\Ups$, $\Ups_{\alb}=\nabla_{\alb}\Ups$,
$\Ups_{\al\beb}=\nabla_{\beb}\nabla_{\al}\Ups$, and indices
are raised using $\bh^{\al\beb}$.
\end{prop}
\begin{rem*}
Note that on the left hand side of the last equation in
the above display the `0-component' is taken with respect to the
splitting of the cotangent bundle induced by $\thetah$, whereas on
the right hand side it is taken with respect to the splitting of the
cotangent bundle induced by $\theta$ (recall Sections \ref{DensitiesAndScales} and \ref{DecomposingTensors}). In other words the operator
$\nablah_{0}$ is taken to be $\hat{\vsig}\nablah_{\hat{T}}$ where
$\thetah=\hat{\vsig}\btheta$, whereas $\nabla_{0}=\vsig\nabla_{T}$ where
$\theta=\vsig\btheta$. Note that $\btheta(\vsig T)=\theta(T)=1$ and similarly $\btheta(\hat{\vsig}\hat{T})=1$, so $\vsig T$ and $\hat{\vsig}\hat{T}$ are natural weighted versions of the Reeb vector fields of $\theta$ and $\thetah$ respectively. 
\end{rem*}
\begin{proof}[Proof of Proposition \ref{TWtransform}] We define the connection $\nablah$ on $\cE_{\be}$ by the formulae above, and extend $\nablah$ to a connection on $TM$ in the obvious way. Precisely, we define $\nablah$ to act on $\cE_{\beb}$ by the conjugates of the above formulae, so that, e.g., $\nablah_{\al}\tau_{\beb}$ is the conjugate of $\nablah_{\alb}\tau_{\be}$ with $\tau_{\be}=\overline{\tau_{\beb}}$. This gives a connection on $\bbC H^*$ which preserves the real subbundle $H^*$ and preserves $J$. Thus by requiring $\thetah$ to be parallel for $\nablah$ we obtain a connection on $T^*M$ and hence on $TM$. To show that this is the Tanaka-Webster connection of $\thetah$ it remains only to show that $\nablah g_{\thetah}=0$ and to verify the torsion conditions.

To show that $\nablah g_{\thetah}=0$ it is sufficient to show that $\nablah$ preserves the Levi form $\hat{h}$ is of $\thetah$. This is a computation using the formulae in the proposition: By the Leibniz rule we have
\begin{equation*}
\nablah_{\al}(\tau_{\be}\zeta_{\gab})=\nabla_{\al}(\tau_{\be}\zeta_{\gab})-\Ups_{\al}(\tau_{\be}\zeta_{\gab})-\Ups_{\be}(\tau_{\al}\zeta_{\gab})+\bh_{\al\gab}\Ups^{\deb}(\tau_{\be}\zeta_{\deb})
\end{equation*}
for a simple section of $\cE_{\be\gab}$. By $\bbC$-linearity we obtain
\begin{equation*}
\nablah_{\al}\hat{h}_{\be\gab}=\nabla_{\al}\hat{h}_{\be\gab}-\Ups_{\al}\hat{h}_{\be\gab}-\Ups_{\be}\hat{h}_{\al\gab}+\bh_{\al\gab}\Ups^{\deb}\hat{h}_{\be\deb};
\end{equation*}
the terms on the right hand side of the above display cancel
in pairs since $\hat{h}_{\be\gab}=e^{\Ups}h_{\be\gab}$. By conjugate
symmetry we also get $\nablah_{\alb}\hat{h}_{\be\gab}=0$. Similarly one computes that
\begin{align*}
\nablah_{0}\hat{h}_{\al\beb} & =\nabla_{0}\hat{h}_{\al\beb}+i\Ups^{\gab}\nabla_{\gab}\hat{h}_{\al\beb}-i\Ups^{\ga}\nabla_{\ga}\hat{h}_{\al\beb}\\
 & \quad -i\left(\Ups^{\ga}{_{\al}}-\Ups^{\ga}\Ups_{\al}\right)\hat{h}_{\ga\beb}+i\left(\Ups^{\gab}{_{\beb}}-\Ups^{\gab}\Ups_{\beb}\right)\hat{h}_{\al\gab}.
\end{align*}
The terms on the right hand side of the above display
all cancel since $\nabla_{0}\hat{h}_{\al\beb}=\Ups_{0}\hat{h}_{\al\beb}$
and $\Upsilon_{\al\beb}-\Ups_{\beb\al}=i\bh_{\al\beb}\Ups_{0}$,
where $\Ups_{0}=\nabla_{0}\Ups$ and $\Ups_{\beb\al}=\nabla_{\al}\nabla_{\beb}\Ups$.

Substituting $\nablah_{\be}f$ ($=\nabla_{\be}f$) for $\tau_{\be}$
in equation \eqref{TWhhTransform} we see that
\begin{equation*}
\nablah_{\alpha}\nablah_{\beta}f-\nablah_{\beta}\nablah_{\alpha}f=0,
\end{equation*}
since $-\Upsilon_{\beta}\tau_{\alpha}-\Upsilon_{\alpha}\tau_{\beta}$
is symmetric. Similarly from \eqref{TWahTransform} we obtain
\begin{equation*}
\nablah_{\alpha}\nablah_{\beb}f-\nablah_{\beb}\nablah_{\alpha}f=-i\bh_{\alpha\bar{\beta}}\left(\nabla_{0}f+i\Ups^{\gab}\nabla_{\gab}f-i\Upsilon^{\ga}\nabla_{\ga}f \right),
\end{equation*}
and from Lemma \ref{0ComponentTransform} we
have that
\begin{equation} \label{0DerivativeTransform}
\nablah_{0}f=\nabla_{0}f+i\Ups^{\gab}\nabla_{\gab}f-i\Upsilon^{\ga}\nabla_{\ga}f.
\end{equation}

From \eqref{TW0hTransform} and \eqref{0DerivativeTransform} one has that
\begin{align}
\nonumber \nablah_{\alpha}\nablah_{0}f-\nablah_{0}\nablah_{\alpha}f= & \,\nablah_{\alpha}\left(\nabla_{0}f-i\Upsilon^{\ga}\nabla_{\ga}f+i\Ups^{\gab}\nabla_{\gab}f\right)\\*
\label{PrePHTorsionTransform} & \,-\left(\nabla_{0}\nabla_{\al}f+i\Ups^{\gab}\nabla_{\gab}\nabla_{\al}f-i\Ups^{\ga}\nabla_{\ga}\nabla_{\al}f\right.\\*
\nonumber & \,\,\qquad\qquad\qquad\quad\left.-i(\Ups^{\ga}{_{\al}}-\Ups^{\ga}\Upsilon_{\al})\nabla_{\ga}f\right).
\end{align}
One can easily compute directly that
\begin{align*}
\nablah_{\al}\left(\nabla_{0}f-i\Upsilon^{\ga}\nabla_{\ga}f+i\Ups^{\gab}\nabla_{\gab}f\right)& =\nabla_{\al}\left(\nabla_{0}f-i\Upsilon^{\ga}\nabla_{\ga}f+i\Ups^{\gab}\nabla_{\gab}f\right)\\*
& \quad +\Ups_{\al}\left(\nabla_{0}f-i\Upsilon^{\ga}\nabla_{\ga}f+i\Ups^{\gab}\nabla_{\gab}f\right)
\end{align*}
(cf.\ the proof of Proposition \ref{TWonDensitiesProp}). Substituting this into \eqref{PrePHTorsionTransform}, expanding using the Leibniz rule and simplifying one obtains 
\begin{equation*}
\nablah_{\alpha}\nablah_{0}f-\nablah_{0}\nablah_{\alpha}f=\left( A_{\al\ga}+i\Ups_{\al\ga}-i\Ups_{\al}\Ups_{\ga} \right)\nablah^{\ga}f
\end{equation*}
where $\Ups_{\al\ga}=\nabla_{\al}\nabla_{\ga}\Ups$ is symmetric.
\end{proof}
Note that in the course of the proof we have established the transformation law
\begin{equation}\label{PHTorsionTransform}
\hat{A}_{\al\be}=A_{\al\be}+i\Ups_{\al\be}-i\Ups_{\al}\Ups_{\be}
\end{equation}
for the pseudohermitian torsion.

\subsubsection{The transformation law for the pseudohermitian curvature tensor}

From the transformation laws for the Tanaka-Webster connection one can directly compute that
\begin{equation}
\hat{R}_{\al\beb\ga\deb}=R_{\al\beb\ga\deb}+\La_{\al\beb}\bh_{\ga\deb}+\La_{\ga\deb}\bh_{\al\beb}+\La_{\al\deb}\bh_{\ga\beb}+\La_{\ga\beb}\bh_{\al\deb}
\end{equation}
where
\begin{equation}\label{SchoutenTransformLambda}
\La_{\al\beb}=-\frac{1}{2}(\Ups_{\al\beb}+\Ups_{\beb\al})-\frac{1}{2}\Ups^{\ga}\Ups_{\ga}\bh_{\al\beb}.
\end{equation}
In particular this tells us that $\hat{S}_{\al\beb\ga\deb}=S_{\al\beb\ga\deb}$
and
\begin{equation}\label{SchoutenTransform}
\hat{P}_{\al\beb}=P_{\al\beb}+\La_{\al\beb}.
\end{equation}

\subsubsection{The transformation laws for the Tanaka-Webster connection on densities}\label{TWonDensities}

We also need to know how the Tanaka-Webster connection transforms
when acting on densities. These transformation laws follow from the
above since it suffices to compare the action of $\nabla$ and $\nablah$
on sections of the canonical bundle $\scrK$.
\begin{prop}\cite[Prop. 2.3]{GoverGraham-CRInvtPowers}\label{TWonDensitiesProp}
Under the transformation $\thetah=e^{\Ups}\theta$ of pseudohermitian contact forms, $\Ups\in C^{\infty}(M)$, the Tanaka-Webster connection acting on sections $f$ of $\mathcal{E}(w,w')$ transforms according to
\begin{align*}
\nablah_{\al}f & =\nabla_{\al}f+w\Ups_{\al}f\\
\nablah_{\alb}f & =\nabla_{\alb}f+w'\Ups_{\alb}f\\
\nablah_{0}f & =\nabla_{0}f+i\Ups^{\gab}\nabla_{\gab}f-i\Ups^{\ga}\nabla_{\ga}f\\*
 & \quad+\tfrac{1}{n+2}\left[(w+w')\Ups_{0}+iw\Ups^{\ga}{_{\ga}}-iw'\Ups^{\gab}{_{\gab}}+i(w'-w)\Ups^{\ga}\Ups_{\ga}\right]f.
\end{align*}
\end{prop}
\begin{proof}
Since $\nabla$ preserves $T$ and $\Gamma(T^{1,0}M)$ the map 
\begin{equation*}
I_{\theta}:\zeta\mapsto (T\hook\zeta)|_{T^{1,0}M}
\end{equation*}
taking sections of $\scrK$ to sections of $\cE_{[\al_1\cdots\al_n]}$ commutes with $\nabla_X$ for all $X\in \mathfrak{X}(M)$. On the other hand $I_{\thetah}:\zeta\mapsto (\hat{T}\hook\zeta)|_{T^{1,0}M}$ intertwines the action of the connection $\nablah$. Now $I_{\thetah}=e^{-\Ups}\circ I_{\theta}$ since if $Y=\hat{T}-e^{-\Ups}T$ then $(Y\hook\zeta)|_{T^{1,0}M}=0$ for any $(n+1,0)$-form $\zeta$; to see this note that $Y$ is contact (by Lemma \ref{ReebTransform}) and the antiholomorphic part of $Y$ hooks into $\zeta$ to give zero, but also $\zeta|_{T^{1,0}M}=0$ since the rank of $T^{1,0}M$ is $n$. Thus we have
\begin{align*}
I_{\theta}(\nablah_X \zeta) & = e^{\Ups}I_{\thetah}(\nablah_X \zeta) = e^{\Ups}\nablah_X I_{\thetah}(\zeta) = e^{\Ups}\nablah_X [e^{-\Ups}I_{\theta}(\zeta)]\\ 
& = \nablah_X I_{\theta}(\zeta) - \d \Ups (X) I_{\theta}(\zeta)
\end{align*}
for any $X\in\fX(M)$, $\zeta \in \Gamma (\scrK)$. So the action of $\nablah$ on $\scrK$ is conjugate under $I_{\theta}$ to the action of $\nablah -\d \Ups$ on $\cE_{[\al_1\cdots\al_n]}$. One now easily translates using $I_{\theta}$ the transformation laws for the Tanaka-Webster connection on $\cE_{[\al_1\cdots\al_n]}$ (obtained from Proposition \ref{TWtransform} by taking traces) to the transformation laws for the Tanaka-Webster connection on the canonical bundle $\scrK$. The transformation laws for $\cE(w,w')$ then follow from those for $\scrK=\cE(-n-2,0)$ in the obvious way.
\end{proof}

\section{CR Tractor Calculus} \label{Tractor-Calc}

It is well known that nondegenerate (hypersurface type) CR geometries admit an equivalent description as parabolic Cartan geometries. The Cartan geometric description of CR manifolds was introduced by Cartan \cite{Cartan-RealHypInC2} in the case of $3$-dimensional CR manifolds, and by Tanaka \cite{Tanaka-pseudoConfConnection, Tanaka65, Tanaka76} and Chern and Moser \cite{Chern-Moser} in the general case. To a signature $(p,q)$ CR manifold $(M,H,J)$ there
is associated a canonical Cartan geometry $(\mathcal{G},\omega)$
of type $(\SU(p+1,q+1),P)$ where the subgroup $P$ of $\SU(p+1,q+1)$
is the stabiliser of a complex null line in $\bbC^{p+1,q+1}$. Moreover,
any local CR diffeomorphism of $(M,H,J)$ with another CR manifold
$(M',H',J')$ lifts to a local equivalence of the canonical Cartan geometries $(\mathcal{G},\omega)$ and $(\mathcal{G}',\omega')$. In the model case of the CR sphere $\mathcal{G}$ is simply the group $G=\SU(n+1,1)$ as a principal bundle over $\mathbb{S}^{2n+1}=G/P$ and $\omega$ is the left Maurer-Cartan form of $G$. Strictly speaking, if we do not wish to impose any global assumptions in the general case we need to quotient $\SU(p+1,q+1)$ and $P$ by their common finite cyclic center, but for the purpose of local calculus we can ignore this.

Given any representation $\mathbb{V}$ of $\SU(p+1,q+1)$ there is associated to the CR Cartan bundle $\mathcal{G}$ a vector bundle $\mathcal{V}=\mathcal{G}\times_P \mathbb{V}$ over $M$. The CR Cartan connection $\omega$ induces on $\mathcal{V}$ a linear connection $\nabla^{\mathcal{V}}$. Such bundles $\mathcal{V}$ are known as \emph{tractor bundles}, and the connection $\nabla^{\mathcal{V}}$ is the (canonical) \emph{tractor connection} \cite{CapGover-TracCalc}. If $\mathbb{T}$ is the standard representation $\bbC^{p+1,q+1}$ of $\SU(p+1,q+1)$ then $\mathcal{T}=\mathcal{G}\times_P \mathbb{T}$ is known as the \emph{standard tractor bundle}. Since $\mathbb{T}$ is a faithful representation of $P$ the CR Cartan bundle $\mathcal{G}$ may be recovered from $\cT$ as an adapted frame bundle. The Cartan connection $\omega$ is easily recovered from $\nabla^{\cT}$. Elementary representation theory tells us that all other irreducible representations of $\SU(p+1,q+1)$ are subbundles of tensor representations constructed from $\mathbb{T}$ (and $\mathbb{T^*}$) given by imposing certain tensor symmetries, so knowing the standard tractor bundle $\mathcal{T}$ and its tractor connection one can easily explicitly obtain all tractor bundles and connections. 

The tractor bundles and their tractor connections, along with certain invariant differential (splitting) operators from irreducible tensor bundles on the CR manifold into tractor bundles, form the basis of a calculus of local invariants and invariant operators for CR manifolds known as the \emph{CR tractor calculus}.

\subsection{The Standard Tractor Bundle}

There are various ways to construct the CR Cartan bundle and hence the standard tractor bundle. However for our purposes it is much better to use the direct construction of the CR standard tractor bundle and connection found in \cite{GoverGraham-CRInvtPowers}. This allows a very concrete description of the standard tractor bundle and connection in terms of the weighted tensor bundles and Tanaka-Webster calculus of Section \ref{Tanaka-Webster-Calc}. 

Since the subgroup $P$ of $\SU(p+1,q+1)$ stabilises a null complex line in $\mathbb{T}$ it also stabilises the orthogonal complement of this null line and so there is a filtration 
\begin{equation}\label{StandardTractorFiltration}
\cT^1\subset\cT^0\subset \cT^{-1}= \cT
\end{equation}
of $\cT$ by subbundles where $\cT^1$ has complex rank $1$ and $\cT^0$ has complex rank $n+1$ (and corank $1$). The starting point for the explicit construction of $\cT$ in \cite{GoverGraham-CRInvtPowers} is the observation that
\begin{equation}
\mathcal{T}^{1}=\cE(-1,0), \;\; \mathcal{T}^{0}/\mathcal{T}^1=\cE^{\al}(-1,0), \;\; \mathrm{and} \;\; \mathcal{T}^{-1}/\mathcal{T}^0=\cE(0,1).
\end{equation}

Let us introduce abstract index notation $\cE^{A}$ for $\cT$, allowing the use of capitalised Latin indices from the start of the alphabet. The dual of $\cT=\cE^{A}$ is denoted by $\cE_{A}$ and the conjugate by $\cE^{\Ab}$. Following \cite{GoverGraham-CRInvtPowers} we present the standard cotractor bundle $\cE_{A}$ rather than $\cE^{A}$ (it makes little difference since there will be a parallel Hermitian metric around). The bundle $\mathcal{E}_{A}$ comes with a naturally defined filtration, dual to \eqref{StandardTractorFiltration}. Given a choice of contact form $\theta$ for $(M,H,J)$ we identify the standard cotractor bundle $\mathcal{E}_{A}$ with
\begin{equation*}
[\cE_A]_{\theta}=\mathcal{E}(1,0)\oplus\mathcal{E}_{\alpha}(1,0)\oplus\mathcal{E}(0,-1);
\end{equation*}
we write $v_{A}\overset{\theta}{=}(\sigma,\tau_{\alpha},\rho)$, 
\begin{equation*}
v_{A}\overset{\theta}{=}\left(\begin{array}{c}
\sigma\\
\tau_{\alpha}\\
\rho
\end{array}\right), 
\quad \mathrm{or} \quad
[v_{A}]_{\theta}=\left(\begin{array}{c}
\sigma\\
\tau_{\alpha}\\
\rho
\end{array}\right)
\end{equation*}
if an element or section of $\mathcal{E}_{A}$ is represented by $(\sigma,\tau_{\alpha},\rho)$ with respect to this identification; the identifications given by two contact forms $\theta$ and $\hat{\theta}=e^{\Upsilon}\theta$
are related by the transformation law 
\begin{equation}\label{TractorTransformation}
[\cE_A]_{\theta} \ni
\left(\begin{array}{c}
\sigma\\
\tau_{\alpha}\\
\rho
\end{array}\right)
\sim
\left(
\begin{array}{ccc}
1 & 0 & 0\\
\Ups_{\al} & \delta^{\be}_{\al} & 0 \\
-\frac{1}{2}(\Ups^{\be}\Ups_{\be}+i\Ups_0) & - \Ups^{\be} & 1 
\end{array}
\right)\left(\begin{array}{c}
\sigma\\
\tau_{\be}\\
\rho
\end{array}\right)
\in [\cE_A]_{\thetah}
\end{equation}
where $\Upsilon_{\alpha}=\nabla_{\alpha}\Upsilon$ and $\Upsilon_{0}=\nabla_{0}\Upsilon$. This transformation law comes from the action of the nilpotent part $P_+$ of $P$ on $\mathbb{T}^*$ (see \cite{CapSlovak-ParabolicGeometries} for the general theory) so that $\sim$ is indeed an equivalence relation on the 
disjoint union of the spaces $[\cE_A]_{\theta}$. We can thus take the standard cotractor bundle $\mathcal{E}_{A}$ to be the quotient of the disjoint union of the $[\cE_A]_{\theta}$ over all pseudohermitian contact forms $\theta$ by the equivalence relation \eqref{TractorTransformation}.

\subsection{Splitting Tractors}\label{SplittingTractors}
From \eqref{TractorTransformation} it is clear that there is an invariant inclusion of $\cE(0,-1)$ into $\mathcal{E}_{A}$ given with respect to any contact form $\theta$ by the map
\begin{equation*}
\rho\mapsto\left(\begin{array}{c}
0\\
0\\
\rho
\end{array}\right).
\end{equation*}
Correspondingly there is an invariant section $Z_{A}$ of $\mathcal{E}_{A}(0,1)$
such that the above displayed map is given by $\rho\mapsto\rho Z_{A}$. The weight $(0,1)$ \emph{canonical tractor} $Z_A$ can be written as 
\begin{equation*}
Z_A=\left(\begin{array}{c}
0\\
0\\
1
\end{array}\right)
\end{equation*}
with respect to any choice of contact form $\theta$. 

Given a fixed choice of $\theta$, we also get the corresponding \emph{splitting tractors}
\begin{equation*}
W_A^{\be}\overset{\theta}{=}\left(\begin{array}{c}
0\\
\delta_{\al}^{\be}\\
0
\end{array}\right)
\quad \mathrm{and} \quad
Y_A\overset{\theta}{=}\left(\begin{array}{c}
1\\
0\\
0
\end{array}\right)
\end{equation*}
which both have weight $(-1,0)$. A standard cotractor $v_{A}\overset{\theta}{=}(\sigma,\tau_{\alpha},\rho)$ may instead be written as $v_{A}= \sigma Y_A + W^{\beta}_A\tau_{\beta}+\rho Z_A$ where we understand that $Y_A$ and $W_A^{\beta}$ are defined in terms of the splitting induced by $\theta$. If $\thetah=e^{\Ups}\theta$ then by \eqref{TractorTransformation} we have
\begin{align}
\label{Wtransformation}\hat{W}_A^{\be}&=W_A^{\be}+\Ups^{\be}Z_A, \\
\label{Ytransformation}\hat{Y}_A&=Y_A-\Ups_{\be}W_A^{\be}-\frac{1}{2}(\Ups_{\be}\Ups^{\be}-i\Ups_{0})Z_A.
\end{align}

\subsection{The Tractor Metric} 
Since the group $P$ preserves the inner product on $\mathbb{T}=\bbC^{p+1,q+1}$ the standard tractor bundle $\cT=\cG\times_P \mathbb{T}$ carries a natural signature $(p+1,q+1)$ Hermitian bundle metric. We denote this bundle metric by $h_{A\bar{B}}$, and its inverse by $h^{A\bar{B}}$. Explicitly the tractor metric $h_{A\bar{B}}$ is given with respect to any contact form $\theta$ by
\begin{equation}
h_{A\bar{B}}=Z_A Y_{\bar{B}} + \bh_{\al\beb}W^{\al}_A W^{\beb}_{\bar{B}} + Y_A Z_{\bar{B}}
\end{equation}
where $Z_{\bar{B}}$, $W^{\beb}_{\bar{B}}$, and $Y_{\bar{B}}$ are the respective conjugates of $Z_{B}$, $W^{\be}_{B}$, and $Y_{B}$. One can easily check directly using \eqref{Wtransformation} and \eqref{Ytransformation} that the above expression does not depend on the choice of $\theta$. Dually, the inverse tractor metric is given by
\begin{equation}\label{TractorMetric}
h^{A\bar{B}}v'_A \overline{v_B} = \sigma\overline{\rho'}+\boldsymbol{h}^{\alpha\bar{\beta}}\tau_{\alpha}\overline{\tau'_{\beta}}+\rho\overline{\sigma'}
\end{equation}
for any two sections $v_{A}\overset{\theta}{=}(\sigma,\tau_{\alpha},\rho)$ and $v'_{A}\overset{\theta}{=}(\sigma',\tau'_{\alpha},\rho')$ of $\cE_A$.

We use the tractor metric to identify $\cE^A$ with $\cE_{\bar{A}}$, the latter of which can be described explicitly as the disjoint union of the spaces
\begin{equation}
[\cE_{\bar{A}}]_{\theta}=\mathcal{E}(0,1)\oplus\mathcal{E}_{\alb}(0,1)\oplus\mathcal{E}(-1,0)
\end{equation}
(over all pseudohermitian contact forms $\theta$) modulo the equivalence relation obtained by conjugating \eqref{TractorTransformation}. Identifying $\mathcal{E}_{\alb}(0,1)$ with $\mathcal{E}^{\al}(-1,0)$ via the CR Levi form we write a standard tractor as  $v^{A}\overset{\theta}{=}(\sigma,\tau^{\alpha},\rho)$ or 
\begin{equation*}
v^{A}\overset{\theta}{=}\left(\begin{array}{c}
\sigma\\
\\
\tau^{\alpha}\\
\\
\rho
\end{array}\right)
\in
\begin{array}{c}
\mathcal{E}(0,1)\\
\oplus\\
\mathcal{E}^{\al}(-1,0)\\
\oplus\\
\mathcal{E}(-1,0)
\end{array}.
\end{equation*}
We may also raise and lower indices on the splitting tractors in order to write $v^{A}= \sigma Y^A+\tau^{\beta}W^A_{\beta}+\rho Z^A$ with respect to $\theta$.

With these conventions the pairing between $v^{A}\overset{\theta}{=}(\sigma,\tau^{\alpha},\rho)$ and $v'_{A}\overset{\theta}{=}(\sigma',\tau'_{\alpha},\rho')$ is given by
\begin{equation}
v^A v'_A = \si\rho'+\tau^{\al}\tau'_{\alpha}+\rho\si'.
\end{equation}
The various contractions of the splitting tractors (for a given $\theta$) are described by the following table
\begin{equation}\label{SplittingTractorContractions}
\begin{array}{c|ccc}
 & Y_{A} & W_{A\beb} & Z_{A}\\
\hline 
Y^{A\vphantom{\hat{A}}} & 0 & 0 & 1\\
W_{\al}^{A} & 0 & \bh_{\al\beb} & 0\\
Z^{A\vphantom{\hat{A}}} & 1 & 0 & 0
\end{array}
\end{equation}
which reflects the form of the tractor metric $h_{A\bar{B}}$.

\subsection{The Tractor Connection}\label{The-Tractor-Connection}

In order to define the canonical (normal) tractor connection we need two further
curvature objects. These are
\begin{equation}\label{Ttensor}
T_{\al}=\frac{1}{n+2}(\nabla_{\al}P_{\be}{^{\be}}-i\nabla^{\be}A_{\al\be})
\end{equation}
and 
\begin{equation}\label{Sscalar}
S=-\frac{1}{n}(\nabla^{\al}T_{\al}+\nabla^{\alb}T_{\alb}+P_{\al\beb}P^{\al\beb}-A_{\al\be}A^{\al\be}).
\end{equation}
These expressions appear in \cite{GoverGraham-CRInvtPowers} and can be determined from the following formulae for
the tractor connection by the normalisation condition on the tractor
curvature (which amounts to certain traces of curvature tensors vanishing,
see \cite{CapGover-FeffermanSpace}). Of course the $S$ and
$T_{\al}$ terms are also needed to make sure that the formulae for
the tractor connection given below transform correctly so as to give
a well defined connection on $\cE_{A}$.

On any section $v_{A}\overset{\theta}{=}(\sigma,\tau_{\alpha},\rho)$
of $\cE_{A}$ the \emph{standard tractor connection} $\nabla^{\cT}$ (or simply $\nabla$) is defined by
the following formulae
\begin{equation}\label{hhTractorConnection}
\nabla_{\beta}v_{A}\overset{\theta}{=}\left(\begin{array}{c}
\nabla_{\beta}\sigma-\tau_{\beta}\\
\nabla_{\beta}\tau_{\alpha}+iA_{\alpha\beta}\sigma\\
\nabla_{\beta}\rho-P_{\beta}{^{\alpha}}\tau_{\alpha}+T_{\beta}\sigma
\end{array}\right),
\end{equation}
\begin{equation}\label{ahTractorConnection}
\nabla_{\bar{\beta}}v_{A}\overset{\theta}{=}\left(\begin{array}{c}
\nabla_{\bar{\beta}}\sigma\\
\nabla_{\bar{\beta}}\tau_{\alpha}+\boldsymbol{h}_{\alpha\bar{\beta}}\rho+P_{\alpha\bar{\beta}}\sigma\\
\nabla_{\bar{\beta}}\rho-iA_{\bar{\beta}}{^{\alpha}}\tau_{\alpha}-T_{\bar{\beta}}\sigma
\end{array}\right),
\end{equation}
and
\begin{equation}\label{0hTractorConnection}
\nabla_{0}v_{A}\overset{\theta}{=}\left(\begin{array}{c}
\nabla_{0}\sigma+\frac{i}{n+2}P\sigma-i\rho\\
\nabla_{0}\tau_{\alpha}+\frac{i}{n+2}P\tau_{\alpha}-iP_{\alpha}{^{\beta}}\tau_{\beta}+2iT_{\alpha}\sigma\\
\nabla_{0}\rho+\frac{i}{n+2}P\rho+2iT^{\alpha}\tau_{\alpha}+iS\sigma
\end{array}\right)
\end{equation}
where $P=P_{\be}{^{\beb}}$.
Using \eqref{PHTorsionTransform}, \eqref{SchoutenTransform}, and Proposition \ref{TWtransform} combined with Proposition \ref{TWonDensitiesProp} one may check directly that the above formulae transform appropriately under rescaling of the contact form $\theta$ (i.e. are compatible with \eqref{TractorTransformation}, and with Lemma \ref{ReebTransform} in the case of \eqref{0hTractorConnection}) so that they give a well defined connection on $\cE_A$.

Coupling the tractor connection with the Tanaka-Webster connection of some contact form $\theta$, the tractor connection is given on the splitting operators by (cf.\ \cite{CapGover-FeffermanSpace}, and also \cite[Proposition 3.1]{EbenfeltHuangZaitsev-Rigidity})
\begin{align}
\label{hNablaY}
\nabla_{\be}Y_A\; &= \;iA_{\al\be}W_A^{\al} + T_{\be}Z_A \\*
\nabla_{\be}W_A^{\al}\; &= \; -\de^{\al}_{\be}Y_A  -P_{\be}{^{\al}}Z_A \\*
\label{Soldering1}
\nabla_{\be}Z_A \; &= \; 0 \\
\nabla_{\beb}Y_A\; &= \;P_{\al\beb}W_A^{\al} - T_{\beb}Z_A \\*
\nabla_{\beb}W_A^{\al}\; &= \; iA^{\al}_{\beb}Z_A \\*
\label{Soldering2}
\nabla_{\beb}Z_A \; &= \; \bh_{\al\beb}W^{\al}_A,
\end{align}
and
\begin{align}
\nabla_{0}Y_A\; &= \;\tfrac{i}{n+2}P Y_A + 2iT_{\al}W_A^{\al} + iS Z_A \\*
\nabla_{0}W_A^{\al}\; &= \; -iP_{\be}{^{\al}}W_A^{\be} + \tfrac{i}{n+2}PW_A^{\al} + 2iT^{\al}Z_A \\*
\label{Soldering3}
\nabla_{0}Z_A \; &= \; -iY_A + \tfrac{i}{n+2}P Z_A.
\end{align}
Using either set of formulae for the tractor connection one can easily show by direct calculation that $\nabla$ preserves $h_{A\bar{B}}$.

\subsubsection{Weyl connections on the tangent bundle}\label{WeylConnections}

The expression \eqref{0hTractorConnection} for $\nabla_{0} v_A$ may be simplified if one absorbs the terms involving $P_{\al}{^{\be}}$ and its trace $P$ into the definition
of the connection on the tangent bundle we are using. This amounts
to working with, in the terminology of \cite{CapSlovak-ParabolicGeometries},
the \emph{Weyl connection} determined by $\theta$ rather than the
Tanaka-Webster connection of $\theta$. The Weyl connection $\nabla^{W}$
determined by $\theta$ agrees with the Tanaka-Webster connection when
differentiating in contact directions, but when differentiating in
the Reeb direction one has
\begin{equation}\label{0WeylConnection}
\nabla_{0}^{W}\tau_{\al}=\nabla_{0}\tau_{\alpha}-iP_{\alpha}{^{\beta}}\tau_{\beta}
\end{equation}
for a section $\tau_{\al}$ of $\cE_{\al}$ (the action on $\cE_{\alb}$ is given by conjugating \eqref{0WeylConnection} and $\nabla_{0}^{W}T=0$). Using the isomorphism $I_{\theta}$ of $\cE(-n-2,0)$ with $\cE_{[\al_1 \cdots \al_n]}$ from the proof of Proposition \ref{TWonDensitiesProp} one obtains from \eqref{0WeylConnection} that
\begin{equation}\label{0WeylConnectionOnDensities}
\nabla_{0}^{W}\si=\nabla_{0}\si+\frac{i}{n+2}P\si
\end{equation}
for a section $\sigma$ of $\cE(1,0)$. Using the Weyl connection rather than the Tanaka-Webster connection in the expression \eqref{0hTractorConnection} for $\nabla_0 v_A$ one has the simpler expression
\begin{equation}\label{0hTractorConnectionWeyl}
\nabla_{0}v_{A}\overset{\theta}{=}\left(\begin{array}{c}
\nabla_{0}^{W}\si-i\rho\\
\nabla_{0}^{W}\tau_{\al}+2iT_{\alpha}\sigma\\
\nabla_{0}^{W}\rho+2iT^{\alpha}\tau_{\alpha}+iS\sigma
\end{array}\right).
\end{equation}

\subsection{The Adjoint Tractor Bundle}\label{AdjointTractorBundle}

Another important bundle on CR manifolds is the \emph{adjoint tractor bundle} $\cA=\cG\times_P \fg$. Since $\fg$ is the space of skew-Hermitian endomorphisms of $\bbC^{p+1,q+1}$ with respect to the signature $(p+1,q+1)$ inner product we may identify $\cA$ with the bundle of $h_{A\bar{B}}$-skew-Hermitian endomorphisms of the standard tractor bundle. Thus we think of $\cA$ as the subbundle of $\cE_A{^B}=\cE_A\otimes \cE^B$ whose sections $t_A{^B}$ satisfy
\begin{equation*}
t_{A\bar{B}}=-\overline{t_{B\bar{A}}}.
\end{equation*}
Since the standard tractor connection $\nabla^{\cT}$ is Hermitian, it induces a connection on $\cA\subset \mathrm{End}(\cT)$ and this is the usual (normal) tractor connection on $\cA$.

The adjoint tractor bundle carries a natural filtration 
\begin{equation}
\cA^2 \subset \cA^1 \subset \cA^0 \subset \cA^{-1} \subset \cA^{-2} = \cA
\end{equation}
corresponding to a $P$-invariant filtration of $\mathfrak{su}(p+1,q+1)$. In particular, $\cA^0=\cG\times_P \mathfrak{p}$ where $\mathfrak{p}=\mathrm{Lie}(P)$. Sections $t$ of $\cA^{-1}$ are those skew-Hermitian endomorphisms which satisfy $t_A{^B}Z^A Z_B = 0$, and sections $\cA^{0}$ are those which additionally satisfy $t_A{^B}Z^A W^{\be}_B = 0$. In any parabolic geometry the subbundle $\cA^1=\cG\times \mathfrak{p}_+$, where $\mathfrak{p}_+$ is the nilpotent part of $\mathfrak{p}$ (in this case $\mathfrak{p}_+$ is a Heisenberg algebra), is canonically isomorphic to $T^*M$. Here the isomorphism is given explicitly by the map 
\begin{equation}\label{CotangentIntoAdjoint}
(\upsilon_{\al},\upsilon_{\alb},\upsilon_0)\mapsto \upsilon_{\al} W_A^{\al} Z_{\bar{B}} - 
\upsilon_{\beb} Z_A W_{\bar{B}}^{\beb} - i \upsilon_0 Z_A Z_{\bar{B}}
\end{equation}
with respect to any contact form $\theta$. Dual to \eqref{CotangentIntoAdjoint} there is a bundle projection from $\cA^*\cong\cA$ to $TM$. Explicitly, the resulting isomorphism of $\cA/\cA^0$ with $TM$ is given with respect to $\theta$ by
\begin{equation}\label{AdjointOntoTangent}
X^{\al}W_{\al}^A Y^{\bar{B}} - X^{\beb}Y^A W^{\bar{B}}_{\beb} + i X^0 Y^A Y^{\bar{B}} + \cA^0 \mapsto (X^{\al},X^{\alb}, X^0).
\end{equation}

\subsection{The Tractor Curvature}\label{TractorCurvature}

The curvature of the standard tractor connection agrees with the usual ($\fg$-valued) curvature of the canonical Cartan connection when the latter is thought of as an adjoint tractor ($\cA=\cG\times_P\fg$) valued two form. To normalise our conventions with the index notation we define the curvature of the tractor bundle by
\begin{equation}
\nabla_a \nabla_b v_C - \nabla_b \nabla_a v_C + T^{\nabla}_{ab}{^e}\nabla_e v_C = -\ka_{ab C}{^D}v_D
\end{equation}
for all sections $v_C$ of $\cE_C$, where $\nabla$ denotes the tractor connection coupled with any connection on the tangent bundle and $T^{\nabla}$ is the torsion of that connection on the tangent bundle. Since we allow for the use of connections with torsion on the tangent bundle in the above, we may compute $\ka_{ab C}{^D}v_D$ explicitly in its decomposition with respect to any contact form $\theta$ using the weighted Tanaka-Webster calculus developed in Section \ref{Tanaka-Webster-Calc}. The resulting expressions were given in \cite{GoverGraham-CRInvtPowers} (cf.\ \cite{Chern-Moser}); we have
\begin{align}
\label{hhTractorCurvature}
\ka_{\al\be C}{^D}& =\; 0,  \\
\label{haTractorCurvature}
\ka_{\al\beb C}{^D}& \overset{\theta}{=} 
\left(
\begin{array}{ccc}
0 & 0 & 0\\
iV_{\al\beb\ga} & S_{\al\beb\ga}{^{\de}} & 0 \\
U_{\al\beb} & -iV_{\al\beb}{^{\de}}  & 0 
\end{array}\right), \\
\label{h0TractorCurvature}
\ka_{\al0 C}{^D} &\overset{\theta}{=} 
\left(
\begin{array}{ccc}
0 & 0 & 0\\
Q_{\al\ga} & V_{\al}{^{\de}}_{\ga} & 0 \\
Y_{\al} & -iU_{\al}{^{\de}}  & 0 
\end{array}\right), 
\end{align}
where
\begin{align}
\label{Vtensor}
V_{\al\beb\ga} &=\nabla_{\beb}A_{\al\ga}+i\nabla_{\ga} P_{\al\beb}-iT_{\ga}\bh_{\al\beb}-2iT_{\al}\bh_{\ga\beb},\\
U_{\al\beb} &= \nabla_{\al}T_{\beb}+\nabla_{\beb}T_{\al}+P_{\al}{^{\ga}}P_{\ga\beb} -A_{\al\ga}A^{\ga}_{\beb}+ S\bh_{\al\beb}, \\
Q_{\al\be} &= i\nabla_0 A_{\al\be} - 2i\nabla_{\be}T_{\al}+2P_{\al}{^{\ga}}A_{\ga\be},\\
Y_{\al}&=\nabla_0 T_{\al}-i\nabla_{\al}S + 2iP_{\al}^{\ga}T_{\ga}-3A_{\al}^{\gab}T_{\gab}.
\end{align}
Here the matrices appearing in \eqref{haTractorCurvature} and \eqref{h0TractorCurvature} are arranged so that the action of $\ka_{\al\beb C}{^D}$ and $\ka_{\al0 C}{^D}$ on $v_D$ is given (with respect to $\theta$) by the action of the respective matrices on the column vector representing $v_D$.
The remaining components of the tractor/Cartan curvature $\ka$ are determined by the obvious symmetries. We also have from \cite{GoverGraham-CRInvtPowers} (again cf.\ \cite{Chern-Moser}) that
\begin{align*}
V_{\al\beb\ga}=V_{\ga\beb\al},& \quad V_{\al}{^{\al}}_{\ga}=0,\\*
U_{\al\beb}=\overline{U_{\be\alb}},& \quad U_{\al}{^{\al}}=0,\\*
\mathrm{and} \;\; Q_{\al\be}&=Q_{\be\al}.
\end{align*}

The tractor connection on the CR sphere $\mathbb{S}^{2n+1}$ is flat, and for a general strictly pseudoconvex CR manifold the tractor curvature is precisely the obstruction to being locally CR equivalent to the sphere (see Theorem \ref{LocallyFlatTheorem}).

\subsection{Invariant Tractor Operators}\label{InvariantOps}

The tractor calculus can be used to give a uniform construction of curved analogues for almost all CR invariant differential operators between irreducible bundles on the model CR sphere. The key idea behind this is to apply Eastwood's `curved translation principle' \cite{EastwoodRice} to the tractor covariant exterior derivative $\d^{\nabla}$ using certain invariant \emph{differential splitting operators} constructed via the `BGG machinery' of \cite{CSS,CalderbankDiemer}. Important exceptional cases are dealt with in \cite{GoverGraham-CRInvtPowers} where the authors construct CR invariant powers of the sublaplacian on curved CR manifolds using the tractor D-operator (which extends one of the BGG splitting operators to a family of operators parametrised by weight). Such invariant differential splitting operators are also very useful in the problem of constructing invariants of CR structures, since they allow the jets of the structure (or rather of some invariant curvature tensor) to be packaged in a tractorial object which can be further differentiated invariantly. In the following we present the most basic and important of these (families of) invariant operators.

\subsubsection{The tractor D-operator(s)}\label{SectionTractorD-op}
Let $\mathcal{E}^{\Phi}$ denote any tractor bundle and
let $\mathcal{E}^{\Phi}(w,w')$ denote the tensor product of $\mathcal{E}^{\Phi}$
with $\mathcal{E}(w,w')$. 
\begin{defn}
The \emph{tractor D-operator} of  \cite{GoverGraham-CRInvtPowers}
\begin{equation*}
\mathsf{D}_{A}:\mathcal{E}^{\Phi}(w,w')\rightarrow\mathcal{E}_{A}\otimes\mathcal{E}^{\Phi}(w-1,w')
\end{equation*}
is defined by
\begin{equation}
\mathsf{D}_{A}f^{\Phi}\overset{\theta}{=}\left(\begin{array}{c}
w(n+w+w')f^{\Phi}\\
(n+w+w')\nabla_{\alpha}f^{\Phi}\\
-\left(\nabla^{\beta}\nabla_{\beta}f^{\Phi}+iw\nabla_{0}f^{\Phi}+w(1+\frac{w'-w}{n+2})Pf^{\Phi}\right)
\end{array}\right)
\end{equation}
where $\nabla$ denotes the tractor connection coupled to
the Tanaka-Webster connection of $\theta$.
\end{defn}
One may easily check directly that $\mathsf{D}_{A}$, as defined,
does not depend on the choice of $\theta$. This operator is an analogue 
of the Thomas tractor D-operator in conformal geometry \cite{BaileyEastwoodGover-Thomas'sStrBundle}. Observe that $\mathsf{D}_A$ is a splitting operator (has a bundle map as left inverse) except at weights where $w(n+w+w')=0$.

Related to the tractor $D$-operator is the $\theta$ dependent operator $\tilde{\mathsf{D}}_A$ given by $wf^{\Phi}Y_A+(\nabla_{\al}f^{\Phi})W^{\al}_A$ on a section $f^{\Phi}$ of $\mathcal{E}^{\Phi}(w,w')$. The operator $\mathbb{D}_{AB}$ defined by
\begin{equation}
\mathbb{D}_{AB}f^{\Phi}=2Z_{[A}\tilde{\mathsf{D}}_{B]}f^{\Phi}
\end{equation}
does not depend on the choice of $\theta$. The operator $\mathbb{D}_{AB}$ has a partner $\mathbb{D}_{A\bar{B}}$ defined by 
\begin{equation}
\mathbb{D}_{A\bar{B}}f^{\Phi}= Z_{\bar{B}} \tilde{\mathsf{D}}_A f^{\Phi} - Z_A \tilde{\mathsf{D}}_{\bar{B}} f^{\Phi} - Z_A Z_{\bar{B}}\left[i\nabla_0f^{\Phi}+\tfrac{w'-w}{n+2}Pf^{\Phi}\right]
\end{equation}
where $\tilde{\mathsf{D}}_{\bar{B}} f^{\Phi}=\overline{\tilde{\mathsf{D}}_{B} \overline{f^{\Phi}}}$. Note that if $f^{\Phi}$ has weight $(0,0)$ then
\begin{equation}\label{haDoubleD-OnWeight0}
\mathbb{D}_{A\bar{B}}f^{\Phi}= Z_{\bar{B}} W^{\al}_A\nabla_{\al} f^{\Phi} - Z_A W^{\beb}_{\bar{B}}\nabla_{\beb} f^{\Phi} - iZ_A Z_{\bar{B}}\nabla_0f^{\Phi},
\end{equation}
cf.\ \eqref{CotangentIntoAdjoint}, so that $\mathbb{D}_{A\bar{B}}$ takes sections of $\mathcal{E}^{\Phi}$ to sections of $\cA\otimes \mathcal{E}^{\Phi}$. The pair of invariant operators  $\mathbb{D}_{AB}$ and $\mathbb{D}_{A\bar{B}}$ acting on sections of $\mathcal{E}^{\Phi}(w,w')$ are called \emph{double-D-operators} \cite{Gover-Aspects}.

\begin{rem}\label{FundamentalAndDouble-D}
The less obvious operator $\mathbb{D}_{A\bar{B}}$ comes from coupling the fundamental derivative of \cite{CapGover-TracCalc} on densities with the tractor connection to give an operator on weighted tractors. The conformal double-D-operator on the Fefferman space, which comes from similarly twisting the fundamental derivative on conformal densities with the conformal tractor connection, can be seen to induce the pair of operators $\mathbb{D}_{A\bar{B}}$, $\mathbb{D}_{AB}$ (and the conjugate operator $\mathbb{D}_{\bar{A}\bar{B}}$) on the underlying CR manifold \cite[Theorem 3.7]{CapGover-FeffermanSpace}.
\end{rem}

\subsubsection{Middle operators}\label{MiddleOps}

One can also create CR invariant differential splitting operators which take weighted sections of tensor bundles of $\cE^{\al}$ to weighted tractors. These are analogues of operators in conformal geometry used by Eastwood for `curved translation' (see, e.g., \cite{EastwoodConformal}). We only construct the particular operators from this family that we will need in the following.

\begin{defn}
The \emph{middle operator} acting on sections of $\cE_{\al}(w,w')$ is the operator $\sfM_A^{\al}:\cE_{\al}(w,w')\rightarrow\cE_{A}(w-1,w')$ given with respect to a choice of contact form $\theta$ by
\begin{equation}\label{hMiddleOperator}
\sfM_A^{\al} \tau_{\al} = (n+w')W_A^{\al}\tau_{\al} - Z_A\nabla^{\al}\tau_{\al}.
\end{equation}
\end{defn}
To see that the operator defined by \eqref{hMiddleOperator} is invariant one simply observes (by combining Proposition \ref{TWtransform} with Proposition \ref{TWonDensitiesProp}) that if $\thetah=e^{\Ups}\theta$ then
\begin{equation}
\hat{\nabla}^{\al}\tau_{\al}=\nabla^{\al}\tau_{\al}+(n+w')\Ups^{\al}\tau_{\al}
\end{equation}
for $\tau_{\al}$ of weight $(w,w')$, and on the other hand from \eqref{Wtransformation} we have $\hat{W}^{\al}_A = W^{\al}_A +\Ups^{\al}Z_A$.
\begin{rem} The operator $\sfM_A^{\al}$ defined by \eqref{hMiddleOperator} is a differential splitting operator, except when $w'=-n$, in which case $\tau_{\al}\mapsto\nablah^{\al}\tau_{\al}$ is an invariant operator and $\sfM_A^{\al}$ simply becomes (minus) the composition of this operator with the bundle map $\rho\mapsto\rho Z_A$ (for $\rho$ of appropriate weight).
\end{rem}

In the same manner, by observing that when $\thetah=e^{\Ups}\theta$ we have
\begin{equation}
\hat{\nabla}^{\al}\tau_{\al\beb}=\nabla^{\al}\tau_{\al\beb}+(n+w'-1)\Ups^{\al}\tau_{\al\beb} -\Ups_{\beb}\tau_{\al}{^{\al}}
\end{equation}
for $\tau_{\al\beb}$ of weight $(w,w')$, we see that there is an invariant operator on trace free sections of $\cE_{\al\beb}(w,w')$ given by
\begin{equation}\label{hMiddleOperatorTwo}
\sfM_A^{\al} \tau_{\al\beb} = (n+w'-1)W_A^{\al}\tau_{\al\beb} - Z_A\nabla^{\al}\tau_{\al\beb}.
\end{equation}
Conjugating one obtains an operator $\sfM_{\bar{B}}^{\beb}$ on trace free sections of $\cE_{\al\beb}(w,w')$ given by
\begin{equation}\label{aMiddleOperatorTwo}
\sfM_{\bar{B}}^{\beb} \tau_{\al\beb} = (n+w-1)W_{\bar{B}}^{\beb}\tau_{\al\beb} - Z_{\bar{B}}\nabla^{\beb}\tau_{\al\beb}.
\end{equation}

\subsection{The Curvature Tractor}\label{CurvatureTractorSection}

The operators $\sfM_A^{\al}$ defined above are all first order, so in particular they are `strongly invariant' meaning that we may couple the Tanaka-Webster connection $\nabla$ used in their definitions with the tractor connection (on any tractor bundle $\cE^{\Phi}$) to obtain invariant operators $\sfM_A^{\al}$ on sections of $\cE_{\al}{^{\Phi}}(w,w')$ and on trace free sections of $\cE_{\al\beb}{^{\Phi}}(w,w')$. We use these strongly invariant middle operators to define a CR analogue of the conformal `W-tractor' of \cite{Gover-Aspects}.

\begin{defn} The \emph{curvature tractor} of a CR manifold is the section of $\cE_{A\bar{B}C\bar{D}}$ given by
\begin{equation}\label{CurvatureTractor}
\ka_{A\bar{B}C\bar{D}}=\sfM_A^{\al}\sfM_{\bar{B}}^{\beb}\ka_{\al\beb C\bar{D}}
\end{equation}
where $\ka_{\al\beb C\bar{D}}=\ka_{\al\beb}{_C}{^E} h_{E\bar{D}}$.
\end{defn}

\begin{rem}
The expression for the curvature tractor $\ka_{A\bar{B}C\bar{D}}$ does not involve the ($\theta$-dependent) component $\ka_{\al0 C\bar{D}}$ of the tractor curvature $\ka_{ab C\bar{D}}$. One way to include this component in a CR invariant tractor is to define
\begin{equation}\label{CurvatureTractorAlternative}
\ka_{AB\bar{B'}C\bar{D}}= \sfM_A^{\al} \left(
\ka_{\al\be C\bar{D}} W_B^{\be} Z_{\bar{B'}} - 
\ka_{\al\beb C\bar{D}} Z_B W_{\bar{B'}}^{\beb} - i \ka_{\al0 C\bar{D}} Z_B Z_{\bar{B'}} 
\right),
\end{equation}
where we have used the map $T^*M\hookrightarrow\cA$, given explicitly by \eqref{CotangentIntoAdjoint}, on the `$b$' index of $\ka_{abC\bar{D}}$ to obtain $\ka_{aB\bar{B'}C\bar{D}}$ and then applied $\sfM_A^{\al}$ to $\ka_{\al B\bar{B'}C\bar{D}}$. Alternatively one can apply the tensorial map $T^*M\hookrightarrow\cA$ to both the `$a$' and `$b$' indices of $\ka_{abC\bar{D}}$ to obtain $\ka_{A\bar{A'}B\bar{B'}C\bar{D}}$ (as is done in Section \ref{PackagingJets}).
\end{rem}

\subsection{Projecting Parts}

If a standard tractor $v$ lies in subbundle $\cT^{s}$ of $\cT$, $s=1,0,-1$ (see \eqref{StandardTractorFiltration}), then the image of $v$ under the projection
\begin{equation*}
\cT^{s}\rightarrow \cT^{s}/\cT^{s+1}
\end{equation*}
(where the subbundle $\cT^{2}$ is the zero section) is called a \emph{projecting part} of $v$. A projecting part may be zero. Since the filtration of the standard tractor bundle induces a filtration of all corresponding tensor bundles (and hence all tractor bundles), we may define a notion of projecting part(s) similarly for sections of any tractor bundle. 

The notion can be easily formalised using the splitting tractors of Section \ref{SplittingTractors}. The invariant `top slot' $v^A Z_A$ of a standard tractor is always a projecting part. If this top slot vanishes, then the `middle slot' $v^A W_A^{\al}$ is independent of the choice of $\theta$ by \eqref{Wtransformation} and is a projecting part. If both $v^A Z_A=0$ and $v^A W_A^{\al}=0$, then the `bottom slot' $v^A Y_A$ is independent of the choice of $\theta$ by \eqref{Ytransformation} and is a projecting part of $v^A$.

To see how this works for higher valence tractors consider a tractor $t^{AB}$ in $\cE^{[AB]}$. Skewness implies $t^{AB}Z_A Z_B=0$, so $t^{AB}Z_A W_B^{\be}$ is independent of the choice of $\theta$ by \eqref{Wtransformation} and is a projecting part. If $t^{AB}Z_A W_B^{\be}=0$ then both $t^{AB}W_A^{\al} W_B^{\be}$ and $t^{AB}Z_A Y_B$ are independent of the choice of $\theta$, and are both called projecting parts (the relevant composition factor of $\cE^{[AB]}$ splits as a direct sum).

\section{CR Embedded Submanifolds and Contact Forms}\label{Submanifolds-and-Contact-Forms}

We turn now to the main subject of the article. We suppose that $\iota:\Si\hrarrow M$ is a CR embedding of a nondegenerate CR manifold $(\Si^{2m+1},H_{\Si},J_{\Si})$
into $(M^{2n+1},H,J)$, that is $\iota$ is an embedding for which
$T\iota$ maps $H_{\Si}$ into $H$ and
\begin{equation*}
J\circ T\iota=T\iota\circ J_{\Si}.
\end{equation*}
Equivalently (the complex linear extension of) $T\iota$ maps
$T^{1,0}\Si$ into $T^{1,0}M$. 

Suppose $(M^{2n+1},H,J)$ has signature $(p,q)$. Without loss of generality $q\leq p$ ($q$ is often alternatively called the signature). If $q<m$ (in particular, if $M$ is strictly pseudoconvex) then $T_x\iota(T_x\Si)\not\subset H_x$ for all $x\in \Si$. In this case a choice of contact form $\theta$ for $H$ induces a choice of contact form for $H_{\Si}$ by pullback. If $q\geq m$ then we need to impose the condition $T_x\iota(T_x\Si)\not\subset H_x$ for all $x\in \Si$ as an additional assumption; such a CR embedding is said to be \emph{transversal}. (Note that if $T_x\iota(T_x\Si)\subset H_x$ then $T_x\iota(T^{1,0}_x\Si)\subset T^{1,0}_x M$ is a totally isotropic subspace, but the maximum dimension of such a subspace is the signature $q$, so $q\geq m$.) We consider transversal CR embeddings in the following. 

We will work in terms of a pair of pseudohermitian structures $(M,H,\\ J,\theta)$ and $(\Si,H_{\Si}, J_{\Si},\iota^{*}\theta)$ and
aim for constructions which are invariant under ambient rescalings
$\theta\rightarrow\thetah=e^{\Ups}\theta$. More precisely, our goal
is to construct operators and quantities which may be expressed in
terms of the Tanaka-Webster calculus of $\theta$ and of $\iota^{*}\theta$
which are invariant under the replacement of the pair $(\theta,\iota^{*}\theta)$
with $(e^{\Ups}\theta,\iota^{*}(e^{\Ups}\theta))$. 

For simplicity
we will initially restrict our attention to the case where $m=n-1$
($m\geq 1$) and where both manifolds are strictly pseudoconvex (i.e. have positive definite Levi form for positively oriented contact forms). The general codimension (and signature) case is treated in Section \ref{HigherCodimension}, and much carries over directly.

\subsection{Notation}
We fix a bundle of $(1,0)$-densities on $\Si$, that is the dual of an $(m+2)^{th}$
root of $\scrK_{\Si}$, and denote it by $\cE_{\Si}(1,0)$. The corresponding
$(w,w')$-density bundles are denoted $\cE_{\Si}(w,w')$. We use abstract
index notation $\cE^{\mu}$ for $T^{1,0}\Si$, and allow the use of Greek indices from the later part of the alphabet: $\mu$, $\nu$, $\la$, $\rho$, $\mu'$, and so on. Of course
then $\cE^{\mub}$ denotes $T^{0,1}\Si$, $\cE_{\mu}$ denotes $(T^{1,0}\Si)^{*}$,
and so on. We denote the CR Levi form of $\Si$ by $\bh_{\mu\nub}$
and its inverse by $\bh^{\mu\nub}$. We also occasionally use abstract index notation for $T\Si$, denoting it by $\cE^i$ and allowing indices $i$, $j$, $k$, $l$, etcetera.

We identify $\Si$ with its image under $\iota$ and write $\cE^{\al}|_{\Si}\rightarrow\Si$
for the restriction of $\cE^{\al}\rightarrow M$ to fibers over $\Si$
(so $\cE^{\al}|_{\Si}=\iota^{*}\cE^{\al}$). We define the section $\Pi_{\mu}^{\al}$
of $\cE^{\al}|_{\Si}\otimes\cE_{\mu}$ to be (the complex linear extension
of) $T\iota$ as a map from $T^{1,0}\Si$ into $T^{1,0}M$, i.e.
if $X\in T^{1,0}\Si$ and $Y=T\iota(X)$ then $Y^{\al}=\Pi_{\mu}^{\al}X^{\mu}$.
We define the section $\Pi_{\al}^{\mu}$ to be the map from $T^{1,0}M|_{\Si}$
onto $T^{1,0}\Sigma$ given by orthogonal projection with respect
to the CR Levi form.
Clearly $\Pi_{\al}^{\mu}\Pi_{\nu}^{\al}=\de_{\nu}^{\mu}$, and $\Pi_{\mu}^{\al}\Pi_{\be}^{\mu}$
is simply the orthogonal projection map from $T^{1,0}M|_{\Si}$ onto $T\iota(T^{1,0}\Sigma)$
given by the Levi form. It is also clear that
\begin{equation}
\bh_{\mu\nub}=\Pi_{\mu}^{\al}\Pi_{\nub}^{\beb}\bh_{\al\beb}
\end{equation}
along $\Si$.

\subsection{Compatible Scales}

In developing the pseudohermitian and CR tractor calculus we have
been making use of the fact that a choice of contact form $\theta$
for $M$ gives us a direct sum decomposition of the complexified tangent
bundle
\begin{equation*}
\bbC TM=T^{1,0}M\oplus T^{0,1}M\oplus\bbC T,
\end{equation*}
$T$ being the Reeb vector field of $\theta$. Now the contact form
$\theta_{\Sigma}=\iota^{*}\theta$ for $\Si$ also determines a direct
sum decomposition
\begin{equation}\label{SigmaComplexTangentSplitting}
\bbC T\Si=T^{1,0}\Si\oplus T^{0,1}\Si\oplus\bbC T_{\Si}
\end{equation}
where $T_{\Si}$ is the Reeb vector field of $\theta_{\Sigma}$. It
is easy to see that in general these two Reeb vector fields will not
agree along $\Si$. Clearly this will become a problem for us when
we try to relate components of ambient tensor fields (decomposed w.r.t.
$\theta$) with components of submanifold tensor fields (decomposed
w.r.t. $\theta_{\Si}$). To remedy this problem we will make use of
a basic lemma (cf.\ \cite[Lemma 4.1]{EbenfeltHuangZaitsev-Rigidity}).
\begin{lem}\label{CompatibleScales}
Let $\iota:\Si\hrarrow M$ be a CR embedding between nondegenerate
CR manifolds. If $\theta_{\Si}$ is a contact form for $\Si$ with Reeb vector 
field $T_{\Si}$, then there exists a contact form $\theta$ for $M$ with $\iota^{*}\theta=\theta_{\Si}$ and whose Reeb vector field agrees with $T_{\Si}$ along $\Si$. Moreover, the $1$-jet of $\theta$ is uniquely determined along $\Si$.
\end{lem}
\begin{proof}
Fix a contact form $\theta'$ for $M$ with $\iota^{*}\theta'=\theta_{\Si}$.
Let $f$ be an arbitrary smooth (real valued) function on $M$ with
$f|_{\Sigma}\equiv0$, and consider the contact form $\theta=e^{f}\theta'$.
First of all we have
\begin{equation*}
(T\iota\cdot T_{\Sigma})\lrcorner\mathrm{d}\theta=e^{f}(T\iota\cdot T_{\Sigma})\lrcorner\mathrm{d}\theta'+e^{f}\mathrm{d}f
\end{equation*}
along $\Si$ since $\theta'(T\iota\cdot T_{\Sigma})=(\iota^{*}\theta')(T_{\Sigma})=\theta_{\Sigma}(T_{\Sigma})=1$.
Now since $\iota^{*}\theta=\theta_{\Si}$ we have
\begin{equation*}
\iota^{*}((T\iota\cdot T_{\Sigma})\lrcorner\mathrm{d}\theta)=T_{\Si}\hook\d\theta_{\Si}=0.
\end{equation*}
This means that $(T\iota\cdot T_{\Sigma})\lrcorner\mathrm{d}\theta$
is zero when restricted to tangential directions. Consequently, we
only need to see if we can make $(T\iota\cdot T_{\Sigma})\lrcorner\mathrm{d}\theta$
zero on the quotient space $TM|_{\Sigma}/T\Sigma$. This requires
choosing $f$ such that along $\Si$
\begin{equation*}
\d f=-(T\iota\cdot T_{\Sigma})\lrcorner\mathrm{d}\theta'
\end{equation*}
on $TM|_{\Sigma}/T\Sigma$, which simply amounts to prescribing the
normal derivatives of $f$ off $\Sigma$. Choosing such an $f$ we
have that $(T\iota\cdot T_{\Sigma})\lrcorner\mathrm{d}\theta=0$ and
$\theta(T\iota\cdot T_{\Sigma})=\theta_{\Sigma}(T_{\Sigma})=1$ as
required.
\end{proof}
\begin{defn}\label{CompatibleScalesDef}
A pair of contact forms $\theta$, $\theta_{\Si}$ for $M$ and $\Si$
respectively will be called \emph{compatible} if $\theta_{\Si}=\iota^{*}\theta$
and the Reeb vector field of $\theta$ restricts to the Reeb vector
field of $\theta_{\Si}$ along $\Si$. A contact form $\theta$ which
is compatible with $\iota^{*}\theta$, i.e. whose Reeb vector field is tangent to $\Si$, will be said to be \emph{admissible} \cite{EbenfeltHuangZaitsev-Rigidity}.
\end{defn}
We will work primarily in terms of compatible contact forms in the
following. When working in terms of compatible contact forms $\theta$
for $M$ and $\theta_{\Si}$ for $\Si$ we identify the density bundles
$\cE(1,1)|_{\Si}$ with $\cE_{\Si}(1,1)$ using the trivialisations
of these bundles induced by $\theta$ and $\theta_{\Si}$ respectively
(in fact this identification is canonical, i.e. it is independent of
the choice of compatible contact forms). We also identify the Reeb
vector field $T_{\Si}$ of $\theta_{\Si}$ with $T|_{\Si}$ where
$T$ is the ambient Reeb vector field. This means that the `$0$-component'
of $X\in T\Si$ taken with respect to either $\theta_{\Si}$ or $\theta$ (identifying $X$ with $T\iota\cdot X$) is the same, and that
our ambient and intrinsic decompositions of tensors will always be
nicely compatible.
\begin{rem}\label{CompatibleScalesRemark}
Note that Lemma \ref{CompatibleScales} holds for general codimension CR embeddings (with the same proof). We can therefore continue to work with compatible contact forms in the general codimension case discussed in Section \ref{HigherCodimension}.
\end{rem}

\subsection{Normal Bundles}\label{sub:Holomorphic-Normal-Fields}

Clearly $T^{1,0}\Si$ has complex corank one inside $T^{1,0}M|_{\Si}$.
The CR Levi form determines then a canonical complex line bundle $\cN^{\al}\subset\cE^{\al}|_{\Si}$
whose sections are those $V^{\al}$ for which
$\Pi_{\al}^{\mu}V^{\al}\equiv0$. There is also the corresponding dual complex
line bundle $\cN_{\al}\subset\cE_{\al}|_{\Si}$ whose sections $V_{\al}$
satisfy $V_{\al}\Pi_{\mu}^{\al}\equiv0$. 
\begin{rem}
Given any choice of ambient contact form $\theta$ the manifold $M$ gains a Riemannian structure from the Webster metric $g_{\theta}$. One can therefore treat $\Sigma$ as a Riemannian submanifold, in particular we have a Riemannian normal bundle to $\Si$. This Riemannian normal bundle will be the same for any admissible contact form $\theta$, and we denote it by $N\Si$. Complexifying we see that $\bbC N\Si = \cN^{\al}\oplus \cN^{\alb}$ where $\cN^{\al}$ is the $i$-eigenspace of $J$.
\end{rem}
\subsubsection{Unit Normal Fields}
Given a choice of ambient contact form $\theta$, one may ask that a section $N_{\al}$ be unit with respect to the Levi form of $\theta$. However, for CR geometry it is more natural to work with sections of the bundle $\cN_{\al}(1,0)=\cN_{\al}\otimes\cE(1,0)|_{\Si}$, which is normed by the CR Levi form. Thus we make the following definition:
\begin{defn}
By a \emph{(weighted) unit holomorphic conormal field} we mean a section $N_{\al}$ of $\cN_{\al}(1,0)$ for which $\bh^{\al\beb}N_{\al}N_{\beb}=1$ where $N_{\beb}=\overline{N_{\be}}$. The field $N^{\al}=\bh^{\al\beb}N_{\beb}$ obtained from such an $N_{\al}$ will be referred to as a \emph{(weighted) unit holomorphic normal field}. 
\end{defn}
\begin{rem} 
The bundles $\cN_{\al}(w+1,-w)$ are also normed by the CR Levi form, but the natural weight for conormals is indeed $(1,0)$. The line bundle $\cN_{\al}(1,0)$ plays an important role in the following since it relates ambient and intrinsic density bundles (see \eqref{ambient-intrinsic-densities} below). Moreover, $\cN_{\al}(1,0)$ can be canonically identified with a non-null subbundle of the ambient cotractor bundle $\cE_A|_{\Si}$, and hence carries a canonical CR invariant connection (see Proposition \ref{NormalTractorConnectionProp}).
\end{rem}

If $N_{\al}$ is a unit holomorphic conormal
then so is $N'_{\al}=e^{i\varphi}N_{\al}$ for any $\varphi\in C^{\infty}(\Si)$,
and $N'^{\al}=e^{-i\varphi}N^{\al}$. However, the
combinations $N^{\al}N_{\be}$ and $N_{\al}N_{\beb}$ are independent
of the choice of holomorphic conormal, and these satisfy
\begin{equation}
\delta_{\be}^{\al}=\Pi_{\be}^{\al}+N^{\al}N_{\be}\quad\mbox{and}\quad\bh_{\al\beb}=\bh_{\mu\nub}\Pi_{\al}^{\mu}\Pi_{\beb}^{\nub}+N_{\al}N_{\beb}
\end{equation}
along $\Si$, where $\Pi_{\be}^{\al}$ is the tangential orthogonal projection $\Pi^{\al}_{\mu}\Pi_{\be}^{\mu}$ and $\bh_{\mu\nub}$ is the CR Levi form of $\Si$.

\subsection{Tangential Derivatives}

Let $\theta$ be an admissible ambient contact form with Tanaka-Webster connection
$\nabla$. The pullback connection $\iota^{*}\nabla$ allows us to
differentiate sections of ambient tensor bundles along $\Si$ in directions
tangential to $\Si$. Recall that we may think of the Tanaka-Webster connection
$\nabla$ as a triple of `partial connections' $(\nabla_{\al},\nabla_{\alb},\nabla_{0})$.
Now suppose that the Reeb vector field $T$ of $\theta$ is tangent
to $\Si$, then $\theta$ and $\theta_{\Si}=\iota^{*}\theta$ are
compatible. Then we can break up $\iota^{*}\nabla$ into a corresponding
triple $(\nabla_{\mu},\nabla_{\mub},\nabla_{0})$.
Precisely, $\nabla_{\mu}$ is defined to act on sections of $\cE^{\al}|_{\Sigma}$
according to the formula
\begin{equation}
\nabla_{\mu}\tau^{\al}=\Pi_{\mu}^{\be}\nabla_{\be}\tilde{\tau}^{\al}
\end{equation}
where $\tilde{\tau}^{\al}$ is any extension of the section $\tau^{\al}$
of $\cE^{\al}|_{\Sigma}$ to a neighbourhd of $\Sigma$, and $\nabla_{\mu}$
is defined similarly on sections of $\cE^{\alb}|_{\Sigma}$, $\cE_{\al}|_{\Sigma}$,
and so on. We define $\nabla_{\mub}$ similarly, and define $\nabla_{0}$
on sections of $\cE^{\al}|_{\Sigma}$ by the formula
\begin{equation}
\nabla_{0}\tau^{\al}=\nabla_{0}\tilde{\tau}^{\al}
\end{equation}
along $\Si$, where $\tilde{\tau}^{\al}$ is any extension of $\tau^{\al}$, and
similarly on sections of $\cE^{\alb}|_{\Sigma}$, $\cE_{\al}|_{\Sigma}$,
and so on (note the independence of the choice of the extension relies
on the fact that $T$ is tangential to $\Si$). 
\begin{rem*}
We have identified
$\cE(1,1)|_{\Si}$ with $\cE_{\Si}(1,1)$ and $T|_{\Si}$ and with
the Reeb vector field $T_{\Si}$ of $\theta_{\Si}$, thus splitting
$\iota^{*}\nabla$ up into $(\nabla_{\mu},\nabla_{\mub},\nabla_{0})$
corresponds precisely to restricting $\iota^{*}\nabla$ to the respective
summands in the direct sum decomposition \eqref{SigmaComplexTangentSplitting} induced by $\theta_{\Si}$.
\end{rem*}

\subsubsection{The normal Tanaka-Webster connection}
The ambient Tanaka-Webster connection also induces a connection on the normal bundle.

\begin{defn} Given an admissible ambient contact form $\theta$, we define the \emph{normal Tanaka-Webster connection} $\nabla^{\perp}$ on $\cN_{\al}$ by differentiating tangentially using the Tanaka-Webster connection $\nabla$ of $\theta$ and then projecting orthogonally onto $\cN_{\al}$ using the Levi form.
\end{defn}

\subsection{The Submanifold Tanaka-Webster Connection}

We may define a connection $D$ on $T^{1,0}\Si=\cE^{\mu}$
(which we identify with $T\iota(T^{1,0}\Si)$ in $T^{1,0}M|_{\Si}$)
by differentiating in tangential directions using $\iota^{*}\nabla$
and projecting the result back onto $T^{1,0}\Si=\cE^{\mu}$ orthogonally
with respect to the Levi form. This means that if $\tau^{\mu}$ is
a section of $\cE^{\mu}$ then we have
\begin{equation}
D_{\nu}\tau^{\mu}=\Pi_{\al}^{\mu}\nabla_{\nu}\tau^{\al}
\end{equation}
where $\tau^{\al}=\Pi_{\la}^{\al}\tau^{\la}$. One may define $D$
to act also on $T^{0,1}\Si=\cE^{\mub}$ by the analogous formula
\begin{equation}
D_{\nu}\tau^{\mub}=\Pi_{\alb}^{\mub}\nabla_{\nu}\tau^{\alb}.
\end{equation}
Thus $D$ may be thought of as a connection on $H_{\Si}$ which preserves
$J_{\Si}$. One may then extend $D$ to a connection on $T\Si$ by
requiring that $T_{\Si}$ be parallel. 
\begin{rem*}
Equivalently one may define
$D$ as a connection on $T\Si$ from the start by differentiating
tangent vectors to $\Si$ in tangential directions using $\iota^{*}\nabla$
and projecting the result back onto $T\Si$ orthogonally with respect
to the Webster metric of $\theta$.
\end{rem*}
Provided $\theta$ and $\theta_{\Si}$ are compatible, the connection $D$ constructed in this manner will be the Tanaka-Webster connection of $\theta_{\Si}$ (cf.\ \cite[Theorem 6.4]{DragomirTomassini}):
\begin{prop}\label{SubmanifoldTWprop}
If $\theta$, $\theta_{\Si}$ are contact forms for $M$ and $\Si$
respectively which are compatible, that is, $\theta_{\Si}=\iota^{*}\theta$
and the Reeb vector field of $\theta$ is tangential to $\Si$, then
the connection $D$ on $T\Si$ induced by the Tanaka-Webster connection
$\nabla$ of $\theta$ (and projection with respect to the ambient
Webster metric) is the Tanaka-Webster connection of $\theta_{\Si}$.
\end{prop}
\begin{proof}
We need to show that $D$ preserves $(H_{\Si},J_{\Si},\theta_{\Si})$ and satisfies the torsion conditions of Section \ref{sub:The-Tanaka-Webster-Connection}. It is clear that $D$ preserves the decomposition \eqref{SigmaComplexTangentSplitting} and gives a linear connection on each of the three direct summands. This implies that $D$ preserves $H$ and $J$ in the appropriate senses.  Since $\nabla$ preserves the Reeb vector field $T$, $\iota^{*}\nabla$ preserves $T|_{\Si}=T_{\Si}$ and hence $DT_{\Si}=0$. Since $D$ preserves $T_{\Si}$ and $H$ it must also preserve $\theta_{\Si}$.

Now let $f\in C^{\infty}(\Si)$ and choose an extension $\tilde{f}$
of $f$ to $M$ such that along $\Si$ we have $\nabla_{\al}\tilde{f}=\Pi_{\al}^{\mu}\nabla_{\mu}f$,
i.e. the derivative of $\tilde{f}$ vanishes in $g_{\theta}$-normal directions along $\Si$ (these directions don't depend on $\theta$ so long as we choose $\theta$ admissible). Then we have that $\nabla_{\beb}\tilde{f}=\Pi_{\beb}^{\lab}\nabla_{\lab}f$
along $\Si$ and hence also that
\begin{align*}
D_{\mu}D_{\nub}f-D_{\nub}D_{\mu}f & =D_{\mu}\nabla_{\nub}f-D_{\nub}\nabla_{\mu}f\\
 & =\Pi_{\nub}^{\beb}\nabla_{\mu}(\Pi_{\beb}^{\lab}\nabla_{\lab}f)-\Pi_{\mu}^{\al}\nabla_{\nub}(\Pi_{\al}^{\la}\nabla_{\la}f)\\
 & =\Pi_{\nub}^{\beb}\nabla_{\mu}\nabla_{\beb}\tilde{f}-\Pi_{\mu}^{\al}\nabla_{\nub}\nabla_{\al}\tilde{f}\\
 & =\Pi_{\mu}^{\al}\Pi_{\nub}^{\beb}(\nabla_{\al}\nabla_{\beb}\tilde{f}-\nabla_{\beb}\nabla_{\al}\tilde{f})\\
 & =\Pi_{\mu}^{\al}\Pi_{\nub}^{\beb}(-i\bh_{\al\be}\nabla_{0}\tilde{f})\\
 & =-i\bh_{\mu\nub}D_{0}f
\end{align*}
where we have used that $D_{\mu}f=\nabla_{\mu}f$ and $D_{\nub}f=\nabla_{\nub}f$
as well as that $D_{0}f=\nabla_{0}f=\nabla_{0}\tilde{f}$ along $\Si$
. Similarly we may easily compute that
\[
D_{\mu}D_{\nu}f-D_{\nu}D_{\mu}f=0.
\]
Finally we have
\begin{align*}
D_{\mu}D_{0}f-D_{0}D_{\mu}f & =\nabla_{\mu}\nabla_{0}f-\nabla_{0}\nabla_{\mu}f\\
 & =\Pi_{\mu}^{\al}(\nabla_{\al}\nabla_{0}\tilde{f}-\nabla_{0}\nabla_{\al}\tilde{f})\\
 & =\Pi_{\mu}^{\al}A^{\gab}{_{\al}}\nabla_{\gab}\tilde{f}\\
 & =\Pi_{\mu}^{\al}\Pi_{\gab}^{\lab}A^{\gab}{_{\al}}\nabla_{\lab}f=A^{\lab}{_{\mu}}D_{\lab}f
\end{align*}
where $A_{\mu\nu}=\Pi_{\mu}^{\al}\Pi_{\nu}^{\be}A_{\al\be}$. Since
$f$ was arbitrary, we conclude that $D$ is the Tanaka-Webster connection
of $\theta_{\Si}$.
\end{proof}
\begin{cor}\label{SubmanifoldTWpropCor}
Given an admissible ambient contact form $\theta$ with pseudohermitian torsion $A_{\al\be}$, the pseudohermitian torsion of $\theta_{\Si}=\iota^*\theta$ is $A_{\mu\nu}= \Pi_{\mu}^{\al}\Pi_{\nu}^{\be}A_{\al\be}$.
\end{cor}
\begin{rem}\label{SubmanifoldTWpropRemark}
Note that Proposition \ref{SubmanifoldTWprop} and Corollary \ref{SubmanifoldTWpropCor} hold in the general codimension case by the same arguments.
\end{rem}

\subsection{The Second Fundamental Form}\label{sub:SFF}

We can now define a second fundamental form using an analogue of the
Gauss formula from Riemannian submanifold geometry. 
\begin{defn}\label{PHsff}
Given $\theta$ and $\theta_{\Si}$ compatible with respective Tanaka-Webster
connections $\nabla$ and $D$ we define the \emph{(pseudohermitian) second fundamental form} by
\begin{equation}\label{GaussFormula}
\nabla_{X}Y=D_{X}Y+\II(X,Y),
\end{equation}
for all $X,Y\in\fX(\Si)$, where we implicitly identify submanifold vector fields with tangential ambient vector fields along $\Si$ and use the pullback connection $\iota^*\nabla$ on the left hand side.
\end{defn}
Clearly $\II(X,Y)$ is tensorial in $X$ and $Y$, and is normal bundle ($N\Si$) valued. It is also clear from the definition that $\II(\,\cdot\,,T_{\Si})=0$ and that 
$\II(\,\cdot\,,\cdot\,)|_{H_{\Si}}$ is complex linear (with respect to $J$ and $J_{\Si}$) in the second
argument, that is 
\begin{equation*}
\II(\,\cdot\,,J_{\Si}\cdot\,)|_{H_{\Si}}=J\II(\,\cdot\,,\cdot\,)|_{H_{\Si}}.
\end{equation*}
In fact, these properties also hold for the first argument, $\II$ being symmetric.
\begin{prop}\label{SFFprop}
The only nonzero components of the (pseudohermitian) second fundamental form $\II$ are $\II_{\mu\nu}{^{\ga}}$ and its conjugate. Moreover
\begin{equation}
\II_{\mu\nu}{^{\ga}}=\II_{\nu\mu}{^{\ga}},
\end{equation}
so that $\II$ is symmetric.
\end{prop}
\begin{proof}
Since $\II(\,\cdot\,,T_{\Si})=0$ and $\II(\,\cdot\,,\cdot\,)|_{H_{\Si}}$ is complex linear in the second argument, to prove the first claim it suffices to show that $\II_{0\nu}{^{\ga}}=0$ and $\II_{\mub\nu}{^{\ga}}=0$.

Let $N_{\al}$ be a section of $\cN_{\al}$ such that $h^{\al\beb}N_{\al}N_{\beb}=1$. From the Gauss formula \eqref{GaussFormula} we have that
\begin{equation}\label{GaussFormulaComponents}
\nabla_i V^{\ga} = \Pi^{\ga}_{\la} D_i V^{\la}+\II_{i\nu}{^{\ga}}V^{\nu}
\end{equation}
for any section $V^{\la}$ of $\cE^{\la}$, where $V^{\ga}=\Pi^{\ga}_{\la}V^{\la}$. Contracting the above display with $N_{\ga}$ and replacing the `$i$' index with `$\mu$', `$\mub$', and `$0$' respectively gives
\begin{align*}
N_{\ga}\II_{\mu\nu}{^{\ga}}=-\Pi^{\ga}_{\nu}\nabla_{\mu}N_{\ga}&, \quad
N_{\ga}\II_{\mub\nu}{^{\ga}}=-\Pi^{\ga}_{\nu}\nabla_{\mub}N_{\ga}, \\*
\mathrm{and} \quad N_{\ga}\II_{0\nu}{^{\ga}}&=-\Pi^{\ga}_{\nu}\nabla_{0}N_{\ga},
\end{align*}
since $N_{\ga}V^{\ga}=0$ for all $V^{\la}\in\Gamma(\cE^{\la})$. By conjugating, one also has that $N_{\gab}\II_{\nu\mub}{^{\gab}}=-\Pi^{\gab}_{\mub}\nabla_{\nu}N_{\gab}$.

Now let $f$ be a real valued function on $M$ which vanishes on $\Si$ and for which $(\nabla_{\al}f,\nabla_{\alb}f,\nabla_{0}f)$ is equal
to $(N_{\al},N_{\alb},0)$ along $\Si$. (Note that we must require $\nabla_{0}f$ to be zero along $\Si$ since $T$ is tangent to $\Si$ and we ask that $f|_{\Si}\equiv0$. Such an $f$ exists because we are simply prescribing the normal derivatives of $f$ off $\Sigma$. Any such $f$ is, locally about $\Sigma$, a defining function for a real hypersurface in $M$ containing $\Sigma$ which is $g_{\theta}$-orthogonal to the real part and tangent to the imaginary part of $N^{\al}$.) From \eqref{torsion1} and \eqref{torsion2} we have that
\begin{equation*}
\nabla_{\al}\nabla_{\be}f=\nabla_{\be}\nabla_{\al}f\quad\mathrm{and}\quad\nabla_{\al}\nabla_{\beb}f=\nabla_{\beb}\nabla_{\al}f
\end{equation*}
along $\Si$. Projecting tangentially along $\Si$ we immediately have that
\begin{equation*}
N_{\ga}\II_{\mu\nu}{^{\ga}}=N_{\ga}\II_{\nu\mu}{^{\ga}}\quad\mathrm{and}\quad N_{\gab}\II_{\mu\nub}{^{\gab}}=N_{\ga}\II_{\nub\mu}{^{\ga}}.
\end{equation*}
The first of these implies that $\II_{\mu\nu}{^{\ga}}=\II_{\nu\mu}{^{\ga}}$.
Since $N_{\ga}$ was arbitrary the second implies that $\II_{\mu\nub}{^{\gab}}=0$ (replacing $N_{\ga}$ with $iN_{\ga}$ gives a minus sign).

Using the same function $f$, \eqref{PHtorsion} states
\begin{equation*}
\nabla_{\al}\nabla_{0}f-\nabla_{0}\nabla_{\al}f=A^{\gab}{_{\al}}\nabla_{\gab}f.
\end{equation*}
Applying $\Pi_{\mu}^{\al}$ to both sides of the above
display we get that
\begin{equation*}
-\Pi_{\mu}^{\al}\nabla_{0}\nabla_{\al}f=\Pi_{\mu}^{\al}A^{\gab}{_{\al}}\nabla_{\gab}f
\end{equation*}
along $\Si$ (since $\Pi_{\mu}^{\al}\nabla_{\al}\nabla_{0}f$ is zero
along $\Si$). We conclude that
\begin{equation*}
N_{\ga}\II_{0\mu}{^{\ga}}=N_{\gab}\Pi_{\mu}^{\al}A_{\al}^{\gab}.
\end{equation*}
Again, since $N_{\ga}$ was arbitrary we must have
\begin{equation}\label{ThookSSF}
\II_{0\mu}{^{\ga}}=0\quad\mathrm{and}\quad N_{\gab}\Pi_{\mu}^{\al}A_{\al}^{\gab}=0.
\end{equation}
\end{proof}
The second of the expressions \eqref{ThookSSF} should be seen as a constraint on the pseudohermitian torsion of an admissible contact form. We state this as a corollary:
\begin{cor}\label{ATanNorCor}
If $\theta$ is an admissible ambient contact form then the pseudohermitian torsion of $\theta$ satisfies
\begin{equation}
\Pi^{\al}_{\mu} A_{\al\be} N^{\be}=0
\end{equation}
for any holomorphic normal field $N^{\be}$.
\end{cor}
\begin{rem}\label{sffHigherCodimensionRemark}
Note that in the higher codimension case (of CR embeddings) if one defines the (pseudohermitian) second fundamental form of a pair of compatible contact forms as in Definition \ref{PHsff} then Proposition \ref{SFFprop} holds with the proof unchanged (and consequently Corollary \ref{ATanNorCor} also holds).
\end{rem}
\begin{rem}\label{Qremark}
Our claim that $\II(T,\,\cdot\,)=0$, and the above corollary, disagree with \cite{Dragomir} and the book \cite{DragomirTomassini}. Our claim that $\II(T,\,\cdot\,)=0$ is confirmed however by the later article \cite{DragomirMinor}.
\end{rem}
\subsubsection{The CR second fundamental form}
We shall now see that the component $\II_{\mu\nu}{^{\ga}}$ does not depend on the choice of compatible contact forms $\theta$ and $\theta_{\Si}$.
\begin{lem}\label{CRsffLemma}
Given compatible contact forms $\theta$ and $\theta_{\Si}$ one has
\begin{equation}\label{CRsff}
\II_{\mu\nu}{^{\ga}}=-N^{\ga}\Pi_{\nu}^{\be}\nabla_{\mu}N_{\be}
\end{equation}
for any unit holomorphic conormal field $N_{\al}$.
\end{lem}
\begin{proof}
From the Gauss formula (cf.\ \eqref{GaussFormulaComponents}) we have 
\begin{equation*}
\nabla_{\mu} V^{\ga} = \Pi^{\ga}_{\la} D_{\mu} V^{\la}+\II_{\mu\nu}{^{\ga}}V^{\nu}
\end{equation*}
for any section $V^{\la}$ of $\cE^{\la}$, where $V^{\ga}=\Pi^{\ga}_{\la}V^{\la}$. Contracting the above display with $N_{\ga}$ and using that $N_{\ga}\nabla_{\mu}V^{\ga}=-V^{\ga}\nabla_{\mu}N_{\ga}$ yields the result.
\end{proof}
\begin{cor}\label{CRsffCor}
The component $\II_{\mu\nu}{^{\ga}}$ of the pseudohermitian second fundamental form does not depend on the pair of compatible contact forms used to define it.
\end{cor}
\begin{proof}
Combining the Tanaka-Webster transformation laws of Proposition \ref{TWtransform} and Proposition \ref{TWonDensitiesProp} we have that
\begin{equation*}
\nablah_{\mu}N_{\be}
=\nabla_{\mu}N_{\be}-\Pi^{\al}_{\mu}\Ups_{\be}N_{\al}-\Ups_{\mu}N_{\be} + \Ups_{\mu}N_{\be}
=\nabla_{\mu}N_{\be}
\end{equation*}
since $N_{\be}$ has weight $(1,0)$. The claim then follows from \eqref{CRsff}.
\end{proof}
We therefore term $\II_{\mu\nu}{^{\ga}}$ the \emph{CR second fundamental form}.
\begin{rem}
The pseudohermitian second fundamental form $\II$ (of a pair of compatible contact forms) is not CR invariant, even though $\II_{\mu\nu}{^{\ga}}$ is, since the direct sum decompositions of $\bbC TM$ and $\bbC T^*\Si$ change under rescaling of the ambient and submanifold contact forms.
\end{rem}

Recall that we write $\Pi_{\be}^{\al}$ for the tangential orthogonal projection $\Pi^{\al}_{\mu}\Pi_{\be}^{\mu}$ on the ambient holomorphic tangent bundle along $\Si$. The following lemma will be useful in the derivations of Section \ref{Pseudohermitian-GCR}:
\begin{lem}\label{ProjectorDerivLemma}
For any admissible ambient contact form we have
\begin{equation}\label{ProjectorDeriv1}
\nabla_{\mu}\Pi^{\ga}_{\be}=\II_{\mu\nu}{^{\ga}}\Pi^{\nu}_{\be}, \qquad 
\nabla_{\mu}\Pi^{\beb}_{\gab}=\II_{\mu}{^{\nub}}_{\gab}\Pi^{\beb}_{\nub},
\end{equation}
\begin{equation}\label{ProjectorDeriv2}
\nabla_{\mub}\Pi^{\gab}_{\beb}=\II_{\mub\nub}{^{\gab}}\Pi^{\nub}_{\beb}, \qquad 
\nabla_{\mub}\Pi^{\be}_{\ga}=\II_{\mub}{^{\nu}}_{\ga}\Pi^{\be}_{\nu},
\end{equation}
and
\begin{equation}\label{ProjectorDeriv3}
\nabla_{0}\Pi^{\ga}_{\be}=0,  \qquad
\nabla_{0}\Pi^{\beb}_{\gab}=0.
\end{equation}
\end{lem}
\begin{proof}
These follow immediately by differentiating $\delta^{\ga}_{\be}-N^{\ga}N_{\be}$, however we wish to give a proof that will also work in the higher codimension case. Pick a section $V^{\nu}$ and let $V^{\be}=\Pi^{\be}_{\nu}V^{\nu}$. Then 
\begin{equation*}
\nabla_{\mu} V^{\ga} = \nabla_{\mu} (\Pi^{\ga}_{\be}V^{\be})=(\nabla_{\mu} \Pi^{\ga}_{\be})V^{\be} + \Pi^{\ga}_{\be}\nabla_{\mu} V^\be.
\end{equation*}
Noting that $\Pi^{\ga}_{\be}\nabla_{\mu} V^\be=\Pi^{\ga}_{\nu}D_{\mu} V^{\nu}$, from the Gauss formula we have
\begin{equation}
\Pi^{\be}_{\nu}\nabla_{\mu}\Pi^{\ga}_{\be}=\II_{\mu\nu}{^{\ga}}
\end{equation}
since $V^{\nu}$ was arbitrary. Now on the other hand if $N^{\al}$ is any unit holomorphic normal then
\begin{equation}
N^{\be}\nabla_{\mu} \Pi^{\ga}_{\be} = - (\nabla_{\mu} N^{\be})\Pi^{\ga}_{\be} = 0
\end{equation}
since $\II_{\mu\nub}{^{\deb}}=0$. The previous two displays imply the first equation of \eqref{ProjectorDeriv1}, and the second then follows by raising and lowering indices. Conjugating these gives \eqref{ProjectorDeriv2}. The expressions \eqref{ProjectorDeriv3} are proved similarly using that $\II_{0\nu}{^{\ga}}=0$ and $\II_{0\nub}{^{\gab}}=0$.
\end{proof}

\subsection{The Pseudohermitian Gauss, Codazzi, and Ricci Equations}\label{Pseudohermitian-GCR}

Here we give pseudohermitian analogues of the Gauss, Codazzi, and Ricci equations from Riemannian submanifold theory. Real forms of these equations can be found in chapter 6 of \cite{DragomirTomassini}, note that $Q=0$ in the pseudohermitian Codazzi equation they give (cf.\ Remark \ref{Qremark}).

When working with compatible contact forms we denote the ambient and submanifold Tanaka-Webster connections by $\nabla$ and $D$ respectively. We write 
\begin{equation}
\rmN^{\al}_{\be} = \delta_{\be}^{\al} - \Pi_{\be}^{\al}
\end{equation}
for the orthogonal projection onto $\cN^{\al}\subset \cE^{\al}|_{\Si}$. In this case  $\rmN^{\al}_{\be} = N^{\al}N_{\be}$ for any unit holomorphic normal $N^{\al}$. We adopt the convention of replacing uppercase root letters with lowercase root letters for submanifold curvature tensors, so the pseudohermitian curvature tensor of $\theta_{\Si}$ will be denoted by $r_{\mu\nub\la\rhob}$, the pseudohermitian Ricci curvature by $r_{\mu\nu}$, and so on. For the ambient curvature tensors along $\Si$ we will use submanifold abstract indices to denote tangential projections, for example
\begin{equation*}
R_{\mu\nub\la\rhob}=
\Pi^{\al}_{\mu}\Pi^{\beb}_{\nub}\Pi^{\ga}_{\la}\Pi^{\deb}_{\rhob}R_{\al\beb\ga\deb}
\quad \mathrm{and} \quad 
R_{\mu\nub \ga \deb}= \Pi^{\al}_{\mu}\Pi^{\beb}_{\nub}R_{\al\beb\ga\deb}.
\end{equation*}

\subsubsection{The pseudohermitian Gauss equation}\label{PseudohermitianGaussSect}
\begin{prop} 
Given compatible contact forms, the submanifold pseudohermitian curvature is related to the ambient curvature via
\begin{equation}\label{pseudohermitian-Gauss}
R_{\mu\nub\la\rhob}=r_{\mu\nub\la\rhob}+\bh_{\ga\deb}\II_{\mu\la}{^{\ga}}\II_{\nub\rhob}{^{\deb}}.
\end{equation}
\end{prop}
\begin{proof}
Let $V$ be a section of $T^{1,0}\Si$, let $V^{\gab}=\Pi_{\lab}^{\gab}V^{\lab}$,
and let $\tilde{V}^{\gab}$ be a smooth extension of $V^{\gab}$ to
a neighbourhood of $\Si$. Proposition \ref{SubmanifoldTWprop} says that $D_{\nub}V^{\lab}=\Pi^{\lab}_{\gab}\Pi^{\beb}_{\nub}\nabla_{\beb}\tilde{V}^{\gab}$ and thus
\begin{align*}
D_{\mu}D_{\nub}V^{\lab} & =\Pi_{\deb}^{\lab}\Pi_{\nub}^{\epb}\nabla_{\mu}\left(\Pi_{\gab}^{\deb}\Pi_{\epb}^{\beb}\nabla_{\beb}\tilde{V}^{\gab}\right)\\
 & =\Pi_{\deb}^{\lab}\Pi_{\nub}^{\epb}(\nabla_{\mu}\Pi_{\gab}^{\deb})\Pi_{\epb}^{\beb}\nabla_{\beb}\tilde{V}^{\gab}+\Pi_{\deb}^{\lab}\Pi_{\nub}^{\epb}\Pi_{\gab}^{\deb}(\nabla_{\mu}\Pi_{\epb}^{\beb})\nabla_{\beb}\tilde{V}^{\gab}\\*
 & \quad+\Pi_{\deb}^{\lab}\Pi_{\nub}^{\epb}\Pi_{\gab}^{\deb}\Pi_{\epb}^{\beb}\nabla_{\mu}\nabla_{\beb}\tilde{V}^{\gab}\\
 & =\II_{\mu}{^{\lab}}_{\gab}\nabla_{\nub}V^{\gab}+\Pi_{\gab}^{\lab}\Pi_{\mu}^{\al}\Pi_{\nub}^{\beb}\nabla_{\al}\nabla_{\beb}\tilde{V}^{\gab},
\end{align*}
where we have used \eqref{ProjectorDeriv1} of Lemma \ref{ProjectorDerivLemma} in the final step. Since $\rmN^{\de}_{\ga}V^{\ga}=0$ we have $\rmN^{\deb}_{\gab}\nabla_{\nub}V^{\gab}=-V^{\gab}\nabla_{\nub}\rmN^{\deb}_{\gab}=V^{\rhob}\II_{\nub\rhob}{^{\deb}}$ using \eqref{ProjectorDeriv2} of Lemma \ref{ProjectorDerivLemma}, and hence by writing $\II_{\mu}{^{\lab}}_{\gab}$ as $\II_{\mu}{^{\lab}}_{\deb}\rmN^{\deb}_{\gab}$ we obtain 
\begin{equation*}
D_{\mu}D_{\nub}V^{\lab}=\II_{\mu}{^{\lab}}_{\deb}\II_{\nub\rhob}{^{\deb}}V^{\rhob}+\Pi_{\gab}^{\lab}\Pi_{\mu}^{\al}\Pi_{\nub}^{\beb}\nabla_{\al}\nabla_{\beb}\tilde{V}^{\gab}. 
\end{equation*}
By a similar calculation with the roles of $\mu$ and $\nu$ interchanged we obtain
\begin{equation*}
D_{\nub}D_{\mu}V^{\lab}=\Pi_{\gab}^{\lab}\Pi_{\mu}^{\al}\Pi_{\nub}^{\beb}\nabla_{\beb}\nabla_{\al}\tilde{V}^{\gab}; 
\end{equation*}
no second fundamental form terms arise since $\II_{\mu\nub}{^{\gab}}=0$. Noting that $D_0 V^{\lab}=\Pi^{\lab}_{\gab}\nabla_0 V^{\gab}$ we have the result.
\end{proof}
\begin{rem}
The above proposition holds with the same proof in the general codimension setting. The equation \eqref{pseudohermitian-Gauss} can also be found in \cite{EbenfeltHuangZaitsev-Rigidity} where it (or its trace free part) is the key to proving rigidity for CR embeddings into the sphere with sufficiently low codimension because it allows one to show that the intrinsic pseudohermitian curvature determines the second fundamental form $\II_{\mu\nu}{^{\ga}}$.
\end{rem}

\subsubsection{The pseudohermitian Codazzi equation}\label{PseudohermitianCodazziSect}

\begin{prop}\label{pseudohermitian-Codazzi-prop}
Given compatible contact forms, 
\begin{equation}\label{pseudohermitian-Codazzi}
R_{\mu\nub}{^{\gab}}_{\rhob} \rmN_{\gab}^{\deb} = -D_{\mu}\II_{\nub\rhob}{^{\deb}}
\end{equation}
where the submanifold Tanaka-Webster connection $D$ is coupled with the normal Tanaka-Webster connection $\nabla^{\perp}$.
\end{prop}
\begin{proof}
Let  $N^{\ga}$ be an unweighted unit normal field and let $\tilde{N}^{\ga}$ be an extension of $N^{\ga}$ to all of $M$ such that, along $\Si$, 
$\rmN_{\be}^{\al}\nabla_{\al}\tilde{N}^{\ga}=0$ and  $\rmN_{\beb}^{\alb}\nabla_{\alb}\tilde{N}^{\ga}=0$.
Then along $\Si$ we have
\begin{equation*}
\nabla_{\beb}\tilde{N}^{\ga} =-\II_{\beb}{^{\ga}}_{\de}N^{\de}+\rmN^{\ga}_{\de}\nabla_{\beb}\tilde{N}^{\de}
\end{equation*}
where $\II_{\beb}{^{\ga}}_{\de}:=\Pi_{\beb}^{\nub}\Pi_{\la}^{\ga}\II_{\nub}{^{\la}}_{\de}$, using that 
$\II_{\nub}{^{\la}}_{\de}N^{\de}=-\Pi_{\ga}^{\la}\nabla_{\nub}N^{\ga}$.
Thus we compute that
\begin{align*}
\Pi_{\mu}^{\al}\Pi_{\nub}^{\beb}\Pi_{\ga}^{\la}\nabla_{\al}\nabla_{\beb}\tilde{N}^{\ga}  
&
= \Pi_{\nub}^{\beb}\Pi_{\ga}^{\la}\left(-\nabla_{\mu}(\II_{\beb}{^{\ga}}_{\de}N^{\de})+(\nabla_{\mu}\rmN^{\ga}_{\de})\nabla_{\beb}N^{\de} \right) \\*
&
= -D_{\mu}(\II_{\nub}{^{\la}}_{\de}N^{\de})
\end{align*}
along $\Si$, where in the first step we used that $\Pi^{\la}_{\ga}\rmN^{\ga}_{\de}=0$ and in the second step we used \eqref{ProjectorDeriv1} to show that $\Pi^{\la}_{\ga}\nabla_{\mu}\rmN^{\ga}_{\de}=0$ and Proposition \ref{SubmanifoldTWprop}.
Now on the other hand (since $\II_{\mu\nub}{^{\gab}}=0$) we have
\begin{equation*}
\nabla_{\al}\tilde{N}^{\ga} = \rmN^{\ga}_{\de}\nabla_{\al}\tilde{N}^{\de}
\end{equation*}
along $\Si$, and this time we compute that
\begin{equation*}
\Pi_{\mu}^{\al}\Pi_{\nub}^{\beb}\Pi_{\ga}^{\la}\nabla_{\beb}\nabla_{\al}\tilde{N}^{\ga} 
 =-\II_{\nub}{^{\la}}_{\de}\nabla_{\mu}N^{\de}
 =-\II_{\nub}{^{\la}}_{\de}\nabla^{\perp}_{\mu}N^{\de}
\end{equation*}
since $\nabla_{\nub}\rmN^{\ga}_{\de}=-\II_{\nub}{^{\rho}}_{\de}\Pi^{\ga}_{\rho}$ by \eqref{ProjectorDeriv2}.
Putting these together we get 
\begin{equation*}
\Pi_{\mu}^{\al}\Pi_{\nub}^{\beb}\Pi_{\ga}^{\la}\left(\nabla_{\al}\nabla_{\beb}-\nabla_{\beb}\nabla_{\al}\right)\tilde{N}^{\ga}=-(D_{\mu}\II_{\nub}{^{\la}}_{\deb})N^{\deb}
\end{equation*}
along $\Si$. Since $\II_{0\mu}{^{\ga}}=0$ we have $\Pi^{\la}_{\ga}\nabla_0 N^{\ga}=0$ and hence from \eqref{PHcurvatureDef} we obtain
\begin{equation}
R_{\mu\nub}{^{\la}}_{\de}N^{\de}=(D_{\mu}\II_{\nub}{^{\la}}_{\deb})N^{\deb}.
\end{equation}
Noting that $R_{\mu\nub\lab\de}=-R_{\mu\nub\de\lab}$ then gives the result.
\end{proof}

\subsubsection{The pseudohermitian Ricci equation}\label{PseudohermitianRicciSect}
Given a compatible pair of contact forms we let $R^{\cN^{\alb}}$ denote the curvature of the normal Tanaka-Webster connection $\nabla^{\perp}$ on the antiholomorphic normal bundle $\cN^{\alb}$. With our conventions we have 
\begin{equation}\label{haNormalTWcurvature}
 \left( \nabla^{\perp}_{\mu}\nabla^{\perp}_{\nub}N^{\gab} -\nabla^{\perp}_{\nub}\nabla^{\perp}_{\mu}N^{\gab}+i\bh_{\mu\nub}\nabla_{0}^{\perp}N^{\gab} \right)
 = -R^{\cN^{\alb}}_{\mu\nub}{^{\gab}}_{\deb}N^{\deb}
\end{equation}
for any section $N^{\alb}$ of $\cN^{\alb}$, where we have coupled the normal Tanaka-Webster connection $\nabla^{\perp}$ with the submanifold Tanaka-Webster connection $D$.  The \emph{pseudohermitian Ricci equation} relates the component $R^{\cN^{\alb}}_{\mu\nub}{^{\gab}}_{\deb}$ of $R^{\cN^{\alb}}$ to the component $R_{\mu\nub}{^{\gab'}}_{\deb'}\rmN^{\gab}_{\gab'}\rmN^{\deb'}_{\deb}$ of the ambient pseudohermitian curvature tensor:
\begin{prop}\label{pseudohermitian-Ricci-equation-prop}
Given compatible contact forms, 
\begin{equation}\label{pseudohermitian-Ricci-equation}
R^{\cN^{\alb}}_{\mu\nub}{^{\gab}}_{\deb}=R_{\mu\nub}{^{\gab'}}_{\deb'}\rmN^{\gab}_{\gab'}\rmN^{\deb'}_{\deb} + 
\bh^{\la\rhob}\II_{\mu\la\deb}\II_{\nub\rhob}{^{\gab}}.
\end{equation}
\end{prop}
\begin{proof}
To facilitate calculation we couple the connection $\nabla^{\perp}$ with the submanifold Tanaka-Webster connection $D$; we also couple $\nabla$ with $D$.
If $N^{\gab}$ is a holomorphic normal field then
\begin{align*}
\nabla^{\perp}_{\mu}\nabla^{\perp}_{\nub}N^{\gab} &= \nabla^{\perp}_{\mu}(\rmN^{\gab}_{\deb}\nabla_{\nub}N^{\deb})\\*
&= \rmN^{\gab}_{\epb}\nabla_{\mu}(\rmN^{\epb}_{\deb}\nabla_{\nub}N^{\deb})\\*
&= \rmN^{\gab}_{\epb}\left( -\II_{\mu}{^{\lab}}_{\deb}\Pi^{\epb}_{\lab}\nabla_{\nub}N^{\deb} + \rmN^{\epb}_{\deb}\nabla_{\mu}\nabla_{\nub}N^{\deb}\right)\\*
&= \rmN^{\gab}_{\deb}\nabla_{\mu}\nabla_{\nub}N^{\deb}
\end{align*}
On the other hand, when we interchange the roles of $\mu$ and $\nu$ we obtain
\begin{align*}
\nabla^{\perp}_{\nub}\nabla^{\perp}_{\mu}N^{\gab} &=
\rmN^{\gab}_{\epb}\left( -\II_{\nub\lab}{^{\epb}}\Pi^{\lab}_{\deb}\nabla_{\mu}N^{\deb} + \rmN^{\epb}_{\deb}\nabla_{\nub}\nabla_{\mu}N^{\deb}\right)\\*
&= \II_{\nub\lab}{^{\gab}}\II_{\mu}{^{\lab}}_{\deb}N^{\deb}  + \rmN^{\gab}_{\deb}\nabla_{\nub}\nabla_{\mu}N^{\deb}
\end{align*}
Now observe that if one extends $N^{\gab}$ off $\Si$ such that $N^{\al}\nabla_{\al}\tilde{N}^{\gab}=0$ and  $N^{\alb}\nabla_{\alb}\tilde{N}^{\gab}=0$, then
\begin{equation*}
\Pi^{\al}_{\mu}\Pi^{\beb}_{\nub}\nabla_{\al}\nabla_{\beb}\tilde{N}^{\deb} = 
\nabla_{\mu}\nabla_{\nub}N^{\deb}
\quad \mathrm{and} \quad 
\Pi^{\al}_{\mu}\Pi^{\beb}_{\nub}\nabla_{\beb}\nabla_{\al}\tilde{N}^{\deb} = 
\nabla_{\nub}\nabla_{\mu}N^{\deb}.
\end{equation*}
Thus by \eqref{haNormalTWcurvature} and \eqref{PHcurvatureDef} (noting that $\nabla^{\perp}_0 N^{\gab}=\nabla_0 N^{\gab}$) one has the result.
\end{proof}
\begin{rem}\label{CurvatureSignRemark}
Since $\cN^{\alb}$ is a line bundle we may think of the curvature $R^{\cN^{\alb}}$ instead as a two form. By convention we take $R^{\cN^{\alb}}_{\mu\nub}$ to be $R^{\cN^{\alb}}_{\mu\nub}{^{\gab}}_{\gab}$, which means that $R^{\cN^{\alb}}$ is minus the usual curvature two form of the connection $\nabla^{\perp}$ on the line bundle $\cN^{\alb}$. With this convention we may write
\begin{equation}\label{pseudohermitian-Ricci-equation-minimal-codim}
R^{\cN^{\alb}}_{\mu\nub}=R_{\mu\nub N\bar{N}}+
\bh_{\ga\deb}\bh^{\la\rhob}\II_{\mu\la}{^{\ga}}\II_{\nub\rhob}{^{\deb}}
\end{equation}
where $R_{\mu\nub N\bar{N}}=R_{\mu\nub \ga\deb}N^{\ga}N^{\deb}$ for any weight $(1,0)$ unit normal field $N^{\al}$. Also, since $\nabla^{\perp}$ is Hermitian with respect to the Levi form of $\theta$ (on $\cN^{\alb}$), one has that $R^{\cN^{\alb}}=R^{\cN_{\al}}$ as two forms. Moreover, by duality one has that $R^{\cN_{\al}}=-R^{\cN^{\al}}$.
\end{rem}

\subsection{Relating Density Bundles}\label{Relating-Densities}
We have already been using compatible contact forms to identify the
density bundles $\cE(1,1)|_{\Si}$ and $\cE_{\Si}(1,1)$, and have
commented in passing that this identification does not in fact depend
on any choice of (compatible) contact forms. Let $\vsig$ be a positive real
element of $\cE(1,1)|_{\Si}$, then there is a unique real
element $\vsig_{\Si}$ in $\cE_{\Si}(1,1)$ such that $\vsig^{-1}\btheta$
pulls back to $\vsig_{\Si}\btheta_{\Si}$ under $\iota$. This correspondence 
induces an isomorphism of complex line bundles.
In this way we obtain canonical identifications between all diagonal
density bundles $\cE(w,w)|_{\Si}$ and $\cE_{\Si}(w,w)$. These identifications also agree with those induced by trivialising the ambient and intrinsic (diagonal) density bundles using an ambient
contact form $\theta$ and its pullback $\iota^{*}\theta$ respectively.

On the other hand it is not \emph{a priori} obvious whether one may canonically identify
the density bundles $\cE(1,0)|_{\Si}$ and $\cE_{\Si}(1,0)$, and
therefore identify all corresponding density bundles $\cE(w,w')|_{\Si}$
and $\cE_{\Si}(w,w')$. We require that any isomorphism of $\cE(1,0)|_{\Si}$
with $\cE_{\Si}(1,0)$ should be compatible with the identification
of $\cE(1,1)|_{\Si}$ with $\cE_{\Si}(1,1)$ already defined. Any
two such isomorphisms of $\cE(1,0)|_{\Si}$ with $\cE_{\Si}(1,0)$
are related by an automorphism of $\cE(1,0)|_{\Si}$ given by multiplication
by $e^{i\varphi}$ for some $\varphi\in C^{\infty}(\Si)$. This is
precisely the same freedom as in the choice of a unit holomorphic
conormal, in fact, we shall see below that these two choices are intrinsically
connected.

\subsubsection{Densities and holomorphic conormals}\label{DensitiesAndNormals}

Let $\Lambda_{\perp}^{1,0}\Si$ denote the
subbundle of $\Lambda^{1,0}M|_{\Sigma}$ consisting of all forms $N$
which vanish on the tangent space of $\Si$. The bundle $\Lambda_{\perp}^{1,0}\Si$ may be canonically identified with $\cN_{\al}$ by restriction to $T^{1,0}M$.
\begin{lem}\label{CanonicalBundlesRelatedLemma}
Along $\Si$ the ambient and submanifold canonical bundles are related by the canonical isomorphism
\begin{equation} \label{CanonicalBundlesRelated}
\scrK|_{\Si} \cong \scrK_{\Si} \otimes \Lambda_{\perp}^{1,0}\Si.
\end{equation}
\end{lem}
\begin{proof}
The map from $\Lambda^{n}(\Lambda^{1,0}\Si)\otimes \Lambda_{\perp}^{1,0}\Si$ to $\Lambda^{n+1}(\Lambda^{1,0}M|_{\Si})$ is given by
\begin{equation}
\zeta_{\Si} \otimes N \mapsto \eta\wedge N
\end{equation}
where $\eta$ is any element of $\Lambda^{n}(\Lambda^{1,0}M|_{\Si})$ with $\iota^*\eta=\zeta_{\Si}$.
\end{proof}
Given a section $\zeta_{\Si} \otimes N$ of $\scrK_{\Si} \otimes \Lambda_{\perp}^{1,0}\Si$ we write $\zeta_{\Si} \wedge N$ for the corresponding section of $\scrK|_{\Si}$.
The above lemma is the key to relating ambient and submanifold densities:
\begin{cor}\label{ambient-intrinsic-densities-cor}
The ambient and submanifold density bundles are related via the canonical isomorphism
\begin{equation}\label{ambient-intrinsic-densities}
\mathcal{E}(-n-1,0)|_{\Si}\cong\mathcal{E}_{\Sigma}(-n-1,0)\otimes \cN_{\al}(1,0).
\end{equation}
\end{cor}
\begin{proof}
By definition $\cE(-n-2,0)=\scrK$ and $\cE_{\Si}(-n-1,0)=\scrK_{\Si}$. Using this in \eqref{CanonicalBundlesRelated}, tensoring both sides with $\cE(1,0)|_{\Si}$, and identifying $\Lambda_{\perp}^{1,0}\Si$ with $\cN_{\al}$ gives the result.
\end{proof}
Thus any trivialisation of the line bundle $\cN_{\al}(1,0)$ gives an identification of the corresponding ambient and submanifold density bundles along $\Si$. One can check that if the trivialisation of $\cN_{\al}(1,0)$ is given by a unit holomorphic conormal then the resulting identification of density bundles will be compatible with the usual identification of $\cE(w,w)|_{\Si}$ with $\cE_{\Si}(w,w)$; this amounts to the claim that, given compatible contact forms $\theta$ and $\theta_{\Si}$, if $\zeta_{\Si}$ is a section of $\scrK_{\Si}$ volume normalised for $\theta_{\Si}$ and $N$ is a section of $\Lambda_{\perp}^{1,0}\Si$ which is normalised with respect to the Levi form of $\theta$ (i.e. satisfies $h^{\al\beb}N_{\al}\overline{N_{\be}}=1$) then the section $\zeta_{\Si}\wedge N$ of $\scrK|_{\Si}$ is volume normalised for $\theta$.
\begin{rem}
The preceding observation motivates the search for a canonical unit holomorphic conormal. One way to approach this search is to observe that for any unit holomorphic conormal $N_{\al}$ the field $\varpi_{\mub}:=N^{\al}\nabla_{\mub}N_{\al}=-N^{\alb}\nabla_{\mub}N_{\alb}$ does not depend on the choice of admissible ambient contact form used to define $\nabla_{\mub}$ and a calculation shows that $\varpi_{\mub}$ satisfies $\nabla_{[\mub}\varpi_{\nub]}=0$. In the case where one has local exactness of the tangential Cauchy-Riemann complex of $\Si$ at $(0,1)$-forms one can then (locally) define a canonical unit holomorphic conormal $N_{\al}$ for which $\varpi_{\mub}=\nabla_{\mub}f$ with $f$ a real valued function; the \emph{a priori} phase freedom in the unit normal is used to eliminate the imaginary part of $f$, leaving no further freedom. However, for smooth (rather than real analytic) embeddings the required local exactness may not hold, as was famously demonstrated by Lewy for the three dimensional Heisenberg group \cite{Lewy}. In the following it will become plain that we should keep $\cN_{\al}(1,0)$ in the picture, rather than trivialise it, and thus we have not pursued this direction further.
\end{rem}

\subsection{Relating Connections on Density Bundles}\label{RelatingConnectionsOnDensities}

Given an admissible ambient contact form $\theta$, the normal Tanaka-Webster connection $\nabla^{\perp}$ on $\cN_{\al}$ can equivalently be thought of as the connection on $\Lambda_{\perp}^{1,0}\Si$ defined by differentiating tangentially using the Tanaka-Webster connection $\nabla$ and then projecting using the Webster metric $g_{\theta}$.
\begin{lem} \label{ConnectionsOnCanonicalBundlesRelatedLemma}
Given any pair of compatible contact forms the isomorphism \eqref{CanonicalBundlesRelated} of Lemma \ref{CanonicalBundlesRelatedLemma} intertwines the respective Tanaka-Webster connections:
\begin{align*}
\scrK|_{\Si} &\: \cong \:\scrK_{\Si} \: \otimes \Lambda_{\perp}^{1,0}\Si \\*
\iota^*\nabla \hphantom{|} &\: \cong \:\:\: D \:\: \otimes \;\; \nabla^{\perp}.
\end{align*}
\end{lem}
\begin{proof}
Let $\zeta_{\Si}\otimes N$ be a section of $\scrK_{\Si}\otimes\Lambda^{1,0}_{\perp}\Si$. Let $\eta$ be any section of $\Lambda^{n}(\Lambda^{1,0}M|_{\Si})$ which pulls back to $\zeta_{\Si}$. Then $\zeta_{\Si}\wedge N := \eta \wedge N$. If $X\in T\Si$ then
\begin{equation*}
\nabla_X (\zeta_{\Si}\wedge N) = (\nabla_X \eta)\wedge N + \eta\wedge (\nabla_X N),
\end{equation*}
but $(\nabla_X \eta)\wedge N=(\Pi_{\Si}\nabla_X \eta)\wedge N$ which is the section of $\scrK|_{\Si}$ corresponding to $(D_X \zeta_{\Si})\otimes N$ (here $\Pi_{\Si}$ denotes submanifold tangential projection with respect to $g_{\theta}$), and $\eta\wedge (\nabla_X N)=\eta\wedge (\nabla^{\perp}_X N)$.
\end{proof}
Observing that the connection $\nabla^{\perp}$ on $\cN_{\al}(1,0)$ agrees with the coupling of  $\nabla^{\perp}$ on $\cN_{\al}$ with $\iota^*{\nabla}$ on $\cE(1,0)|_{\Si}$ we have the following corollary:
\begin{cor}\label{RelatingTWconnectionsOnDensities}
Given any pair of compatible contact forms the (canonical) isomorphism \eqref{ambient-intrinsic-densities} of Corollary \ref{ambient-intrinsic-densities-cor} intertwines the respective Tanaka-Webster connections:
\begin{align*}
\mathcal{E}(-n-1,0)|_{\Si} &\: \cong \:\mathcal{E}_{\Sigma}(-n-1,0) \otimes \cN_{\al}(1,0)\\*
\iota^*\nabla\hphantom{---} &\: \cong \:\hphantom{---} D \hphantom{--}\:\:\: \otimes \;\;\;\: \nabla^{\perp}.
\end{align*}
\end{cor}
This means that if we want to identify corresponding ambient and submanifold density bundles (along $\Si$) in such a way that the ambient and submanifold Tanaka-Webster connections of a pair of compatible contact forms agree (in the sense that $\iota^*\nabla=D$), then we must trivialise $\cN_{\al}(1,0)$ using a section which is parallel for the normal Tanaka-Webster connection $\nabla^{\perp}$. This is not a CR invariant condition on the section of $\cN_{\al}(1,0)$, and the following lemma shows that it is not possible to find a parallel section in general because of curvature:
\begin{lem}\label{WeightedNormalTWcurvatureLemma} 
Let $\theta$ and $\theta_{\Si}$ be compatible contact forms and let $R^{\cN_{\al}(1,0)}$ denote the curvature of $\nabla^{\perp}$ on the bundle $\cN_{\al}(1,0)$, then the $(1,1)$-component of $R^{\cN_{\al}(1,0)}|_{H_{\Si}}$ satisfies
\begin{equation}\label{WeightedNormalTWcurvature}
R^{\cN_{\al}(1,0)}_{\mu\nub}=\frac{n+1}{n+2}R_{\mu\nub}-r_{\mu\nub},
\end{equation}
where $R_{\mu\nub}=\Pi^{\al}_{\mu}\Pi^{\beb}_{\nub}R_{\al\beb}$.
\end{lem}
\begin{proof}
By Proposition \ref{TWDensityCurvature} the $(1,1)$-component of the restriction to $H$ of the curvature of the Tanaka-Webster connection on the line bundle $\cE(1,0)$ is  $\frac{1}{n+2}R_{\al\beb}$. Thus the $(1,1)$-component of the the restriction to $H_{\Si}$ of the curvature of $\iota^*\nabla$ on $\cE(1,0)|_{\Si}$ is $\frac{1}{n+2}R_{\mu\nub}$. Combining this with the Ricci equation \eqref{pseudohermitian-Ricci-equation-minimal-codim} for $R^{\cN^{\alb}}_{\mu\nub}=R^{\cN_{\al}}_{\mu\nub}$ we have
\begin{equation*}
R^{\cN_{\al}(1,0)}_{\mu\nub}=R_{\mu\nub N\bar{N}} +\bh_{\ga\deb}\II_{\mu\la}{^{\ga}}\II_{\nub}{^{\la\deb}}-\frac{1}{n+2}R_{\mu\nub}.
\end{equation*}
Using the once contracted Gauss equation
\begin{equation*}
R_{\mu\nub} - R_{\mu\nub N\bar{N}} = r_{\mu\nub} + \bh_{\ga\deb}\II_{\mu\la}{^{\ga}}\II_{\nub}{^{\la\deb}}
\end{equation*}
obtained from \eqref{pseudohermitian-Gauss} we have the result.
\end{proof}
\begin{rem}
Here, because of our conventions (cf.\ Remark \ref{CurvatureSignRemark}), we take $R^{\cN_{\al}(1,0)}$ to be minus the usual curvature of $\cN_{\al}(1,0)$ as a line bundle.
\end{rem}

\subsection{The Ratio Bundle of Densities}\label{RatioBundleSec}

The observations of Sections \ref{Relating-Densities} and \ref{RelatingConnectionsOnDensities} motivate us to look at the relationship between corresponding ambient and submanifold density bundles rather than seeking to identify them (along $\Si$). We therefore make the following definition:
\begin{defn}\label{ratio-bundle-defn}
The \emph{ratio bundle of densities} of weight $(w,w')$ is the complex line bundle
\begin{equation}\label{RatioBundleOfDensities}
\cR(w,w'):=\cE(w,w')|_{\Si}\otimes\cE_{\Sigma}(-w,-w')
\end{equation}
on the submanifold $\Si$. Equivalently $\cR(w,w')$ is the bundle whose sections are endomorphisms from $\cE_{\Sigma}(w,w')$ to $\cE(w,w')|_{\Si}$.
\end{defn}
Note that the bundles $\cR(w,w)$ are canonically trivial, and therefore $\cR(w,w')$ is canonically isomorphic to $\cR(w-w',0)$. Also by definition $\cR(-n-1,0)$ is canonically isomorphic to $\cN_{\al}(1,0)$, and we make this into an identification
\begin{equation}
\cR(-n-1,0)=\cN_{\al}(1,0).
\end{equation}

\subsection{The Canonical Connection on the Ratio Bundles} \label{CanonicalConnectionOnRatio}

Borrowing insight from Section \ref{Submanifolds-and-Tractors} below we observe that the bundle $\cN_{\al}(1,0)$ carries a natural CR invariant connection, which induces connections on the density ratio bundles $\cR(w,w')$. The reason is that $\cN_{\al}(1,0)$ is canonically isomorphic to a subbundle $\cN_A$ of the ambient cotractor bundle $\cE_A$ along $\Si$ which has an invariant connection induced by the ambient tractor connection (Proposition \ref{NormalTractorConnectionProp}). We denote this canonical invariant connection on $\cR(w,w')$ by $\nabla^{\cR}$. It turns out to be very naturally expressed in terms of Weyl connections (recall Section \ref{WeylConnections}). Hence we make the following definition:
\begin{defn} 
Given an admissible ambient contact form $\theta$, the \emph{normal Weyl connection} $\nabla^{W,\perp}$ on $\cN^{\al}(w,w')$ is the connection induced by $\nabla^{W}$ (projecting tangential derivatives of sections back into $\cN^{\al}$ using the Levi form). Dually, the connection $\nabla^{W,\perp}$ acts on $\cN_{\al}(-w,-w')$.
\end{defn}
For calculational purposes we will need the following lemma:
\begin{lem}\label{CanonicalNormalConnectionLemma}
Given an admissible ambient contact form $\theta$ the connection $\nabla^{W,\perp}$ on $\cN_{\al}(1,0)$ acts on a section $\tau_{\al}$ by
\begin{equation}
\nabla_{\mu}^{W,\perp}\tau_{\al}=\nabla^{\perp}_{\mu}\tau_{\al}, \quad 
\nabla_{\mub}^{W,\perp}\tau_{\al}=\nabla^{\perp}_{\mub}\tau_{\al}
\end{equation}
and
\begin{equation}\label{0NormalWeylConnection}
\nabla_{0}^{W,\perp}\tau_{\al}=\nabla^{\perp}_{0}\tau_{\al}-iP_{N\bar{N}}\tau_{\al}+\frac{i}{n+2}P\tau_{\al}
\end{equation}
where $P_{N\bar{N}}=P_{\al\beb}N^{\al}N^{\beb}$ for any (weighted) unit holomorphic normal $N^{\al}$ and $P=P_{\al}{^{\al}}$.
\end{lem}
\begin{proof}
This follows immediately from the definitions of the Weyl and normal Weyl connections and the formula \eqref{0WeylConnectionOnDensities}.
\end{proof}
The connection $\nabla^{\cR}$ on $\cR(-n-1,0)=\cN_{\al}(1,0)$ turns out to agree precisely with the normal Weyl connection of \emph{any} admissible contact form. In particular the normal Weyl connection $\nabla^{W,\perp}$ on the bundle $\cN_{\al}(1,0)$ does not depend on the choice of admissible ambient contact form. This follows from Proposition \ref{NormalTractorConnectionProp} below, but here we give a direct proof. Before we prove this we make an important technical observation, stated in the following lemma:
\begin{lem}\label{UpsilonTangentialLemma}
Let $\theta$ be an admissible ambient contact form. The contact form $\thetah=e^{\Ups}\theta$ is admissible if and only if
\begin{equation}
\Ups_{\al}=\Pi_{\al}^{\be}\Ups_{\be}
\end{equation}
along $\Si$.
\end{lem}
\begin{proof}
This follows immediately from the transformation law for the Reeb vector field given in Lemma \ref{ReebTransform} since both $T$ and $\hat{T}$ must be tangent to $\Si$.
\end{proof}
\begin{prop}\label{NormalWeylConnectionInvariant}
The normal Weyl connection $\nabla^{W,\perp}$ on the bundle $\cN_{\al}(1,0)$ does not depend on the choice of admissible contact form.
\end{prop}
\begin{proof}
Fix a pair of compatible contact forms $\theta$, $\theta_{\Si}$ and suppose $\thetah=e^{\Ups}\theta$ is any other admissible ambient contact form. Let $\tau_{\al}$ be a section of $\cN_{\al}(1,0)$. Extend $\tau_{\al}$ arbitrarily off $\Si$. When differentiating in contact directions the connections $\nabla^{W,\perp}$ and $\nabla^{\perp}$ agree, so from \eqref{TWhhTransform} and Proposition \ref{TWonDensitiesProp} we have
\begin{align*}
\nablah^{W,\perp}_{\mu} \tau_{\vphantom{\hat{T}}\de} & = \rmN^{\be}_{\de}\Pi^{\al}_{\mu}\nablah_{\al} \tau_{\be} \\
& = \rmN^{\be}_{\de}\Pi^{\al}_{\mu}(\nabla_{\al}\tau_{\be} -\Ups_{\be}\tau_{\al} -\Ups_{\al}\tau_{\be} + \Ups_{\al}\tau_{\be})\\
& = \rmN^{\be}_{\de}\Pi^{\al}_{\mu}\nabla_{\al} \tau_{\be}
\end{align*}
since $\Pi^{\al}_{\mu}\tau_{\al}=0$ (note that $\rmN^{\be}_{\ga}\Ups_{\be}$ also vanishes since $\theta$ and $\thetah$ are admissible). Similarly, from \eqref{TWahTransform} and Proposition \ref{TWonDensitiesProp} we have
\begin{align*}
\nablah^{W,\perp}_{\mub} \tau_{\vphantom{\hat{T}}\de} & = \rmN^{\be}_{\de}\Pi^{\alb}_{\mub}\nablah_{\alb} \tau_{\be} \\
& = \rmN^{\be}_{\de}\Pi^{\alb}_{\mub}(\nabla_{\alb}\tau_{\be} +\bh_{\be\alb}\Ups^{\ga}\tau_{\ga})\\
& = \rmN^{\be}_{\de}\Pi^{\alb}_{\mub}\nabla_{\alb} \tau_{\be}
\end{align*}
since $\rmN^{\be}_{\de}\Pi^{\al}_{\mu}\bh_{\be\alb}=0$. 

The operators $\nabla^W_0$ and $\nabla_0$ acting on $\tau_{\al}$ are related by
\begin{equation*}
\nabla^W_0 \tau_{\al} = \nabla_0 \tau_{\al} -iP_{\al}{^{\be}}\tau_{\be}+\frac{i}{n+2}P\tau_{\al}.
\end{equation*}
Now, on the one hand, by \eqref{TW0hTransform} and Proposition \ref{TWonDensitiesProp}, noting that $\rmN^{\be}_{\de}\Ups_{\be}=0$ by Lemma \ref{UpsilonTangentialLemma}, we have
\begin{align*}
\nablah^{\perp}_0 \tau_{\de} & = \rmN^{\be}_{\de} [ \nabla_0 \tau_{\be} + i\Ups^{\gab}\nabla_{\gab}\tau_{\be} - i\Ups^{\ga}\nabla_{\ga}\tau_{\be} -i\Ups^{\ga}{_{\be}}\tau_{\ga} \\*
& \qquad \quad  + \tfrac{1}{n+2}(\Ups_0 +i\Ups^{\ga}{_{\ga}}-i\Ups^{\ga}\Ups_{\ga})\tau_{\be} ].
\end{align*}
On the other hand from \eqref{SchoutenTransformLambda} and \eqref{SchoutenTransform}, noting that  $\Ups_{\al\beb}+\Ups_{\beb\al}=2\Ups_{\beb\al}-i\bh_{\al\beb}\Ups_0$ by \eqref{torsion1}, we have
\begin{align*}
\rmN^{\be}_{\de}[i\hat{P}_{\be}{^{\ga}}\tau_{\ga}-\tfrac{i}{n+2}\hat{P}\tau_{\be}] 
& = 
\rmN^{\be}_{\de}[i(P_{\be}{^{\ga}}  -\Ups^{\ga}{_{\be}} +\tfrac{i}{2}\Ups_0\delta^{\ga}_{\be} -\tfrac{1}{2}\Ups^{\ep}\Ups_{\ep}\delta^{\ga}_{\be})\tau_{\ga}
\\*
&\qquad \quad - \tfrac{i}{n+2}(P  -\Ups^{\ga}{_{\ga}} +\tfrac{i n}{2}\Ups_0 -\tfrac{n}{2}\Ups^{\ep}\Ups_{\ep})\tau_{\be}].
\end{align*}
Since $\tfrac{1}{2}-\tfrac{n}{2(n+2)}=\tfrac{1}{n+2}$ we obtain that
\begin{equation*}
\nablah^{W,\perp}_0 \tau_{\de} = \nabla^{W,\perp}_0 \tau_{\be} + i\Ups^{\mub}\nabla^{W,\perp}_{\mub}\tau_{\be} - i\Ups^{\mu}\nabla^{W,\perp}_{\mu}\tau_{\be},
\end{equation*}
as required (recall that the `0-direction' has a different meaning on the left and right hand sides of the above display, cf.\ Lemma \ref{0ComponentTransform}).
\end{proof}
\begin{rem}
Both Lemma \ref{UpsilonTangentialLemma} and Proposition \ref{NormalWeylConnectionInvariant} hold in the general codimension case with the same proof (as does Proposition \ref{NormalTractorConnectionProp}).
\end{rem}
We therefore take $\nabla^{\cR}$ to be the connection induced on the ratio bundles by the normal Weyl connection of an admissible contact form on $\cN_{\al}(1,0)$, and give later in Proposition \ref{NormalTractorConnectionProp} of Section \ref{sec:NormalTrac} the tractor explanation for this invariant connection. In order to compute with $\nabla^{\cR}$ we will need the following lemma:
\begin{lem}\label{RatioConnectionLemma}
In terms of a compatible pair of contact forms, $\theta$, $\theta_{\Si}$, the connection $\nabla^{\cR}$ on a section $\phi\otimes\si$ of $\cE(w,w')|_{\Si}\otimes\cE_{\Sigma}(-w,-w')$ is given by 
\begin{equation}
\nabla^{\cR}_{\mu}(\phi\otimes\si)=(\nabla_{\mu}\phi)\otimes\si+\phi\otimes(D_{\mu}\si),
\end{equation}
\begin{equation}
\nabla^{\cR}_{\mub}(\phi\otimes\si)=(\nabla_{\mub}\phi)\otimes\si+\phi\otimes(D_{\mub}\si),
\end{equation}
and
\begin{equation}
\nabla^{\cR}_{0}(\phi\otimes\si)=(\nabla_{0}\phi)\otimes\si+\phi\otimes(D_{0}\si) + \tfrac{w-w'}{n+1}(iP_{N\bar{N}}-\tfrac{i}{n+2}P)\phi\otimes\si.
\end{equation}
\end{lem}
\begin{proof}
This follows from Lemma \ref{CanonicalNormalConnectionLemma} combined with Corollary \ref{RelatingTWconnectionsOnDensities}.
\end{proof}
\begin{cor}
The connection $\nabla^{\cR}$ on the diagonal bundles $\cR(w,w)$ is flat and agrees with the exterior derivative of sections in the canonical trivialisation.
\end{cor}
\begin{proof}
This follows from Lemma \ref{RatioConnectionLemma} combined with Lemma \ref{TWonDiagonalDensitiesLemma}.
\end{proof}
\begin{rem}
By coupling with the connection $\nabla^{\cR}$ we can invariantly convert connections (and hence other operators) acting on intrinsic densities to ones on ambient densities, and vice versa. This will allow us to relate the intrinsic and ambient tractor connections, their difference giving rise to the basic CR invariants of the embedding.
\end{rem}

\subsubsection{The curvature of the canonical ratio bundle connection} \label{CanonicalConnectionOnRatioCurvature}
We shall see that the connection $\nabla^{\cR}$ is not flat in general, making it unnatural to identify the ambient and submanifold density bundles along $\Si$. 

Let $\ka^{\cR(w,w')}$ denote the curvature of $\nabla^{\cR}$ on the line bundle $\cR(w,w')$, and let $R^{\cN^*}$ denote the curvature of $\nabla^{W,\perp}$ on $\cN_{\al}(1,0)$ for any admissible contact form $\theta$. By convention $R^{\cN^*}$ has the opposite sign to the usual line bundle curvature $\ka^{\cR(-n-1,0)}$. Clearly the curvatures $\ka^{\cR(w,w')}$ are determined by $R^{\cN^*}$, in particular
\begin{equation*}
\ka^{\cR(1,0)}=\frac{1}{n+1}R^{\cN^*}.
\end{equation*}
Here we give expressions for the components of $R^{\cN^*}$. Note that the components of the restriction $R^{\cN^*}|_{H_{\Si}}$ must be invariant.
\begin{prop}\label{haCurvatureOfRatioBundleProp}
The $(1,1)$-part of $R^{\cN^*}|_{H_{\Si}}$ satisfies
\begin{equation}
R^{\cN^*}_{\mu\nub} = (n+1)(P_{\mu\nub}-p_{\mu\nub}) 
+ (P- P_{N\bar{N}} - p)\bh_{\mu\nub},
\end{equation}
where $P_{N\bar{N}}=P_{\al\beb}N^{\al}N^{\beb}$ for any (weighted) unit holomorphic normal $N^{\al}$, $P=P_{\al}{^{\al}}$, and $p=p_{\mu}{^{\mu}}$.
\end{prop}
\begin{proof}
Recall that $\nabla^{\cR}=\nabla^{W,\perp}$ on $\cR(-n-1,0)=\cN_{\al}(1,0)$ for any admissible ambient contact form $\theta$. Fixing $\theta$ admissible we have
\begin{align*}
-R^{\cN^*}_{\mu\nub} \tau_{\al}& = \left(
\nabla_{\mu}^{W,\perp}\nabla_{\nub}^{W,\perp}-\nabla_{\nub}^{W,\perp}\nabla_{\mu}^{W,\perp}
+ i\bh_{\mu\nub}\nabla_{0}^{W,\perp} \right) \tau_{\al} \\
& = \left( \nabla_{\mu}^{\perp}\nabla_{\nub}^{\perp}-\nabla_{\nub}^{\perp}\nabla_{\mu}^{\perp} 
+ i\bh_{\mu\nub}\nabla_{0}^{\perp} \right) \tau_{\al} \\
& \qquad + ( P_{N\bar{N}} -\frac{1}{n+2}P )\bh_{\mu\nub}\tau_{\al}
\end{align*}
for any section $\tau_{\al}$ of $\cN_{\al}(1,0)$, using Lemma \ref{CanonicalNormalConnectionLemma}. Thus from \eqref{WeightedNormalTWcurvature} of Lemma \ref{WeightedNormalTWcurvatureLemma} we have that
\begin{equation*}
R^{\cN^*}_{\mu\nub} = \frac{n+1}{n+2}R_{\mu\nub}-r_{\mu\nub} 
- ( P_{N\bar{N}} -\frac{1}{n+2}P )\bh_{\mu\nub}.
\end{equation*}
Now using that $R_{\al\beb}=(n+2)P_{\al\beb} + P\bh_{\al\beb}$, from the definition of $P_{\al\beb}$, and using the corresponding expression for $r_{\mu\nub}$, we have the result.
\end{proof}
Note that $P- P_{N\bar{N}} - p$ is the trace of $P_{\mu\nub}-p_{\mu\nub}$, with respect to $\bh_{\mu\nub}$. The following lemma therefore manifests the CR invariance of $R^{\cN^*}_{\mu\nub}$.
\begin{lem}\label{SchoutenDifferenceLemma}
Given any pair of compatible contact forms, the difference $P_{\mu\nub}-p_{\mu\nub}$ satisfies
\begin{align}\label{SchoutenDifferenceInvariantExpression}
P_{\mu\nub}-p_{\mu\nub} &= \frac{1}{n+1}(S_{\mu\nub N\bar{N}}+\tfrac{1}{2n}S_{N\bar{N} N\bar{N}} \bh_{\mu\nub})\\*
\nonumber &\quad +\frac{1}{n+1}(\II_{\mu\la\gab}\II_{\nub}{^{\la\gab}}+\tfrac{1}{2n}\II_{\rho\la\gab}\II^{\rho\la\gab} \bh_{\mu\nub}),
\end{align}
where $S_{\mu\nub N\bar{N}}=\Pi^{\al}_{\mu}\Pi^{\beb}_{\nub}S_{\al\beb\ga\deb}N^{\ga}N^{\deb}$ and $S_{N\bar{N} N\bar{N}}=S_{\al\beb\ga\deb}N^{\al}N^{\beb}N^{\ga}N^{\deb}$ for any (weighted) unit holomorphic normal $N^{\al}$.
\end{lem}
\begin{proof}
Taking the trace free part of the Gauss equation \eqref{pseudohermitian-Gauss} one has
\begin{equation*}
\tfrac{1}{n+1}(S_{\mu\nub\la}{^{\la}}-\tfrac{1}{n}S_{\rho}{^{\rho}}_{\la}{^{\la}}\bh_{\mu\nub})+P_{\mu\nub} = p_{\mu\nub} +
\tfrac{1}{n+1}(\II_{\mu\la\gab}\II_{\nub}{^{\la\gab}}+\tfrac{1}{2n}\II_{\rho\la\gab}\II^{\rho\la\gab} \bh_{\mu\nub})
\end{equation*}
and noting that $S_{\mu\nub\la}{^{\la}}=S_{\mu\nub\ga\deb}(\bh^{\ga\deb}-N^{\ga}N^{\deb})=-S_{\mu\nub N\bar{N}}$ and similarly that $S_{\rho}{^{\rho}}_{\la}{^{\la}}=S_{N\bar{N} N\bar{N}}$ one has the result.
\end{proof}
\begin{rem} The difference $P_{\mu\nub}-p_{\mu\nub}$ is the CR analogue of the so called Fialkow tensor \cite{CurryGover-Conformal,Vyatkin-Thesis} in conformal submanifold geometry, though here it is showing up in a completely new role. 
\end{rem}
\begin{prop}
We have
\begin{equation}
R^{\cN^*}_{\mu\nu}=0 \quad \mathrm{and} \quad R^{\cN^*}_{\mub\nub}=0.
\end{equation}
\end{prop}
\begin{proof}
By a straightforward calculation along the lines of the proof of Proposition \ref{pseudohermitian-Ricci-equation-prop} we have, given compatible contact forms, that
\begin{equation*}
\nabla^{\perp}_{\mu}\nabla^{\perp}_{\nu}N_{\al} -\nabla^{\perp}_{\nu}\nabla^{\perp}_{\mu}N_{\al}=
N_{\al}N^{\be}(\nabla_{\mu}\nabla_{\nu}N_{\be} -\nabla_{\nu}\nabla_{\mu}N_{\be})
\end{equation*}
for any unit holomorphic conormal field $N_{\al}$ (where both $\iota^*\nabla$ and $\nabla^{\perp}$ are coupled with the submanifold Tanaka-Webster connection $D$). Noting that $\nabla_{\mu}\nabla_{\nu}=\nabla_{\nu}\nabla_{\mu}$ on densities by Proposition \ref{TWDensityCurvature}, we get that $-R^{\cN^*}_{\mu\nu}=\Pi^{\al}_{\mu}\Pi^{\be}_{\nu}R_{\al\be}{^{\ga}}_{\de}N_{\ga}N^{\de}$; this is zero by \eqref{BianchiSymm3}, noting that $R_{\al\be\gab\de}=-R_{\al\be\de\gab}$. In a similar manner one shows that $R^{\cN^*}_{\mub\nub}$ also vanishes.
\end{proof}

Given compatible contact forms, one also has the component $R^{\cN^*}_{\mu0}$. By a similar but more tedious calculation one arrives at the expression
\begin{equation}\label{h0NormalCurvature}
R^{\cN^*}_{\mu0} = V_{\mu\bar{N}N}-iP_{N}{^{\vphantom{g}\nu}}\II_{\mu\nu}
\end{equation}
where $V_{\mu\bar{N}N}=\Pi^{\al}_{\mu}V_{\al\beb\ga}N^{\beb}N^{\ga}$ (with $V_{\al\beb\ga}$ as in \eqref{Vtensor}), $P_{N}{^{\nu}}=P_{\ga}{^{\be}}N^{\ga}\Pi^{\nu}_{\be}$ and $\II_{\mu\nu}=\II_{\mu\nu}{^{\ga}}N_{\ga}$ for any (weighted) unit holomorphic normal field $N^{\al}$. One can obtain this expression more easily using the description of the canonical connection on $\cN_{\al}(1,0)$ in terms of the ambient tractor connection given below.

\section{CR Embedded Submanifolds and Tractors}\label{Submanifolds-and-Tractors}

Here we continue to work in the setting where $\iota:\Si\hrarrow M$
is a CR embedding of a hypersurface type CR manifold $(\Si^{2m+1},H_{\Si},J_{\Si})$
into a strictly pseudoconvex CR manifold $(M^{2n+1},H,J)$ with $m=n-1$.
We adopt the notation $\cT M$ rather than $\cT$ for the
standard tractor bundle of $M$, and write
$\cT\Si$ for the standard tractor bundle of $\Si$. Similarly we
will denote the adjoint tractor bundles of $M$ and $\Si$ by $\cA M$
and $\cA\Si$ respectively. We will also use the abstract index notation $\cE^{I}$ for $\cT \Si$
and allow the use of indices $I$, $J$, $K$, $L$, $I'$, and so
on.

\subsection{Normal Tractors}\label{sec:NormalTrac}
Given any unit section $N_{\al}$ of $\cN_{\al}(1,0)$ we define the corresponding \emph{(unit) normal (co)tractor}
$N_{A}$ to be the section of $\mathcal{E}_{A}|_{\Sigma}$, the ambient tractor bundle
restricted to fibers over $\Si$, given by
\begin{equation}
N_{A}\overset{\theta}{=}\left(\begin{array}{c}
0\\
N_{\al}\\
-\sfH
\end{array}\right)
\end{equation}
where $\sfH=\frac{1}{n-1}\bh^{\mu\nub}\Pi_{\mu}^{\al}\nabla_{\nub}N_{\al}$
and $\nabla_{\nub}$ denotes the Tanaka-Webster connection of $\theta$
acting in tangential antiholomorphic directions along $\Si$; the tractor field $N_{A}$ does not depend on the choice of ambient contact form $\theta$ since from \eqref{TWahTransform} of Proposition \ref{TWtransform} combined with Proposition \ref{TWonDensitiesProp} we have that
\begin{equation*}
\hat{\sfH}=\sfH+\Ups^{\al}N_{\al}
\end{equation*}
when $\thetah=e^{\Ups}\theta$ (with $\Ups^{\al}=\nabla^{\al}\Ups$), as required by \eqref{TractorTransformation}.
If $\theta$ is admissible for the submanifold $\Si$ then $\sfH=0$ (since $\II_{\nub\mu}{^{\ga}}=0$) and
\begin{equation}
N_{A}\overset{\theta}{=}\left(\begin{array}{c}
0\\
N_{\alpha}\\
0
\end{array}\right).
\end{equation}
\begin{rem}
The normal tractor $N_{A}$ associated to a unit holomorphic conormal $N_{\al}$ is an analogue of the
normal tractor associated to a weighted unit (co)normal field in conformal
hypersurface geometry defined first in \cite{BaileyEastwoodGover-Thomas'sStrBundle}.
\end{rem}
\begin{defn}
The \emph{normal cotractor bundle} $\cN_{A}$ is the subbundle of $\cE_{A}|_{\Si}$, the ambient cotractor bundle along $\Si$, spanned by the normal tractor $N_{A}$ given any unit holomorphic conormal field $N_{\al}$. The \emph{normal tractor bundle} $\cN^{A}$ is the dual line subbundle of $\cE^{A}|_{\Si}$ spanned by $N^{A}=h^{A\bar{B}}\overline{N_{B}}$. We alternatively denote $\cN^{A}$ and $\cN_{A}$ by $\cN$ and $\cN^*$ respectively.
\end{defn}
Since the ambient tractor bundle carries a parallel Hermitian bundle metric the ambient tractor connection induces a connection $\nabla^{\cN}$ on the non-null subbundle $\cN_{A}$ of $\cE_{A}|_{\Si}$. Explicitly, if $\rmN^A_B$ is the orthogonal projection from $\cE_A|_{\Si}$ onto $\cN_B$ then we have
\begin{equation}
\nabla^{\cN}_i v_B = \rmN^A_B \nabla_i v_A
\end{equation}
for any section $v_B$ of $\cN_B$, where $\nabla_i$ is the ambient standard tractor connection (pulled back via $\iota$). We can now explain the origin of the canonical connection on $\cN_{\al}(1,0)$.
\begin{prop}\label{NormalTractorConnectionProp}
The weighted conormal bundle $\cN_{\al}(1,0)$ is canonically isomorphic to the normal cotractor bundle $\cN_{A}$ via the map
\begin{equation}
\tau_{\al}\mapsto \tau_A \overset{\theta}{=}\left(\begin{array}{c}
0\\
\tau_{\alpha}\\
0
\end{array}\right)
\end{equation}
defined with respect to any admissible ambient contact form $\theta$. Moreover, the above isomorphism intertwines the normal tractor connection $\nabla^{\cN}$ on $\cN_{A}$ with the normal Weyl connection on $\cN_{\al}(1,0)$ of any admissible $\theta$.
\end{prop}
\begin{proof} The first part follows from the fact that if $\theta$ is admissible then $\thetah=e^{\Ups}\theta$ is admissible if and only if $\Ups_{\al}N^{\al}=0$, where $N^{\al}$ is a holomorphic normal field (a consequence of Lemma \ref{ReebTransform}). The second part follows from the explicit formulae for the tractor connection given in Section \ref{The-Tractor-Connection} (noting in particular \eqref{0hTractorConnectionWeyl}) and the observation that the orthogonal projection $\cE_A|_{\Si}\rightarrow \cN_A$ is given, with respect to any admissible ambient contact form, by
\begin{equation}
\left(\begin{array}{c}
\si\\
\tau_{\alpha}\\
\rho
\end{array}\right)
\mapsto \left(\begin{array}{c}
0\\
\rmN^{\be}_{\al}\tau_{\be}\\
0
\end{array}\right).
\end{equation}
\end{proof}
\begin{rem}
Clearly the isomorphism of Proposition \ref{NormalTractorConnectionProp} is Hermitian; in particular if $N_A$ is the normal tractor corresponding to a unit normal field $N_{\al}$ then
\begin{equation*}
N^A N_A = N^{\al}N_{\al}=1,
\end{equation*}
so $N_A$ is indeed a unit normal tractor. Although a unit normal tractor is determined only up to phase, the tractors
\begin{equation*}
N_{A}N_{\bar{B}}\quad\mbox{and}\quad N^{A}N_{B},
\end{equation*}
are independent of the choice of unit length section of $\cN_{A}$.
Indeed, $N^{A}N_{B}=\rmN^A_B$ and the section
\begin{equation*}
\Pi_{B}^{A}=\delta_{B}^{A}-N^{A}N_{B}
\end{equation*}
projects orthogonally from $\mathcal{E}_{A}|_{\Sigma}$ onto the orthogonal
complement $\cN_{A}^{\perp}$ of $\cN_{A}$ in $\mathcal{E}_{A}|_{\Sigma}$.
\end{rem}

\subsection{Tractor Bundles and Densities} \label{TractorsAndDensities}
Clearly $\cN_{A}^{\perp}$ has the same rank as $\cE_{I}$; they also have the same rank subbundles in their canonical filtration structures. Moreover, both $\cN_{A}^{\perp}$ and $\cE_{I}$ carry canonical Hermitian bundle metrics (and Hermitian connections). On the other hand we note that
for $\cN_{A}^{\perp}$ we have the canonical map
\begin{equation*}
\begin{array}{ccc}
\cN_{A}^{\perp} & \rarrow & \cE(1,0)|_{\Si}\\
v_{A} & \mapsto & Z^{A}v_{A}
\end{array}
\end{equation*}
where $Z^{A}$ is the ambient canonical tractor, whereas for $\cE_{I}$ we have the canonical map
\begin{equation*}
\begin{array}{ccc}
\cE_{I} & \rarrow & \cE_{\Si}(1,0)\\
v_{I} & \mapsto & Z^{I}v_{I}
\end{array}
\end{equation*}
where $Z^{I}$ is the canonical tractor of $\Si$. It seems natural
that we should look to identify these bundles (canonically), but doing
so clearly also involves identifying the density bundles $\cE(1,0)|_{\Si}$
and $\cE_{\Si}(1,0)$ (also canonically). The following lemma shows us that this is the only thing stopping us from identifying $\cE_{I}$ with $\cN_{A}^{\perp}$:
\begin{lem}\label{StandardTractorMapsLemma}
Fix a local isomorphism $\psi:\cE_{\Si}(1,0)\rarrow\cE(1,0)|_{\Si}$ (compatible
with the canonical identification of $\cE_{\Si}(1,1)$ with $\cE(1,1)|_{\Si}$)
and identify all corresponding density bundles $\cE_{\Si}(w,w')$
and $\cE(w,w')|_{\Si}$ using $\psi$. Then locally there is a canonically
induced map from $\cE_{I}$ to $\cN_{A}^{\perp}$, given with respect
to any pair $\theta$, $\theta_{\Si}$ of compatible contact forms
by
\begin{equation}
v_{I}\overset{\theta_{\Si}}{=}\left(\begin{array}{c}
\sigma\\
\tau_{\mu}\\
\rho
\end{array}\right)\mapsto v_{A}\overset{\theta}{=}\left(\begin{array}{c}
\sigma\\
\tau_{\al}\\
\rho
\end{array}\right)
\end{equation}
where $\tau_{\al}=\Pi_{\al}^{\mu}\tau_{\mu}$, which is a filtration
preserving isomorphism of Hermitian vector bundles.
\end{lem}
\begin{proof}
Let us start by fixing $\theta$ and $\theta_{\Si}$. That the map
described above is a filtration preserving bundle isomorphism is obvious.
That the map pulls back the Hermitian bundle metric of $\cN_{A}^{\perp}$
to that of $\cE_{I}$ is also obvious. It remains to show that the map is independent of the choice of compatible contact forms. To see this we suppose that $\thetah=e^{\Ups}\theta$ is any other admissible contact form and let $\thetah_{\Si}=\iota^*\thetah=e^{\Ups}\theta_{\Si}$. We need to compare the submanifold and ambient versions of the tractor transformation law \eqref{TractorTransformation}. By the compatibility of $\theta$ and $\theta_{\Si}$ we have $\nabla_{0}\Ups=D_{0}\Ups$ along $\Si$, and by Lemma \ref{UpsilonTangentialLemma} we also have $\Ups_{\al}=\Pi_{\al}^{\mu}\Ups_{\mu}$ where $\Ups_{\mu}=D_{\mu}\Ups$. These observations ensure that the map is well defined.
\end{proof}
The local bundle isomorphism $\psi:\cE_{\Si}(1,0)\rarrow\cE(1,0)|_{\Si}$ in the above lemma can also be thought of as a nonvanishing local section (or local trivialisation) of the ratio bundle $\cR(1,0)$. The bundle $\cR(1,1)$ is canonically trivial because of the canonical isomorphism of $\cE_{\Si}(1,1)$ with $\cE(1,1)|_{\Si}$, so that $\cR(1,0)$ carries a natural Hermitian bundle metric (i.e. is a $\mathrm{U}(1)$-bundle) and the compatibility of $\psi$ with the identification $\cE_{\Si}(1,1)=\cE(1,1)|_{\Si}$ is equivalent to $\psi$ being a unit section of $\cR(1,0)$. The ratio bundle $\cR(1,0)$ will prove to be the key to relating the tractor bundles (globally) without making an unnatural (local) identification of density bundles.

\subsection{Relating Tractor Bundles}\label{RelatingTractorBundles}

If we tensor $\cE_I$ with $\cE_{\Si}(0,1)$ then choosing a submanifold contact form $\theta_{\Si}$ identifies this bundle with
\begin{equation*}
[\cE_I]_{\theta_{\Si}} \otimes\cE_{\Si}(0,1) = \cE_{\Si}(1,1)\oplus\cE_{\mu}(1,1)\oplus\cE_{\Si}(0,0)
\end{equation*}
where $\cE_{\Si}(0,0)$ is the trivial bundle $\Sigma\times\mathbb{C}$. Similarly, given an ambient contact form $\theta$ we may identify the $\cN_A^{\perp}\otimes\cE(0,1)|_{\Si}$ with
\begin{equation*}
[\cN_A^{\perp}]_{\theta}\otimes\cE(0,1)|_{\Si}=\cE(1,1)|_{\Si}\oplus\cN^{\perp}_{\al}(1,1)\oplus\cE(0,0)|_{\Si}
\end{equation*}
where $\cE(0,0)$ is the trivial bundle $M\times \bbC$ and $\cN_{\al}^{\perp}$ denotes the orthogonal complement to $\cN_{\al}$ in $\cE_{\al}|_{\Si}$. Since $\cE_{\Si}(w,w)$ is canonically identified with $\cE(w,w)|_{\Si}$ we have the following theorem:
\begin{thm}\label{DiagonalTractorMap}
There is a canonical filtration preserving bundle isomorphism
\begin{equation*}\label{DiagonalTractorMapEqn}
\Pi^I_A:\cE_I \otimes\cE_{\Si}(0,1)\rightarrow \cN_A^{\perp}\otimes\cE(0,1)|_{\Si}
\end{equation*}
given with respect to a pair of compatible contact forms $\theta$, $\theta_{\Si}$ by
\begin{equation*}
[\cE_I]_{\theta_{\Si}} \otimes\cE_{\Si}(0,1) \ni \left(\begin{array}{c}
\sigma\\
\tau_{\mu}\\
\rho
\end{array}\right)\mapsto
\left(\begin{array}{c}
\sigma\\
\tau_{\al}\\
\rho
\end{array}\right) \in [\cN_A^{\perp}]_{\theta}\otimes\cE(0,1)|_{\Si}
\end{equation*}
where $\tau_{\al}=\Pi^{\mu}_{\al}\tau_{\mu}$.
\end{thm}
\begin{proof}
We only need to establish that the map is independent of the choice of compatible contact forms, and this follows from comparing the submanifold and ambient versions of \eqref{TractorTransformation} noting that $\nabla_0\Upsilon=D_0\Upsilon$ and $\Upsilon_{\al}=\Pi^{\mu}_{\al}\Upsilon_{\mu}$ as in the proof of Lemma \ref{StandardTractorMapsLemma}.
\end{proof}
\begin{rem}\label{DiagonalTractorMapRemark}
If $\psi:\cE_{\Si}(1,0)\rarrow\cE(1,0)|_{\Si}$ is a local bundle isomorphism (unit as a local section of $\cR(1,0)$) and we denote by $T_{\psi}\iota$ the local isomorphism $\cE_I\rightarrow\cN_A^{\perp}$ given by Lemma \ref{StandardTractorMapsLemma} then isomorphism of Proposition \ref{DiagonalTractorMap} agrees with $T_{\psi}\iota\otimes\overline{\psi}$ where this is defined.
\end{rem}
Conjugating the map \eqref{DiagonalTractorMapEqn} and raising tractor indices one gets an isomorphism
\begin{equation*}
\Pi_I^A:\cE^I\otimes \cE_{\Si}(1,0)\rightarrow (\cN^A)^{\perp}\otimes\cE(1,0)|_{\Si}
\end{equation*}
and tensoring both sides with $\cE_{\Si}(-1,0)$ one gets another isomorphism
\begin{equation}\label{TwistedTractorMap1}
\cE^I\rightarrow (\cN^A)^{\perp}\otimes\underbrace{\cE(1,0)|_{\Si}\otimes\cE_{\Si}(-1,0)}_{\cR(1,0)}
\end{equation}
which we may also denote by $\Pi_I^A$. We think of $\Pi_I^A$ as a section of $\cE_I\otimes\cE^A|_{\Si}\otimes\cR(1,0)$ and $\Pi_A^I$ as a section of $\cE^I\otimes\cE_A|_{\Si}\otimes\cR(0,1)$. Thinking about these objects as sections emphasises that they can be interpreted as maps in a variety of ways. 
\begin{defn} \label{TwistedTractorMapDefn}
The isomorphism \eqref{TwistedTractorMap1} gives an injective bundle map
\begin{equation}
\cT^{\cR}\iota : \cT\Si \rightarrow\cT M|_{\Si}\otimes\cR(1,0)
\end{equation}
which we term the \emph{twisted tractor map}.
\end{defn}
The twisted tractor map is clearly filtration preserving, and restricts to an isomorphism $\cT^{1}\Si \rightarrow\cT^{1} M|_{\Si}\otimes\cR(1,0)$. This is just the trivial isomorphism
\begin{equation}\label{TrivialIsomorphism}
\cE_{\Si}(-1,0)\cong \cE(-1,0)|_{\Si}\otimes\cE(1,0)|_{\Si}\otimes\cE_{\Si}(-1,0).
\end{equation}
Since it is filtration preserving $\cT^{\cR}\iota$ also induces an injective bundle map $\cT^{0}\Si / \cT^{1}\Si \rightarrow(\cT^{0} M|_{\Si} / \cT^{1} M|_{\Si})\otimes\cR(1,0)$ and this is simply the tangent map $\cE^{\mu}\rightarrow \cE^{\al}|_{\Si}$ tensored with the isomorphism \eqref{TrivialIsomorphism}. The map $\cT\Si / \cT^{0}\Si \rightarrow(\cT M|_{\Si} / \cT^{0} M|_{\Si})\otimes\cR(1,0)$ induced by the twisted tractor map is the isomorphism
\begin{equation*}
\cE_{\Si}(0,1)\cong \cE(0,1)|_{\Si}\otimes\cR(1,0)
\end{equation*}
which simply comes from noting that $\cR(1,0)=\cR(0,-1)$ since $\cE_{\Si}(1,1)=\cE(1,1)|_{\Si}$. Note that since $\cR(1,0)$ is Hermitian, so is $\cT M|_{\Si}\otimes\cR(1,0)$, and $\cT^{\cR}\iota$ is clearly a Hermitian bundle map. These properties characterise the twisted tractor map.

\subsubsection{The adjoint tractor map}\label{SectionAdjointTractorMap}
Since $\cR(1,1)$ is canonically trivial the section $\Pi^A_I\Pi^J_B$ gives us a canonical bundle map
\begin{equation*}
\mathrm{End}(\cT\Si)\rightarrow\mathrm{End}(\cT M).
\end{equation*}
Since the twisted tractor map is metric preserving by restricting to skew-Hermitian endomorphisms we get a map
\begin{equation*}
\cA\iota : \cA\Si \rightarrow \cA M
\end{equation*}
which we term the \emph{adjoint tractor map}. Recalling the projection $\cA M\rightarrow TM$ given by \eqref{AdjointOntoTangent} we note that the diagram
\begin{equation}
\begin{array}{ccc}
\cA \Si & \rightarrow & \cA M \\
\downarrow &  & \downarrow\\
T\Si & \rightarrow & TM\\
\end{array}
\end{equation}
is easily seen to commute. So the adjoint tractor map is a lift of the tangent map.

\subsection{Relating Tractor Connections on $\cT\Si$}\label{RelatingTracConnsSec}
Using the twisted tractor map and the connection $\nabla^\cR$ we obtain a connection $\check{\nabla}$ on the standard (co)tractor bundle induced by the ambient tractor connection. Given a standard tractor field $u^J$ and a cotractor field $v_J$ on $\Si$ we define
\begin{equation}\label{cnabla-defn}
\check{\nabla}_i u^J = \Pi^J_B \nabla_i (\Pi^B_K u^K) \quad \mathrm{and} \quad \check{\nabla}_i v_J = \Pi^B_J \nabla_i (\Pi^K_B v_K)
\end{equation}
where by $\nabla$ we mean the ambient standard tractor connection $\nabla$ differentiating in directions tangent to $\Si$ (i.e. pulled back by $\iota$) coupled with the connection $\nabla^{\cR}$.

From Section \ref{The-Tractor-Connection} the submanifold intrinsic tractor connection $D$ on a section $v_{I}\overset{\theta_{\Si}}{=}(\si,\tau_{\mu},\rho)$ is given by
\begin{equation}
D_{\mu}v_{J}\overset{\theta_{\Si}}{=}\left(\begin{array}{c}
D_{\mu}\si-\tau_{\mu}\\
D_{\mu}\tau_{\nu}+iA_{\mu\nu}\si\\
D_{\mu}\rho-p_{\mu}{^{\nu}}\tau_{\nu}+t_{\mu}\si
\end{array}\right),
\quad
D_{\mub}v_{J}\overset{\theta_{\Si}}{=}\left(\begin{array}{c}
D_{\mub}\si\\
D_{\mub}\tau_{\nu}+\bh_{\mu\nub}\rho+p_{\nu\mub}\si\\
D_{\mub}\rho-iA_{\mub}{^{\nu}}\tau_{\nu}-t_{\mub}\si
\end{array}\right),
\end{equation}
and
\begin{equation}
D_{0}v_{J}\overset{\theta_{\Si}}{=}\left(\begin{array}{c}
D_{0}\si+\frac{i}{n+1}p\si-i\rho\\
D_{0}\tau_{\nu}+\frac{i}{n+1}p\tau_{\nu}-ip_{\nu}{^{\la}}\tau_{\la}+2it_{\nu}\si\\
D_{0}\rho+\frac{i}{n+1}p\rho+2it^{\nu}\tau_{\nu}+is\si
\end{array}\right)
\end{equation}
where $t_{\mu}$ and $s$ are the submanifold intrinsic versions of $T_{\al}$ and $S$ defined by \eqref{Ttensor} and \eqref{Sscalar} respectively. By contrast, for $\check{\nabla}$ we have:
\begin{prop}\label{cnabla-prop}
The connection $\check{\nabla}$ on a section $v_{I}\overset{\theta_{\Si}}{=}(\si,\tau_{\mu},\rho)$ of $\cE_I$ is given, in terms of any ambient contact form compatible with $\theta_{\Si}$, by
\begin{equation}\label{cnabla1}
\cnabla_{\mu}v_{J}\overset{\theta_{\Si}}{=}\left(\begin{array}{c}
D_{\mu}\si-\tau_{\mu}\\
D_{\mu}\tau_{\nu}+iA_{\mu\nu}\si\\
D_{\mu}\rho-P_{\mu}{^{\nu}}\tau_{\nu}+T_{\mu}\si
\end{array}\right),
\end{equation}
\begin{equation}\label{cnabla2}
\cnabla_{\mub}v_{J}\overset{\theta_{\Si}}{=}\left(\begin{array}{c}
D_{\mub}\si\\
D_{\mub}\tau_{\nu}+\bh_{\mu\nub}\rho+P_{\nu\mub}\si\\
D_{\mub}\rho-iA_{\mub}{^{\nu}}\tau_{\nu}-T_{\mub}\si
\end{array}\right),
\end{equation}
and
\begin{equation}\label{cnabla3}
\cnabla_{0}v_{J} \overset{\theta_{\Si}}{=}\left(\begin{array}{c}
D_{0}\si+\frac{i}{n+1}P_{\la}{^{\la}}\si-i\rho\\
D_{0}\tau_{\nu}+\frac{i}{n+1}P_{\la}{^{\la}}\tau_{\nu}-iP_{\nu}{^{\la}}\tau_{\la}+2iT_{\nu}\si\\
D_{0}\rho+\frac{i}{n+1}P_{\la}{^{\la}}\rho+2iT^{\nu}\tau_{\nu}+iS\si
\end{array}\right).
\end{equation}
\end{prop}
\begin{proof}
Choose any local isomorphism $\psi:\cE_{\Si}(1,0)\rarrow\cE(1,0)|_{\Si}$ compatible
with the canonical identification of $\cE_{\Si}(1,1)$ with $\cE(1,1)|_{\Si}$. Replacing $\si$ with $f\si$ where $f\in C^\infty(\Si,\bbC)$ we may take $\si$ to satisfy $\si\overline{\si}=\vsig_{\Si}$ where $\theta_{\Si}=\vsig_{\Si}\btheta_{\Si}$. We can thus factor the components of $v_I$ so that $v_{I}\overset{\theta_{\Si}}{=}(f\si,\xi_{\mu}\si,g\si)$, where $\xi_{\mu} \in \Gamma(\cE_{\mu})$ and $g\in \Gamma(\cE_{\Si}(-1,-1))$. If $\theta$ is an ambient contact form compatible with $\theta_{\Si}$ then (splitting the tractor bundles w.r.t. $\theta$, $\theta_{\Si}$) under the map $\cT_{\psi}\iota\otimes\overline{\psi}$ of Remark \ref{DiagonalTractorMapRemark}
\begin{equation*}
(f\si,\xi_{\mu}\si,g\si)\otimes\overline{\si}\mapsto(f\phi,\xi_{\al}\phi,g\phi)\otimes\overline{\phi}
\end{equation*}
where $\phi=\psi(\si)$ and $\xi_{\al}=\Pi_{\al}^{\mu}\xi_{\mu}$. Thus by definition we have
\begin{equation*}
v_B=\Pi^J_B v_{J}\overset{\theta}{=}\left(\begin{array}{c}
f\phi\\
\xi_{\be}\phi\\
g\phi
\end{array}\right)\otimes( \overline{\phi}\otimes \overline{\si}^{-1})
\end{equation*}
as a section of $\cE_B|_{\Si}\otimes \cR(1,0)$. Now one simply computes $\nabla_i v_B$ using the formulae \eqref{hhTractorConnection}, \eqref{ahTractorConnection}, and \eqref{0hTractorConnection} for the tractor connection along with Lemma \ref{RatioConnectionLemma} which relates $\nabla^{\cR}$ to the Tanaka-Webster connections on the ambient and intrinsic density bundles. We have
\begin{align}\label{cnablaDefCalc1}
\Pi_{C}^{B}\nabla_{\mu}v_{B} & \overset{\theta}{=}\left(\begin{array}{c}
(\nabla_{\mu}f)\phi+f\nabla_{\mu}\phi-\xi_{\mu}\phi\\
\Pi^{\be}_{\ga}(\nabla_{\mu}\xi_{\be})\phi+\xi_{\ga}\nabla_{\mu}\phi+i\Pi^{\be}_{\ga}A_{\mu\be} f\phi\\
(\nabla_{\mu}g)\phi+g\nabla_{\mu}\phi-P_{\mu}{^{\ga}}\xi_{\ga}\phi+T_{\mu}f\phi
\end{array}\right)\otimes ( \overline{\phi}\otimes \overline{\si}^{-1})\\*
\nonumber & \quad+\left(\begin{array}{c}
f\phi\\
\xi_{\ga}\phi\\
g\phi
\end{array}\right)\otimes\nabla^{\cR}_{\mu}( \overline{\phi}\otimes \overline{\si}^{-1})
\end{align}
where $A_{\mu\be}=\Pi^{\al}_{\mu}A_{\al\be}$, $P_{\mu}{^{\ga}}=\Pi^{\al}_{\mu}P_{\al}{^{\ga}}$, and $T_{\mu}=\Pi^{\al}_{\mu}T_{\al}$. By Corollary \ref{ATanNorCor} we have $\Pi^{\be}_{\ga}A_{\mu\be}=\Pi^{\nu}_{\ga}A_{\mu\nu}$. We also have $\Pi^{\be}_{\ga}\nabla_{\mu}\xi_{\be}=\Pi^{\nu}_{\ga}D_{\mu}\xi_{\nu}$. 
Now by Lemma \ref{RatioConnectionLemma} we have 
\begin{align*}
\nabla^{\cR}_{\mu}( \overline{\phi}\otimes \overline{\si}^{-1})
&=(\nabla_{\mu}\overline{\phi})\otimes\overline{\si}^{-1} + \overline{\phi}\otimes D_{\mu}(\overline{\si}^{-1})\\*
&=(\nabla_{\mu}\overline{\phi})\otimes\overline{\si}^{-1} + (\si^{-1}D_{\mu}\si) \overline{\phi}\otimes \overline{\si}^{-1}
\end{align*}
using that $\overline{\si}D_{\mu}(\overline{\si}^{-1})=-\overline{\si}^{-1}D_{\mu}\overline{\si}=\si^{-1}D_{\mu}\si$ since $D_{\mu}\vsig_{\Si}=0$. If $\theta=\vsig\btheta$ then since $\phi=\psi(\si)$ we must have $\phi\overline{\phi}=\vsig|_{\Si}$, and this implies that 
\begin{equation*}
(\nabla_{\mu}\phi)\otimes\overline{\phi}+\phi\otimes \nabla_{\mu}\overline{\phi}=0.
\end{equation*}
Using these to simplify \eqref{cnablaDefCalc1} we have
\begin{align*}
\Pi_{C}^{B}\nabla_{\mu}v_{B} & \overset{\theta}{=}\left(\begin{array}{c}
(D_{\mu}f)\phi-\xi_{\mu}\phi\\
\Pi^{\nu}_{\ga}(D_{\mu}\xi_{\nu})\phi+i\Pi^{\nu}_{\ga}A_{\mu\nu} f\phi\\
(D_{\mu}g)\phi-P_{\mu}{^{\nu}}\xi_{\nu}\phi+T_{\mu}f\phi
\end{array}\right)\otimes ( \overline{\phi}\otimes \overline{\si}^{-1})\\*
\nonumber & \quad+(\si^{-1}D_{\mu}\si)\left(\begin{array}{c}
f\phi\\
\xi_{\ga}\phi\\
g\phi
\end{array}\right)\otimes( \overline{\phi}\otimes \overline{\si}^{-1}).
\end{align*}
Applying $\Pi^{C}_{J}$ to the above display gives
\begin{equation*}
\Pi_{J}^{B}\nabla_{\mu}v_{B} \overset{\theta_{\Si}}{=}\left(\begin{array}{c}
(D_{\mu}f)\si-\xi_{\mu}\si\\
(D_{\mu}\xi_{\nu})\si+iA_{\mu\nu} f\si\\
(D_{\mu}g)\si-P_{\mu}{^{\nu}}\xi_{\nu}\si+T_{\mu}f\si
\end{array}\right)
+ 
\left(\begin{array}{c}
fD_{\mu}\si\\
\xi_{\nu}D_{\mu}\si\\
gD_{\mu}\si
\end{array}\right)
\end{equation*}
which proves \eqref{cnabla1}. Formula \eqref{cnabla2} is obtained similarly. In following the same process for \eqref{cnabla3} we obtain that
\begin{align*}
\cnabla_{0}v_{J} & \overset{\theta_{\Si}}{=}\left(\begin{array}{c}
D_{0}\si+\frac{i}{n+2}P\si-i\rho\\
D_{0}\tau_{\nu}+\frac{i}{n+2}P\tau_{\nu}-iP_{\nu}{^{\la}}\tau_{\la}+2iT_{\nu}\si\\
D_{0}\rho+\frac{i}{n+2}P\rho+2iT^{\nu}\tau_{\nu}+iS\si
\end{array}\right)\\*
& \quad +\left(-\frac{i}{(n+1)}P_{N\bar{N}}+\frac{i}{(n+2)(n+1)}P\right)\left(\begin{array}{c}
\si\\
\tau_{\nu}\\
\rho
\end{array}\right),
\end{align*}
the second term arising from the use of Lemma \ref{RatioConnectionLemma}. Simplifying this gives the result.
\end{proof}
\begin{rem}
By construction $\check{\nabla}$ preserves the tractor metric $h_{J\bar{K}}$. One can therefore obtain the formulae for $\check{\nabla}$ acting on sections of $\cE^I$ by conjugating the above formulae and using the identification of $\cE^I$ with $\cE_{\bar{I}}$ via the tractor metric.
\end{rem}

One can now easily compare the two connections $\cnabla$ and $D$ on $\cT\Si$.
\begin{defn}\label{difference-tractor-defn}
The \emph{difference tractor} $\sfS$ is the tractor endomorphism valued 1-form on $\Si$ given by the difference between  $\cnabla$ and $D$ on $\cT\Si$. Precisely, we have
\begin{equation}
\cnabla_X u = D_X u + \sfS(X) u \quad \mathrm{and} \quad \cnabla_X v = D_X v - v \circ \sfS(X)
\end{equation}
for $X\in \fX(\Si)$, $u\in \Gamma(\cT\Si)$, and $v\in\Gamma(\cT^*\Si)$.
\end{defn}
Given a contact form $\theta_{\Si}$ on $\Si$ the difference tractor $\sfS$ splits into components  $\sfS_{\mu J}{^{K}}$, $\sfS_{\mub J}{^{K}}$ and $\sfS_{0J}{^{K}}$ (with only the last of these depending on $\theta_{\Si}$). From the above formulae for  $\cnabla$ and $D$ we have, in terms of a compatible pair of contact forms,
\begin{equation}\label{hDifferenceTractor}
\sfS_{\mu J}{^{K}}= (P_{\mu}{^{\la}}-p_{\mu}{^{\la}})Z_J W^K_{\la} - (T_{\mu}-t_{\mu})Z_J Z^K,
\end{equation}
\begin{equation}\label{aDifferenceTractor}
\sfS_{\mub J}{^{K}}= -(P_{\nu\mub}-p_{\nu\mub})W^{\nu}_J Z^K + (T_{\mub}-t_{\mub})Z_J Z^K,
\end{equation}
and
\begin{align}\label{0DifferenceTractor}
\sfS_{0 J}{^{K}} & = -\tfrac{i}{m+2}(P_{\la}{^{\la}}-p)\delta^K_J + i(P_{\nu}{^{\la}}-p_{\nu}{^{\la}})W^{\nu}_J W_{\la}^K \\
\nonumber &\quad
 -2i(T_{\nu}-t_{\nu})W^{\nu}_J Z^{K} -2i(T^{\la}-t^{\la})Z_J W_{\la}^K - i(S-s) Z_J Z^K,
\end{align}
where $m+2=n+1$ in this case. Both $\sfS_{\mu J}{^{K}}$ and $\sfS_{\mub J}{^{K}}$ are invariant objects. Both have as projecting part the difference $P_{\mu\nub}-p_{\mu\nub}$; a manifestly CR invariant expression for this difference was given in Lemma \ref{SchoutenDifferenceLemma}. We can also give matrix formulae for the difference tractor, following the same conventions used in Section \ref{TractorCurvature} we have
\begin{equation*}
\sfS_{\mu J}{^{K}}= \left(
\begin{array}{ccc}
0 & 0 & 0\\
0 & 0 & 0 \\
t_{\mu}-T_{\mu} & P_{\mu}{^{\la}}-p_{\mu}{^{\la}}  & 0 
\end{array}\right),
\quad
\sfS_{\mub J}{^{K}}= \left(
\begin{array}{ccc}
0 & 0 & 0\\
p_{\nu\mub}-P_{\nu\mub} & 0 & 0 \\
T_{\mub}-t_{\mub} & 0  & 0 
\end{array}\right)
\end{equation*}
and
\begin{equation*}
\sfS_{0 J}{^{K}}= \left(
\begin{array}{ccc}
-\tfrac{i(P_{\la}{^{\la}}-p)}{m+2} & 0 & 0\\
-2i(T_{\nu}-t_{\nu}) & i(P_{\nu}{^{\la}}-p_{\nu}{^{\la}}) -\tfrac{i(P_{\la}{^{\la}}-p)}{m+2} & 0 \\
- i(S-s) & -2i(T^{\la}-t^{\la}) & -\tfrac{i(P_{\la}{^{\la}}-p)}{m+2}
\end{array}\right).
\end{equation*}

\begin{rem}\label{DifferenceTractorA0-valued}
Since both tractor connections $\cnabla$ and $D$ preserve the tractor metric on $\cT\Si$, the difference tractor must take values in skew-Hermitian endomorphisms of the tractor bundle (i.e. $\sfS$ is an $\cA\Si$-valued 1-form). This can also easily be seen from \eqref{hDifferenceTractor},  \eqref{aDifferenceTractor} and \eqref{0DifferenceTractor}, from which we see that $\sfS$ is in fact $\cA^0\Si$-valued.
\end{rem}

\subsection{The Tractor Gauss Formula}\label{TractorGaussFormula}

In order to write down the Gauss formula in Riemannian geometry one needs the tangent map (more precisely the pushforward) of the embedding, though one typically suppresses this from the notation. In order to give a standard tractor analogue we have sought a canonical `standard tractor map', but ended up instead with the twisted tractor map $\cT^{\cR}\iota$. However this poses no problem for constructing a Gauss formula, since the line bundle $\cR(1,0)$ we have had to twist with carries an invariant connection $\nabla^{\cR}$. 

Letting $\iota_*$ denote the induced map on sections coming from $\cT^{\cR}\iota$ we make he following definition:
\begin{defn}\label{tractor-sff-defn}
We define the \emph{tractor second fundamental form} $\bL$ by the \emph{tractor Gauss formula}
\begin{equation}\label{tractor-Gauss-formula}
\nabla_X \iota_* u \: = \: \iota_*(D_X u + \sfS(X) u) + \bL(X) \iota_* u
\end{equation}
which holds for any $X\in\fX(\Si)$ and $u\in \Gamma(\cT\Si)$, where $\nabla$ denotes the ambient tractor connection coupled with $\nabla^{\cR}$.
\end{defn}
This (combined with Theorem \ref{DiagonalTractorMap}) establishes Theorem \ref{TwistedTractorMapResult} for the case $m=n-1$, the result generalises straightforwardly (Section \ref{RelatingTractorsHC}).
\begin{rem}
By the definition of the difference tensor $\sfS$, for any $X\in\fX(\Si)$ and $u\in \Gamma(\cT\Si)$ we have that $\bL(X) \iota_* u$ is the orthogonal projection of $\nabla_X \iota_* u$ onto $\cN\otimes\cR(1,0)$. By definition then $\bL$ is a 1-form on $\Si$ valued in $\mathrm{Hom}(\cN^{\perp}\otimes\cR(1,0),\cN\otimes\cR(1,0))=\mathrm{Hom}(\cN^{\perp},\cN)$.
\end{rem}
Suppressing $\iota_*$ we write the tractor Gauss formula as
\begin{equation}
\nabla_X u \; = \; \underbrace{D_X u + \sfS(X) u}_{\text{`tangential \: part'}}\;\;  + \underbrace{\bL(X) u}_{\text{`normal \: part'}}
\end{equation}
for any $X\in\fX(\Si)$ and $u\in \Gamma(\cT\Si)$.

Writing $\Pi^B_J u^J$ as $u^B$ and contracting the Gauss formula on both sides with a unit normal cotractor $N_A$ we get that 
\begin{equation*}
N_C \bL(X)_{B}{^C} u^B=N_B \nabla_X u^B=-u^B\nabla_X N_B
\end{equation*}
for all sections $u^J$ of $\cE^J$ and $X\in\fX(\Si)$. Thus $\bL$ is given by
\begin{equation}\label{TractorSFF}
\bL_{i B}{^C}=-N^C\Pi^{B'}_B\nabla_{i} N_{B'}
\end{equation}
for any unit normal cotractor $N_C$. From this we have:
\begin{prop}\label{tractor-sff-prop}
With respect to a compatible pair of contact forms the components $\bL_{\mu B}{^C}$, $\bL_{\mub B}{^C}$, and $\bL_{0 B}{^C}$ of the tractor second fundamental form $\bL$ are given by
\begin{equation}\label{hTractorSFF}
\bL_{\mu B}{^C} = \II_{\mu\nu}{^{\ga}}\Pi^{\nu}_{\be}W^{\be}_B W_{\ga}^C + P_{\mu\bar{N}}N^{\ga} Z_B W_{\ga}^C,
\end{equation}
\begin{equation}\label{aTractorSFF}
\bL_{\mub B}{^C}=0,
\end{equation}
and
\begin{equation}\label{0TractorSFF}
\bL_{0 B}{^C}= i\Pi^{\be'}_{\be}P_{\be'\bar{N}}N^{\ga}W^{\be}_B W_{\ga}^C-2iT_{\bar{N}}N^{\ga} Z_B W_{\ga}^C
\end{equation}
where $N^{\al}$ is some unit holomorphic normal field, $P_{\mu\bar{N}}=\Pi^{\al}_{\mu} P_{\al\beb}N^{\beb}$, and $T_{\bar{N}}=T_{\alb}N^{\alb}$.
\end{prop}
\begin{proof}
One simply chooses a unit holomorphic normal field $N^{\al}$ and corresponding normal tractor $N^A$, then calculates $\Pi^{B'}_B\nabla_{i}N_{B'}$ using the formulae \eqref{hhTractorConnection}, \eqref{ahTractorConnection}, and \eqref{0hTractorConnection} for the ambient tractor connection. Using \eqref{TractorSFF} one immediately obtains \eqref{hTractorSFF}; for \eqref{aTractorSFF} one also has to use that $\II_{\mub\nu}{^{\ga}}=0$ (by Proposition \ref{SFFprop}) and $\Pi^{\alb}_{\mub}A_{\alb\beb}N^{\beb}=0$ (by Corollary \ref{ATanNorCor}), and for \eqref{0TractorSFF} one also has to use that $\II_{0\nu}{^{\ga}}=0$ (again by Proposition \ref{SFFprop}).
\end{proof}
The proposition shows that the invariant projecting part of $\bL_{\mu B}{^C}$ is $\II_{\mu\nu}{^{\ga}}\Pi^{\nu}_{\be}$, giving a manifestly CR invariant way of defining the CR second fundamental form.

\section{Higher Codimension Embeddings} \label{HigherCodimension}

It is straightforward to adapt our treatment of CR embeddings in the minimal codimension case to general codimension transversal CR embeddings. Here
we consider a CR embedding of $\iota:\Si^{2m+1}\rightarrow M^{2n+1}$ with $n=m+d$, and $m,d>0$. We keep our notation for bundles on $\Si$ and $M$ as before. We now have a rank $2d$ real conormal bundle $N^*\Si$, and the complexification of $N^*\Si$ splits as
\begin{equation}
\bbC N^*\Si = \cN_{\al}\oplus\cN_{\alb}
\end{equation}
where $\cN_{\al}$ is the annihilator of $T^{1,0}\Si$ in $(T^{1,0}M)^*|_{\Si}=\cE_{\al}|_{\Si}$ and $\cN_{\alb}=\overline{\cN_{\al}}$. We denote by $\rmN^{\al}_{\be}$ the orthogonal projection of $\cE^{\be}|_{\Si}$ onto the holomorphic normal bundle $\cN^{\al}$, and by $\Pi^{\al}_{\be}$ the tangential projection, so that $\Pi^{\al}_{\be}+\rmN^{\al}_{\be}=\delta^{\al}_{\be}$. We will also write $\rmN^{\al\beb}$ for $\bh^{\ga\deb}\rmN^{\al}_{\ga}\rmN^{\beb}_{\deb}=\bh^{\ga\beb}\rmN^{\al}_{\ga}$.

\begin{rem} Note that in passing to the general codimension there is no restriction on the signatures (or relative signature) of the CR manifolds, provided we have a nondegenerate transversal CR embedding.
\end{rem}

\subsection{Pseudohermitian Calculus}
We may continue to work with compatible contact forms in the general codimension case (see Remark \ref{CompatibleScalesRemark}). By Remark \ref{SubmanifoldTWpropRemark} the Tanaka-Webster connection $\nabla$ of an admissible ambient contact form $\theta$ induces the Tanaka-Webster connection $D$ of $\theta_{\Si}$ via the Webster metric $g_{\theta}$ as in Proposition \ref{SubmanifoldTWprop}. We can therefore define the (pseudohermitian) second fundamental form of a pair of compatible contact forms as in Definition \ref{PHsff} (i.e. via a Gauss formula). By Remark \ref{sffHigherCodimensionRemark} the only nontrivial components of the pseudohermitian second fundamental form are $\II_{\mu\nu}{^{\ga}}$ and its conjugate. Also by Remark \ref{sffHigherCodimensionRemark} the pseudohermitian torsion of any admissible ambient contact form satisfies $\Pi^{\al}_{\mu}A_{\al\be}\rmN^{\be}_{\ga}=0$.

The higher codimension analogue of Lemma \ref{CRsffLemma} is:
\begin{lem}\label{CRsffHigherCodim}
Given compatible contact forms one has
\begin{equation}
N_{\ga}\II_{\mu\nu}{^{\ga}}=-\Pi^{\be}_{\nu}\nabla_{\mu}N_{\be}
\end{equation}
for any holomorphic conormal field.
\end{lem}
From Lemma \ref{CRsffHigherCodim} we see again that the component $\II_{\mu\nu}{^{\ga}}$ of the pseudohermitian second fundamental form does not depend on the compatible pair of contact forms used to define it (cf.\ Corollary \ref{CRsffCor}). 

The Gauss, Codazzi and Ricci equations given in the three propositions of Section \ref{Pseudohermitian-GCR} hold in the general codimension case with the same proofs (noting that the normal fields used in the proofs of Proposition \ref{pseudohermitian-Codazzi-prop} and Proposition \ref{pseudohermitian-Ricci-equation-prop} were arbitrary).

\subsection{Relating Densities}\label{RelatingDensities-HC}
As before we define $\Lambda^{1,0}_{\perp}\Si$ to be the bundle of forms in $\Lambda^{1,0}M|_{\Si}$ annihilating $T\Si$. Again we may identify $\Lambda^{1,0}_{\perp}\Si$ with $\cN_{\al}$ by restriction to $T^{1,0}M|_{\Si}$. We write $\Lambda^{d,0}_{\perp}\Si$ for the line bundle $\Lambda^d(\Lambda^{1,0}_{\perp}\Si)$. The following lemma is easily established (cf.\ Lemma \ref{CanonicalBundlesRelatedLemma} and Lemma \ref{ConnectionsOnCanonicalBundlesRelatedLemma}):
\begin{lem} \label{CanonicalBundlesRelatedLemmaHC}
Along $\Si$ the submanifold and ambient canonical bundles are related by the canonical isomorphism which intertwines the Tanaka-Webster connections of any compatible pair of contact forms
\begin{align*}
\scrK|_{\Si} &\: \cong \:\scrK_{\Si} \: \otimes \Lambda_{\perp}^{d,0}\Si \\*
\iota^*\nabla \hphantom{|} &\: \cong \:\:\: D \:\: \otimes \;\; \nabla^{\perp}.
\end{align*}
\end{lem}
Identifying $\Lambda^{1,0}_{\perp}\Si$ with $\cN_{\al}$ we may write $\Lambda^{d,0}_{\perp}\Si$ as $\cN_{[\al_1 \cdots \al_d]}$. Tensoring both sides of the isomorphism of Lemma \ref{CanonicalBundlesRelatedLemmaHC} with $\cE(d,0)|_{\Si}$ we obtain (cf.\ Corollary \ref{ambient-intrinsic-densities-cor} and Corollary \ref{RelatingTWconnectionsOnDensities}):
\begin{cor} \label{DensityBundlesRelatedCorHC}
Along $\Si$ the submanifold and ambient density bundles are related by the canonical isomorphism which intertwines the Tanaka-Webster connections of any compatible pair of contact forms
\begin{align*}
\mathcal{E}(-m-2,0)|_{\Si} &\: \cong \:\mathcal{E}_{\Sigma}(-m-2,0) \otimes \cN_{[\al_1 \cdots \al_d]}(d,0)\\*
\iota^*\nabla\hphantom{---} &\: \cong \:\hphantom{---} D \hphantom{--}\:\:\:\:\:\! \otimes \hphantom{--}\;\; \nabla^{\perp}.
\end{align*}
\end{cor}
Note that the line bundle $\cN_{[\al_1 \cdots \al_d]}(d,0)$ is the $d^{th}$ exterior power of $\cN_{\al}(1,0)$. Once again this bundle will turn out to be canonically isomorphic to a subbundle $\cN^*=\cN_A$ of the ambient cotractor bundle $\cE_A|_{\Si}$ (see Section \ref{RelatingTractorsHC}), and hence once again $\cN_{\al}(1,0)$ carries a canonical invariant connection. As before this connection turns out to be explicitly realised as the normal Weyl connection on $\cN_{\al}(1,0)$ of any admissible ambient contact form. The normal Weyl connection on $\cN_{\al}(1,0)$ agrees with the normal Tanaka-Webster connection when differentiating in contact directions; when differentiating in Reeb directions the two are related by 
\begin{equation}\label{NormalWeyl0HC}
\nabla^{W,\perp}_0 \tau_{\al} = \nabla^{\perp}_0 \tau_{\al} -i\rmN^{\al'}_{\al}P_{\al'}{^{\be}}\tau_{\be}+\frac{i}{n+2}P\tau_{\al}
\end{equation}
for any section $\tau_{\al}$ of $\cN_{\al}(1,0)$. The curvature $R^{\Lambda^d\cN^*}$ of this connection on $\cN_{[\al_1 \cdots \al_d]}(d,0)$ is again generically non zero, and we have
\begin{equation}
R^{\Lambda^d\cN^*}_{\mu\nub}=(m+2)(P_{\mu\nub}-p_{\mu\nub})+(P_{\la}{^{\la}}-p)\bh_{\mu\nub}
\end{equation}
(cf.\ Lemma \ref{haCurvatureOfRatioBundleProp}), $R^{\Lambda^d\cN^*}_{\mu\nu}=0$, $R^{\Lambda^d\cN^*}_{\mub\nub}=0$, and (cf.\ \eqref{h0NormalCurvature})
\begin{equation}
R^{\Lambda^d\cN^*}_{\mu0} = V_{\mu\beb\ga}\rmN^{\ga\beb} - iP_{\ga}{^{\nu}}\II_{\mu\nu}{^{\ga}}.
\end{equation}

We thus define the ratio bundle of densities $\cR(w,w')$ as before (Definition \ref{ratio-bundle-defn}) and see from Corollary \ref{DensityBundlesRelatedCorHC} that these bundles carry a canonical connection $\nabla^{\cR}$ coming from the connection $\nabla^{\cN}$ on $\Lambda^d\cN^* = \cN_{[\al_1 \cdots \al_d]}(d,0)$. We have therefore established Proposition \ref{RatioAndNormalBundlesResult}. Using Corollary \ref{DensityBundlesRelatedCorHC} and \eqref{NormalWeyl0HC} we may relate the connection $\nabla^{\cR}$ to the coupled submanifold-ambient Tanaka-Webster connection (cf.\ Lemma \ref{RatioConnectionLemma}):
\begin{lem}\label{RatioConnectionLemmaHC}
In terms of a compatible pair of contact forms, $\theta$, $\theta_{\Si}$, the connection $\nabla^{\cR}$ on a section $\phi\otimes\si$ of $\cE(w,w')|_{\Si}\otimes\cE_{\Sigma}(-w,-w')$ is given by 
\begin{equation}
\nabla^{\cR}_{\mu}(\phi\otimes\si)=(\nabla_{\mu}\phi)\otimes\si+\phi\otimes(D_{\mu}\si),
\end{equation}
\begin{equation}
\nabla^{\cR}_{\mub}(\phi\otimes\si)=(\nabla_{\mub}\phi)\otimes\si+\phi\otimes(D_{\mub}\si),
\end{equation}
and
\begin{equation}
\nabla^{\cR}_{0}(\phi\otimes\si)=(\nabla_{0}\phi)\otimes\si+\phi\otimes(D_{0}\si) + \tfrac{w-w'}{m+2}(iP_{\al\beb}\rmN^{\al\beb}-\tfrac{i}{n+2}P)\phi\otimes\si.
\end{equation}
\end{lem}

\subsection{Relating Tractors}\label{RelatingTractorsHC}
As before we have a canonical isomorphism from $\cN_{\al}(1,0)$ to a subbundle $\cN_A$ of $\cE_{A}|_{\Si}$, given with respect to any admissible ambient contact form $\theta$ by
\begin{equation}
\tau_{\al}\mapsto \tau_A \overset{\theta}{=}\left(\begin{array}{c}
0\\
\tau_{\alpha}\\
0
\end{array}\right).
\end{equation}
There is a corresponding isomorphism of $\cN^{\al}(-1,0)$ with a subbundle $\cN^A$ of $\cE^A|_{\Si}$, and we alternatively denote the dual pair $\cN^A$ and $\cN_A$ by $\cN$ and $\cN^*$ respectively. The normal tractor connection $\nabla^{\cN}$ on $\cN_A$ agrees with the normal Weyl connection of any admissible ambient contact form on $\cN_{\al}(1,0)$ (cf.\ Proposition \ref{NormalTractorConnectionProp}).

Sections \ref{TractorsAndDensities} and \ref{RelatingTractorBundles} are valid without change in the general codimension case. In particular, Lemma \ref{StandardTractorMapsLemma} and Theorem \ref{DiagonalTractorMap} hold. Thus we may talk about the twisted standard tractor map
\begin{equation*}
\cT^{\cR}\iota : \cT\Si \rightarrow\cT M|_{\Si}\otimes \cR(1,0)
\end{equation*}
and the corresponding sections $\Pi^A_I$ of $\cE_I\otimes\cE^A|_{\Si}\otimes\cR(1,0)$ and $\Pi^I_A$ of $\cE^I\otimes\cE_A|_{\Si}\otimes\cR(0,1)$. This allows us to define the connection $\cnabla$ on $\cT\Si$ as in \eqref{cnabla-defn}; one can then easily establish the expressions for $\cnabla$ given in Proposition \ref{cnabla-prop} in the general codimension setting (the proof is essentially the same, with Lemma \ref{RatioConnectionLemmaHC} generalising Lemma \ref{RatioConnectionLemma}). The difference tractor $\sfS$, defined as in Definition \ref{difference-tractor-defn}, is then still given in component form by \eqref{hDifferenceTractor}, \eqref{aDifferenceTractor}, and \eqref{0DifferenceTractor}.  

We define the tractor second fundamental form $\bL$ by a tractor Gauss formula as in Definition \ref{tractor-sff-defn}. This establishes Theorem \ref{TwistedTractorMapResult}. One then also has that
\begin{equation}\label{TractorSFFonNormal}
\bL_{i B}{^C} N_C = -\Pi^{B'}_B\nabla_i N_{B'}
\end{equation}
for any section $N_{A}$ of $\cN_{A}$. From this we get (cf.\ Proposition \ref{tractor-sff-prop}):
\begin{prop}\label{tractor-sff-prop-HC}
With respect to a compatible pair of contact forms the components $\bL_{\mu B}{^C}$, $\bL_{\mub B}{^C}$, and $\bL_{0 B}{^C}$ of the tractor second fundamental form $\bL$ are given by
\begin{equation}\label{hTractorSFF-HC}
\bL_{\mu B}{^C} = \II_{\mu\nu}{^{\ga}}\Pi^{\nu}_{\be}W^{\be}_B W_{\ga}^C + P_{\mu\deb}N^{\ga\deb} Z_B W_{\ga}^C,
\end{equation}
\begin{equation}\label{aTractorSFF-HC}
\bL_{\mub B}{^C}=0,
\end{equation}
and
\begin{equation}\label{0TractorSFF-HC}
\bL_{0 B}{^C}= i\Pi^{\be'}_{\be}P_{\be'\deb}N^{\ga\deb}W^{\be}_B W_{\ga}^C -2iT_{\deb}N^{\ga\deb} Z_B W_{\ga}^C.
\end{equation}
\end{prop}

\section{Invariants of CR Embedded Submanifolds}\label{InvariantsOfSubmanifolds}

For many problems in geometric analysis it is important to construct
the invariants that are, in a suitable sense, polynomial in the jets
of the structure. Riemannian theory along these lines was developed by
Atiyah-Bott-Patodi for their approach to the heat equation asymptotics
\cite{ABP}, and in \cite{FeffermanParabolic} Fefferman initiated a corresponding
programme for conformal geometry and hypersurface type CR geometry. As
explained in \cite{BaileyEastwoodGraham} there are two steps to such
problems. The first is to capture the jets (preferably to all orders)
of the geometry invariantly and in an algebraically manageable
manner. The second is to use this algebraic structure to construct all
invariants. The latter boils down to Lie representation theory, for
the case of parabolic geometries this is difficult, and despite the
progress in \cite{BaileyEastwoodGraham} and \cite{Gover-Advances} for conformal geometry
and CR geometry many open problems remain. For the conformal and CR
cases the first part is treated by the Fefferman and Fefferman-Graham
ambient metric constructions \cite{FeffermanParabolic,FeffermanGraham84,FGbook} and alternatively
by the tractor calculus \cite{BaileyEastwoodGover-Thomas'sStrBundle,CapGover-TracCalc,Gover-Advances}. It is beyond the
scope of the current work to fully set up and treat the corresponding
invariant theory for CR submanifolds. However we wish to indicate here
that the first geometric step, of capturing the jets effectively, is
solved via the tools developed above. In particular we will show that
it is straightforward to proliferate invariants of a (transversally embedded) CR submanifold. 
It seems reasonable to hope that these methods will form
the basis of a construction of all invariants of CR embeddings (in an
appropriate sense).

\subsection{Jets of the Structure}\label{JetsOfStructure}

We now show that the jets of the structure of a CR embedding are captured effectively by the basic invariants we have introduced in our `tractorial' treatment of CR embeddings.

Observe that the tractor Gauss formula \eqref{tractor-Gauss-formula} may be rewritten in the form
\begin{equation}\label{tractor-Gauss-alternate}
\nabla_X \cT^{\cR}\iota = \cT^{\cR}\iota \circ \sfS(X) + \bL(X)\circ \cT^{\cR}\iota
\end{equation}
for any $X\in T\Si$, where $\cT^{\cR}\iota$ is interpreted as a section of $\cT M|_{\Si}\otimes \cT^* \Si \otimes \cR(1,0)$ and $\nabla$ here denotes the (pulled back) ambient tractor connection coupled with the submanifold tractor connection and the canonical connection $\nabla^{\cR}$. Using this we have the following proposition:
\begin{prop}\label{JetsOfEmbeddingProp}
Given a transversal CR embedding $\iota:\Si\rightarrow M$, the $2$-jet of the map $\iota$ at a point $x\in\Si$ is encoded by $\iota(x)$, $\cT^{\cR}_x\iota$, $\sfS_x$ and $\bL_x$.
\end{prop}
\begin{proof}
Recalling Section \ref{SectionAdjointTractorMap} we note that the twisted tractor map $\cT^{\cR}\iota$ determines the adjoint tractor map $\cA\iota$ (by restricting $\cT^{\cR}\iota\otimes \overline{\cT^{\cR}\iota}$). Since the adjoint tractor map lifts the tangent map, the 1-jet $(\iota(x),T_x \iota)$ of $\iota$ at a point $x\in\Si$ is also determined by the pair $(\iota(x),\cT^{\cR}_x \iota)$. The proposition then follows from \eqref{tractor-Gauss-alternate}.
\end{proof}
In the \emph{jets of the structure} of a CR embedding $\iota:\Si\rightarrow M$ we include the jets of the ambient and submanifold CR structures, along with the jets of the map $\iota$. A CR invariant of the embedding should depend only on these jets evaluated along the submanifold. The jets of the ambient and submanifold CR structures are determined by the respective tractor curvatures. Thus from Proposition \ref{JetsOfEmbeddingProp} we have:
\begin{prop}\label{BasicInvariantsProp}
The jets of the structure of a transversal CR embedding $\iota:\Si\rightarrow M$ at $x\in \Si$ are determined algebraically by $\iota(x)$, the submanifold and ambient CR structures (as Cartan geometries) at $x$, the twisted tractor map $\cT_x^{\cR}\iota$, as well as the jets of the difference tensor $\sfS$, the tractor second fundamental form $\bL$, the submanifold tractor curvature $\ka^{\Si}$ at $x$, and the full (i.e. ambient) jets of the ambient tractor curvature $\ka$ at $\iota(x)$.
\end{prop}
In order to complete the first step of the invariant theory programme we need to package the jets of $\sfS$, $\bL$, $\ka^{\Si}$ and $\ka$ in an algebraically manageable way. Note that the standard tractor bundle, tractor metric, and canonical tractor $Z$ are all determined algebraically from structure of a CR geometry (as a Cartan geometry). Thus if we package the jets of $\sfS$, $\bL$, $\ka^{\Si}$ and $\ka$ into sequences of CR invariant tractors then one may combine these tractors by tensoring them and using the submanifold and ambient metrics to contract indices. One can also use the twisted tractor map to change submanifold tractor indices to ambient ones before making contractions (the ratio bundle of densities is also determined algebraically from the submanifold and ambient CR structures). This would not only complete the first step of the invariant theory programme, but would also suggest an obvious approach to the second of the two steps.

\subsection{Packaging the Jets}\label{PackagingJets}

One way to define iterated derivatives of the difference tractor $\sfS$ and submanifold curvature $\ka^{\Si}$ would be to repeatedly apply the submanifold fundamental derivative (or D-operator) of \cite{CapGover-TracCalc}. Denoting the submanifold fundamental derivative by $\boldsymbol{D}$, if $f$ is $\sfS$ or $\ka_{\Si}$ then by Theorem 3.3 of \cite{CapGover-TracCalc} the $k$-jet of $f$ at $x\in \Si$ is determined by the section
\begin{equation*}
(f,\boldsymbol{D}f,\boldsymbol{D}^2 f, \ldots , \boldsymbol{D}^k f)
\end{equation*}
of $\bigoplus^k_{l=0}\left(\bigotimes^l \cA^*\Si \otimes \mathcal{W}\right)$ evaluated at $x$, where $\mathcal{W}$ equals $\Lambda^1\Si\otimes\cA\Si$ or $\Lambda^2\Si\otimes\cA\Si$ respectively. The ambient jets of $\ka$ can be similarly captured by iterating the ambient fundamental derivative, and one can also capture the jets of $\bL$ by using the submanifold fundamental derivative twisted with the ambient tractor connection. Here instead we parallel the approach taken in \cite{Gover-Advances} to conformal invariant theory by first putting the tractor valued forms $\sfS$, $\bL$, $\ka^{\Si}$ and $\ka$ into tractors (invariantly and algebraically) using the natural inclusion of the cotangent bundle into the adjoint tractor bundle, and then using double-D-operators.

Let $B^a_{A\bar{B}}$ denote the map $T^* M \hookrightarrow \cA M$ given explicitly by \eqref{CotangentIntoAdjoint} and
$B^i_{I\bar{J}}$ denote the map $T^*\Si\hookrightarrow \cA\Si$. 
\begin{defn}\label{LiftedTractorExpressions}
We define the respective \emph{lifted (tractor) expressions} of the tractor valued forms $\sfS$, $\bL$, $\ka^{\Si}$ and $\ka$ to be
\begin{equation*}
\sfS^{\phantom{i}}_{I\Ibp J}{^K}=B^i_{I\Ibp}\sfS_{iJ}{^K}, \quad 
\bL_{I\Ibp B}{^C}=B^i_{I\Ibp}\bL_{iB}{^C}, \quad \ka^{\Si}_{I\Ibp J\bar{J'}K\bar{L}}=B^i_{I\bar{I}'}B^j_{J\bar{J'}}\ka^{\Si}_{ijK\bar{L}},
\end{equation*}
and
\begin{equation*}
\ka_{A\bar{A'}B\bar{B'}C\bar{D}}=B^a_{A\bar{A'}}B^b_{B\bar{B'}}\ka_{abC\bar{D}}.
\end{equation*}
Explicitly this means, for example, that
\begin{equation*}
\sfS_{I\bar{J}K}{^L}= 
\sfS_{\mu K}{^L} W_I^{\mu} Z_{\bar{J}} - 
\sfS_{\nub K}{^L} Z_I W_{\bar{J}}^{\nub} - i \sfS_{0K}{^L} Z_I Z_{\bar{J}}.
\end{equation*}
\end{defn}

By \eqref{haDoubleD-OnWeight0} the double-D-operator $\mathbb{D}_{A\bar{B}}$ acting on unweighted ambient tractors can be written as 
\begin{equation}\label{haAmbientDoubleD-OnWeight0}
\mathbb{D}_{A\bar{B}}=B^a_{A\bar{B}}\nabla_a
\end{equation}
where $\nabla$ is the ambient tractor connection. Similarly the double-D-operator $D_{I\bar{J}}$ acting on unweighted submanifold tractors can be written as 
\begin{equation} \label{haSubmanifoldDoubleD-OnWeight0}
\mathbb{D}_{I\bar{J}}=B^i_{I\bar{J}}D_i
\end{equation}
where $D_i$ denotes the submanifold tractor connection. By coupling $D_i$ in \eqref{haSubmanifoldDoubleD-OnWeight0} with $\nabla_i$ we enable the double-D-operator operator $\mathbb{D}_{I\bar{J}}$ to act iteratively on the unweighted (mixed) tractor $\bL_{K\bar{L} C}{^D}$. Noting that each of the lifted tractor expressions given in Definition \ref{LiftedTractorExpressions} is unweighted we therefore have:
\begin{prop}\label{PropJetsInTractors}
Let $\iota:\Si\rightarrow M$ be a transversal CR embedding, and let $\mathbb{D}$ denote the submanifold double-D-operator $\mathbb{D}_{I\bar{J}}$. If $f$ equals $\sfS$, $\bL$, or $\ka^{\Si}$ then the $k$-jet of $f$ at $x$ is determined by the section
\begin{equation}\label{PropJetsInTractorsEqn}
(\tilde{f},\mathbb{D}\tilde{f},\mathbb{D}^2 \tilde{f}, \ldots , \mathbb{D}^k \tilde{f})
\end{equation}
of $\bigoplus^k_{l=0}\left(\bigotimes^l \cA\Si \otimes \mathcal{W}\right)$ evaluated at $x$, where $\tilde{f}$ is the lifted tractor expression for $f$ and $\mathcal{W}$ equals $\otimes^2\cA\Si$, $\cA\Si\otimes\cA M|_{\Si}$, or $\otimes^3\cA\Si$ respectively.
\end{prop}
\noindent Along with the corresponding proposition for the ambient curvature:
\begin{prop}\label{PropAmbientCurvatureJetsInTractors}
Let $\iota:\Si\rightarrow M$ be a transversal CR embedding, and let $\mathbb{D}$ denote the ambient double-D-operator $D_{A\bar{B}}$. The $k$-jet of the ambient curvature $\ka$ at $\iota(x)$ is determined by the section 
\begin{equation}\label{PropAmbientCurvatureJetsInTractorsEqn}
(\tilde{\ka},\mathbb{D}\tilde{\ka},\mathbb{D}^2 \tilde{\ka}, \ldots , \mathbb{D}^k \tilde{\ka})
\end{equation}
of $\bigoplus^k_{l=0}\left(\bigotimes^{l+3} \cA M \right)$ evaluated at $\iota(x)$.
\end{prop}

By packaging the jets of the basic invariants $\sfS$, $\bL$, $\ka^{\Si}$ and $\ka$ into sequences of tractors (i.e. sections of associated bundles corresponding to representations of the appropriate pseudo-special unitary groups) we have solved the first step of the invariant theory.

\subsection{Tensor and Scalar Invariants}

Given a compatible pair of contact forms there is a natural notion of \emph{submanifold pseudohermitian Weyl invariant}, analogous to the notions of Riemannian Weyl invariant and hypersurface Riemannian Weyl invariant in \cite{ABP, FGbook, Gover-Advances, GoverWaldronWillmore}. Submanifold pseudohermitian Weyl invariants are sections of the bundles
\begin{equation*}
\cE_{\mu_1\cdots\mu_s\nub_1\cdots\nub_t}\otimes\cE_{\al_1\cdots\al_{s'}\beb_1\cdots\beb_{t'}}|_{\Si}.
\end{equation*}
The pseudohermitian Weyl invariants of a given tensor type are generated (complex linearly) by partial (or in the scalar case, full) contractions of tensor products of covariant derivatives of the submanifold and ambient pseudohermitian curvatures, the submanifold and ambient pseudohermitian torsions, and the CR second fundamental form (this being the only nontrivial component of the pseudohermitian second fundamental form); each of the covariant derivatives may be in the holomorphic, antiholomorphic, or Reeb direction(s); the contractions are made using the submanifold and ambient CR Levi forms, as well as the restriction $\mathrm{N}_{\al\beb}$ of the ambient Levi form to the normal bundle, and the tensors $\Pi_{\mu}^{\al}$ and $\Pi^{\al}_{\mu}$ (and their conjugates) may be used to contract ambient indices with submanifold indices; any remaining upper indices are lowered using the submanifold and ambient CR Levi forms. Since $r_{\mu\nub}{^{\lab}}_{\rhob}$, $R_{\al\beb}{^{\gab}}_{\deb}$, $A_{\mu\nu}$, $A_{\al\be}$, and $\II_{\mu\nu}{^{\ga}}$ all have no weight, the generators we describe all have diagonal weight and can be treated as weightless, since the fixed pair of compatible contact forms gives a canonical trivialisation of $\cE_{\Si}(w,w)\cong\cE(w,w)|_{\Si}$. Some examples of scalar pseudohermitian Weyl invariants are
\begin{equation*}
\II_{\mu\nu}{^{\ga}}\II^{\mu\nu}{_{\ga}}, \quad \Pi^{\nu}_{\al}\Pi^{\la}_{\be}\Pi^{\mub}_{\gab}(\nabla^{\al}R^{\be\gab})D_0 D_{\mub}A_{\nu\la}, \quad R_{\al\beb}N^{\al\beb}
\end{equation*}
and any complex linear combination of these. Some examples of tensor pseudohermitian invariants are
\begin{equation*}
\II_{\mu\nu}{^{\ga}}, \quad D_{\mu_1}\cdots D_{\mu_s}\II_{\mu\nu}{^{\ga}}, \quad r_{\mu\nub\la\rhob}\II^{\mu\la}{_{\ga}},
\end{equation*}
where $D$ denotes the submanifold Tanaka-Webster connection coupled with the normal Tanaka-Webster connection $\nabla^{\perp}$.

Although we define pseudohermitian Weyl invariants to be of weight zero (since we may trivialise the diagonal density bundles) each of the generators described above has its own natural weight. Using the canonical CR invariant identification of $\cE(w,w)|_{\Si}$ with $\cE_{\Si}(w,w)$ each of the generators may be thought of as a section of 
\begin{equation}\label{weightedPWIbundle}
\cE_{\mu_1\cdots\mu_s\nub_1\cdots\nub_t}(w,w)\otimes\cE_{\al_1\cdots\al_{s'}\beb_1\cdots\beb_{t'}}|_{\Si}
\end{equation}
for some $w\in \mathbb{Z}$. For example $\II_{\mu\nu}{^\ga}\II^{\mu\nu}{_{\ga}}$ is naturally a section of $\cE_{\Si}(-1,-1)$. When fixing a pair of compatible contact forms it makes sense to trivialise $\cE_{\Si}(w,w)\cong\cE(w,w)|_{\Si}$ and add together invariants of different weights. However, if one is really interested in CR invariants it is unnatural to trivialise the diagonal density bundles and one should only consider pseudohermitian Weyl invariants which are homogeneous in weight. Being homogeneous in weight is a necessary (but not sufficient) condition for a pseudohermitian Weyl invariant to be \emph{conformally covariant}, i.e. to transform merely by factor $e^{w\Ups}$ for some $w\in \mathbb{Z}$ when the the compatible pair of contact forms are rescaled by $e^{\Upsilon}$. A pseudohermitian Weyl invariant which is conformally covariant is in fact \emph{conformally invariant} (i.e. CR invariant) as a section of the appropriately weighted bundle \eqref{weightedPWIbundle}. By \emph{local CR invariant tensor (or scalar) field} in the following we mean a pseudohermitian Weyl invariant which is homogeneous in weight and is thought of as a weighted tensor (or scalar) field. Although it is easy to write down pseudohermitian Weyl invariants which are homogeneous in weight, these will in general have very complicated transformation laws under a conformal rescaling of the compatible pair of contact forms. This makes local CR invariant tensor (or scalar) fields very difficult to find by na\"ive methods, hence the need for an invariant theory.

\subsection{Making All Invariants}\label{MakingAllInvariants}

By tensoring together tractors of the form appearing in \eqref{PropJetsInTractorsEqn} and \eqref{PropAmbientCurvatureJetsInTractorsEqn} along $\Si$, making partial contractions, and taking projecting parts one may construct a large family of local CR invariant (weighted) scalars and tensors. It is an algebraic problem to show that all scalar (or all tensor) invariants can be obtained by such a procedure. This is a subtle and difficult problem, which extends Fefferman's parabolic invariant theory programme to the submanifold-relative case (where there are two parabolics around, $P$ and $P_{\Si}$). Even in the original case of invariant theory for CR manifolds, despite much progress, important questions remain unresolved \cite{BaileyEastwoodGraham, Hirachi}. We do not attempt to resolve these issues here.

We do wish to indicate, however, that there is scope for development of the invariant theory for CR manifolds, and now CR embeddings, along the lines of the treatment of invariant theory for conformal and projective structures in \cite{Gover-ProjectiveSpace, Gover-ParabolicSubgroupSL, Gover-InvariantsProjectiveGeometries, Gover-Advances}. The tractor calculus we have developed for CR embeddings provides all the machinery needed to emulate the constructions of conformal Weyl and quasi-Weyl invariants in \cite{Gover-Advances}. We anticipate that further insight from the projective case \cite{Gover-InvariantsProjectiveGeometries} will be needed, and our machinery is sufficient for this also. With all the tools in hand this article therefore puts us in good stead in terms of our ability to construct (potentially all) invariants of CR embeddings.

\subsection{Practical Constructions}\label{PracticalConstructions}
Although in principle one may need only the invariant tractors appearing in Propositions \ref{PropJetsInTractors} and \ref{PropAmbientCurvatureJetsInTractors} for construction of general invariants, in practice it is much more efficient to use the richer calculus which is available. First of all, there are many alternative ways to construct tractor expressions from the basic invariants (recall for instance the curvature tractor of Section \ref{CurvatureTractorSection}). Secondly, there are several invariant operators besides the double-D-operators $\mathbb{D}_{I\bar{J}}$ and $\mathbb{D}_{A\bar{B}}$ that can be used to act on these tractor expressions.

\subsubsection{Alternative tractor expressions}\label{AlternativeTractorExp}

Along with the lifted tractor expressions for the submanifold and ambient tractor curvatures one may of course construct invariants using the curvature tractor of Section \ref{CurvatureTractorSection} or using the tractor defined in equation \eqref{CurvatureTractorAlternative} of that section. Correspondingly we may also use the middle operators of Section \ref{MiddleOps} to construct tractors from our basic invariants $\sfS$ and $\bL$
\begin{equation*}
\sfS_{IJ}{^K}= \sfM_I^{\mu}\sfS_{\mu J}{^K}, \quad  \sfS_{\bar{I}J}{^K}= \sfM_{\bar{I}}^{\mub}\sfS_{\mub J}{^K}, \quad \mathrm{and} \quad
\bL_{IB}{^C}= \sfM_I^{\mu}\bL_{\mu B}{^C}
\end{equation*}
using indices to distinguish them from the difference tractor $\sfS$ and the tractor second fundamental form $\bL$ (and from their lifted tractor expressions in Section \ref{PackagingJets}). Recall that $\bL_{\mub B}{^C}=0$.

From \eqref{Wtransformation} it follows immediately that
\begin{equation*}
Z^{\phantom{\be}}_{[A} W_{B]}^{\be} = \frac{1}{2}(Z^{\phantom{\be}}_{A} W_{B}^{\be} - Z^{\phantom{\be}}_{B} W_{A}^{\be})
\end{equation*}
does not depend on the choice of contact form, so is CR invariant. Using $Z^{\phantom{\be}}_{[A} W_{B]}^{\be}$ and $Z^{\phantom{\nu}}_{[I} W_{J]}^{\nu}$ we construct the tractors
\begin{equation*}
\sfS_{II'J}{^K}= Z^{\phantom{\mu}}_{[I} W_{I']}^{\mu}\sfS_{\mu J}{^K}, \qquad \sfS_{\bar{I}\bar{I}'J}{^K}= Z^{\phantom{\mu}}_{[\bar{I}} W_{\bar{I}']}^{\mub}\sfS_{\mub J}{^K},
\end{equation*}
\begin{equation*}
\bL_{II'B}{^C}=Z^{\phantom{\mu}}_{[I} W_{I']}^{\mu}\bL_{\mu B}{^C}, 
\end{equation*}
\begin{equation*}
\ka^{\Si}_{II'\bar{J}\bar{J}'K\bar{L}}=Z^{\phantom{\mu}}_{[I} W_{I']}^{\mu}Z^{\phantom{\nub}}_{[\bar{J}} W_{\bar{J}']}^{\nub}\ka^{\Si}_{\mu\nub K\bar{L}},
\end{equation*}
and
\begin{equation*}
\ka_{AA'\bar{B}\bar{B}'C\bar{D}}=Z^{\phantom{\al}}_{[A} W_{A']}^{\al}Z^{\phantom{\beb}}_{[\bar{B}} W_{\bar{B}']}^{\beb}\ka_{\al\beb C\bar{D}}.
\end{equation*}

Of course one can also make invariant tractors from the invariant components $\sfS_{\mu J}{^K}$, $\sfS_{\mub J}{^K}$, $\bL_{\mu B}{^C}$, and so on, by making contractions with the submanifold (or ambient) CR Levi form. For example, we have the following invariant tractors on $\Si$
\begin{equation*}
\bh^{\mu\nub} \sfS_{\mu I}{^J}\sfS_{\nub K}{^{L}},\quad \bh^{\mu\nub} \bL_{\mu A}{^B}\sfS_{\nub \bar{K}}{^{\bar{L}}}, \quad \bh^{\mu\rhob}\bh^{\la\nub}\ka^{\Si}_{\mu\nub K\bar{L}}\bL_{\la A}{^B}\bL_{\rhob \bar{C}}{^{\bar{D}}},
\end{equation*}
where $\sfS_{\nub \bar{K}}{^{\bar{L}}}=\overline{\sfS_{\nu K}{^{L}}}$ and $\bL_{\rhob \bar{C}}{^{\bar{D}}}=\overline{\bL_{\rho C}{^{D}}}$.
One can also contract some, or all, of the tractor indices. Note that
\begin{equation}
\bL_{\mu B \bar{C}}\bL^{\mu B \bar{C}} = \II_{\mu\nu\gab}\II^{\mu\nu\gab}
\end{equation}
whereas
\begin{equation}
\sfS_{\mu J \bar{K}}\sfS^{\mu J \bar{K}} = 0 \quad \mathrm{and} \quad \Pi^J_B \Pi^{\bar{K}}_{\bar{C}}\sfS_{\mu J \bar{K}}\bL^{\mu B \bar{C}} = 0
\end{equation}
from the explicit formulae for $\sfS$ and $\bL$ in terms of compatible contact forms and the orthogonality relations \eqref{SplittingTractorContractions} between the splitting tractors.
\begin{rem}
Although $\sfS_{\mu J \bar{K}}\sfS^{\mu J \bar{K}} = 0$ one can extract a scalar invariant from the partial contraction $\sfS_{\mu J \bar{K}}\sfS^{\mu J' \bar{K}}$ by observing that this tractor is of the form $f Z_J Z^{J'}$, so that the $(-1,-1)$ density $f$ must be CR invariant. In fact $f$ is simply the invariant
\begin{equation*}
(P_{\mu\nub} - p_{\mu\nub}) (P^{\mu\nub}-p^{\mu\nub}).
\end{equation*}
One of the difficulties inherent in constructing all invariants is predicting when this type of phenomenon will happen when dealing with various contractions of higher order invariant tractors (such as those appearing in Proposition \ref{PropJetsInTractors}).
\end{rem}

\subsubsection{Invariant operators}\label{SubmanifoldInvtOps}

Along with the double-D-operators \eqref{haAmbientDoubleD-OnWeight0} and \eqref{haSubmanifoldDoubleD-OnWeight0} used in Section \ref{PackagingJets} one may of course use the submanifold and ambient tractor D-operators of Section \ref{SectionTractorD-op} and the other double-D-operators $\mathbb{D}_{IJ}$ and $\mathbb{D}_{AB}$. In order to act on tractors of mixed (submanifold-ambient) type, with potentially submanifold and ambient weights, we will need to appropriately couple the submanifold intrinsic invariant D-operators with the ambient tractor connection and with the canonical connection on the density ratio bundles. Note that these operators also form the building blocks for constructing invariant differential operators on CR embedded submanifolds.

We first need to use the ratio bundles to eliminate ambient weights. Let $\cE^{\Phi}_{\Si}$ denote any submanifold intrinsic tractor bundle and let $\cE^{\Phi}_{\Si}(w,w')$ denote $\cE^{\Phi}_{\Si}\otimes\cE_{\Si}(w,w')$. Let $\cE^{\tilde{\Phi}}(\tilde{w},\tilde{w}')$ denote any ambient tractor bundle, weighted by ambient densities. We make the identification
\begin{equation}\label{WeightedTractorIdentification}
\mathcal{E}_{\Si}^{\Phi}(w,w')\otimes\cE^{\tilde{\Phi}}(\tilde{w},\tilde{w}')|_{\Si}=\mathcal{E}_{\Si}^{\Phi}(w+\tilde{w},w'+\tilde{w}')\otimes\cE^{\tilde{\Phi}}|_{\Si}\otimes\cR(\tilde{w},\tilde{w}')
\end{equation}
which motivates the following definition:
\begin{defn}
We define the \emph{reduced weight} of a section $f^{\Phi\tilde{\Phi}}$ of the bundle \eqref{WeightedTractorIdentification} to be $(\check{w},\check{w}')=(w+\tilde{w},w'+\tilde{w}')$.
\end{defn}
One can extend any of the submanifold D-operators to act on sections of the bundle \eqref{WeightedTractorIdentification} by taking the relevant D-operator acting on submanifold tractors with the reduced weight, expressed in terms of a choice of contact form $\theta_{\Si}$, and coupling the Tanaka-Webster connection of $\theta_{\Si}$ with the (pulled back) ambient tractor connection and the ratio bundle connection $\nabla^{\cR}$.

We illustrate how this works for the submanifold tractor D-operator $\mathsf{D}_I$. We define the CR invariant operator 
\begin{equation*}
\mathsf{D}_{I}:\mathcal{E}_{\Si}^{\Phi}(\check{w},\check{w}')\otimes\cE^{\tilde{\Phi}}|_{\Si}\otimes\cR(\tilde{w},\tilde{w}') \rightarrow\mathcal{E}_{I}\otimes\mathcal{E}_{\Si}^{\Phi}(\check{w}-1,\check{w}')\otimes\cE^{\tilde{\Phi}}|_{\Si}\otimes\cR(\tilde{w},\tilde{w}')
\end{equation*}
by
\begin{equation}
\mathsf{D}_{I}f^{\Phi\tilde{\Phi}}\overset{\theta_{\Si}}{=}\left(\begin{array}{c}
\check{w}(m+\check{w}+\check{w}')f^{\Phi\tilde{\Phi}}\\
(m+\check{w}+\check{w}')D_{\mu}f^{\Phi\tilde{\Phi}}\\
-\left(D^{\nu}D_{\nu}f^{\Phi\tilde{\Phi}}+i\check{w}D_{0}f^{\Phi\tilde{\Phi}}+\check{w}(1+\frac{\check{w}'-\check{w}}{m+2})pf^{\Phi\tilde{\Phi}}\right)
\end{array}\right)
\end{equation}
where $D$ denotes the Tanaka-Webster connection of $\theta_{\Si}$ coupled with the submanifold tractor connection, the (pulled back) ambient tractor connection, and the ratio bundle connection $\nabla^{\cR}$.

\subsubsection{Computing higher order invariants}\label{ComputingInvariants}

Using the tractor calculus we have developed it is now straightforward to construct further local (weighted scalar, or other) invariants of a CR embedding. One can differentiate the various tractors constructed from the basic invariants in Sections \ref{PackagingJets} and \ref{AlternativeTractorExp} using the invariant operators of Section \ref{SectionTractorD-op} and Section \ref{SubmanifoldInvtOps}, tensor these together, and make contractions using the tractor metrics (and the twisted tractor map). One can also make partial contractions and take projecting parts.

To illustrate our construction we give an example invariant and compute the form of the invariant in terms of the Tanaka-Webster calculus of a pair of compatible contact forms: Consider the nontrivial reduced weight $(-2,-2)$ density
\begin{equation}\label{ExampleInvariant}
\mathcal{I}=\mathsf{D}^I\mathsf{D}^{\bar{J}}(\Pi^B_I \Pi^{\bar{D}}_{\bar{J}}\bh^{\mu\nub} h_{C\bar{E}} \bL_{\mu B}{^C}\bL_{\nub\bar{D}}{^{\bar{E}}}).
\end{equation}
Since $\Pi^B_I \Pi^{\bar{D}}_{\bar{J}}$ is by definition a section of $\cE_{I \bar{J}}\otimes \cE^{B\bar{D}}|_{\Si} \otimes \cR(1,1)$, and $\cR(1,1)$ is canonically trivial and flat, we see that 
\begin{equation}
f_{I\bar{J}}=\Pi^B_I \Pi^{\bar{D}}_{\bar{J}}\bh^{\mu\nub} h_{C\bar{E}} \bL_{\mu B}{^C}\bL_{\nub\bar{D}}{^{\bar{E}}}
\end{equation}
has reduced weight $(-1,-1)$ and no ratio bundle weight (diagonal ratio bundle weights can be ignored). Therefore in this case we do not need to couple the submanifold tractor D-operator with any ambient connection in order to define $\mathsf{D}^I\mathsf{D}^{\bar{J}}f_{I\bar{J}}$. From the definition of $\mathsf{D}_J$ we have
\begin{align}\label{ExampleStep1}
\mathsf{D}^{\bar{J}}f_{I\bar{J}} &= -(m-2)Y^{\bar{J}} f_{I\bar{J}} + (m-2)W^{\nu\bar{J}}D_{\nu} f_{I\bar{J}} \\*
\nonumber & \quad - Z^{\bar{J}}\left( D^{\nu}D_{\nu} f_{I\bar{J}} -iD_0 f_{I\bar{J}} - p f_{I\bar{J}}\right)
\end{align}
where $D$ denotes the submanifold tractor connection coupled with the Tanaka-Webster connection of some submanifold contact form $\theta_{\Si}$ and $Z$, $W$, $Y$ are the splitting tractors corresponding to the choice of $\theta_{\Si}$. The tractor $\mathsf{D}^{\bar{J}}f_{I\bar{J}}$ has weight $(-2,-1)$ and so, applying $\mathsf{D}_{\bar{I}}$ and contracting, we have
\begin{align}\label{ExampleStep2}
\mathsf{D}^I\mathsf{D}^{\bar{J}}f_{I\bar{J}} &= -2(m-3)Y^I \mathsf{D}^{\bar{J}}f_{I\bar{J}} + (m-3)W^{\mub I}D_{\mub}\mathsf{D}^{\bar{J}}f_{I\bar{J}} \\*
\nonumber &\quad -Z^I \left( D^{\mub}D_{\mub} \mathsf{D}^{\bar{J}}f_{I\bar{J}} - 2i D_0 \mathsf{D}^{\bar{J}}f_{I\bar{J}} - 2\frac{m+3}{m+2}p \mathsf{D}^{\bar{J}} f_{I\bar{J}} \right).
\end{align}
If we choose $\theta$ admissible and compatible with $\theta_{\Si}$ then \eqref{hTractorSFF-HC} implies
\begin{equation}\label{ExampleStep3}
f_{I\bar{J}} = \II_{\mu\la}{^{\ga}}\II_{\nub\rho}{^{\epb}}\bh^{\mu\nub} \bh_{\ga\epb} W^{\la}_I W^{\rhob}_{\bar{J}} + P_{\mu\alb}\rmN^{\be\alb}P_{\be\nub}\bh^{\mu\nub} Z_I Z_{\bar{J}}.
\end{equation}
In each term on the right hand side of \eqref{ExampleStep1} and \eqref{ExampleStep2} there is a contraction with a tractor, using the orthogonality relations between the tractor projectors simplifies the calculation significantly since one can ignore terms that will vanish after these contractions. So, for example, one easily computes that
\begin{equation*}
W^{\nu\bar{J}}D_{\nu}f_{I\bar{J}}=D^{\rhob}(\II_{\mu\la}{^{\ga}}\II^{\mu}{_{\rhob\ga}}W^{\la}_I) + m P_{\mu\alb}\rmN^{\be\alb}P_{\be}{^{\mu}}Z_I.
\end{equation*}
Another efficient way to compute terms is to commute the splitting tractors forward past each appearance of the connection $D$ using the submanifold versions of \eqref{hNablaY}-\eqref{Soldering3}. Since $Z^{\bar{J}}f_{I\bar{J}}=0$ and $[D_{\nu},Z^{\bar{J}}]=D_{\nu}Z^{\bar{J}}=0$ we have $Z^{\bar{J}}D_{\nu}f_{I\bar{J}}=0$, from which we get 
\begin{equation*}
Z^{\bar{J}}D^{\nu}D_{\nu}f_{I\bar{J}}=-W^{\nu\bar{J}}D_{\nu}f_{I\bar{J}}
\end{equation*}
using $[D^{\nu},Z^{\bar{J}}]=D^{\nu}Z^{\bar{J}}=W^{\nu\bar{J}}$; thus two of the terms in \eqref{ExampleStep1} coincide, up to a factor, simplifying our calculations significantly. Computing similarly
\begin{align*}
Z^{\bar{J}}D_{0}f_{I\bar{J}} &= D_{0}(Z^{\bar{J}}f_{I\bar{J}}) - (D_0 Z^{\bar{J}})f_{I\bar{J}} \\*
&=  iP_{\mu\alb}\rmN^{\be\alb}P_{\be}{^{\mu}}Z_I
\end{align*}
using that $Z^{\bar{J}}f_{I\bar{J}}=0$ and $D_0 Z^{\bar{J}}=-iY^{\bar{J}} +\frac{i}{m+2}pZ^{\bar{J}}$. Putting these together yields
\begin{equation}
\mathsf{D}^{\bar{J}}f_{I\bar{J}} = (m-1) D^{\rhob}(\II_{\mu\la}{^{\ga}}\II^{\mu}{_{\rhob\ga}}W^{\la}_I) + (m-1)^2 P_{\mu\alb}\rmN^{\be\alb}P_{\be}{^{\mu}}Z_I.
\end{equation}
Repeating this procedure for \eqref{ExampleStep2} we eventually obtain
\begin{align}
\label{ExampleInvariantComputed} 
\mathcal{I} = (m-1)\left[\vphantom{N^{\be\alb}} \right.
&(m-2)D^{\la}D^{\rhob}(\II_{\mu\la}{^{\ga}}\II^{\mu}{_{\rhob\ga}})  \\* 
\nonumber \qquad \qquad \quad  \; &+ (D^{\rhob}D_{\rhob} - 2iD_0 -\tfrac{(m+4)(m-2)}{m+2} p)(\II_{\mu\la}{^{\ga}}\II^{\mu\la}{_{\ga}})  \\*
\nonumber   &- \:\! (m-2)(m-4)p^{\la\rhob}\II_{\mu\la}{^{\ga}}\II^{\mu}{_{\rhob\ga}} \\*
\nonumber   &+ \:\! \left. (m-1)(m-2)(m-4) P_{\mu\alb}\rmN^{\be\alb}P_{\be}{^{\mu}} \right].
\end{align}

\section{A CR Bonnet Theorem}\label{CRBonnet}

In classical surface theory the Bonnet theorem (or fundamental theorem of surfaces) says that if a covariant 2-tensor $\II$ on an abstract Riemannian surface $(\Si,g)$ satisfies the Gauss and Codazzi equations then (locally about any point) there exists an embedding of $(\Si,g)$ into Euclidean 3-space which realises the tensor $\II$ as the second fundamental form. A more general version of the Bonnet theorem states that if we specify on a Riemannian manifold $(\Si^m,g)$ a rank $d$ vector bundle $N\Si$ with bundle metric and metric preserving connection and an $N\Si$-valued symmetric covariant 2-tensor $\II$ satisfying the Gauss, Codazzi and Ricci equations then (locally) there exists an embedding of $(\Si^m,g)$ into Euclidean $n$-space, where $n=m+d$, realising $N\Si$ as the normal bundle and $\II$ as the second fundamental form. Here we give a CR geometric analogue of this theorem.

\subsection{Locally Flat CR Structures}

The Bonnet theorem given in section \ref{BonnetTheoremSub} generalises and is motivated by the following well-known theorem on locally flat CR structures. The proof we give will be adapted to give a proof of the Bonnet theorem.
\begin{thm}\label{LocallyFlatTheorem}
A nondegenerate CR manifold $(M^{2n+1},H,J)$ of signature $(p,q)$ with vanishing tractor curvature is locally equivalent to the signature $(p,q)$ model hyperquadric $\mathcal{H}$.
\end{thm}
\begin{proof}
The signature $(p,q)$ model hyperquadric $\mathcal{H}$ can be realised as the space of null (i.e. isotropic) complex lines in the projectivisation of $\mathbb{C}^{p+1,q+1}$. Since the tractor curvature vanishes one may locally identify the standard tractor bundle $\cT M$ with the trivial bundle $M\times\mathbb{C}^{p+1,q+1}$ so that the tractor connection becomes the trivial flat connection and the tractor metric becomes the standard inner product on $\mathbb{C}^{p+1,q+1}$. The canonical null line subbundle $L=\cT^{1} M$ of $\cT M$ (spanned by the weighted canonical tractor $Z^A$) then gives rise to a map from $M$ into the model hyperquadric given by
\begin{equation}\label{LocallyFlatMap}
M \ni x \mapsto L_x \subset \mathbb{C}^{p+1,q+1}.
\end{equation}
We need to show that the map $f:M\rightarrow \mathbb{P}(\mathbb{C}^{p+1,q+1})$ given by \eqref{LocallyFlatMap} is a local CR diffeomorphism.

The maximal complex subspace in the tangent space to $\mathcal{H}$ at the point $\ell$, where $\ell \subset \mathbb{C}^{p+1,q+1}$ is an isotropic line, is the image of $\ell^{\perp}$ under the tangent map of the projection $\mathbb{C}^{p+1,q+1}\rightarrow \mathbb{P}(\mathbb{C}^{p+1,q+1})$. Choosing a nowhere zero local section $\rho$ of $L=\cE(-1,0)$ we get a lift of the map $f$ to a map $f_{\rho}:M\rightarrow \mathbb{C}^{p+1,q+1}$. The map $L=\cE(-1,0)\hookrightarrow\cT M=\cE^A$ is given explicitly by $\rho\mapsto \rho Z^A$. Since the tractor connection is flat the tangent map of $f_{\rho}$ at $x \in M$ is given by
\begin{equation*}
T_x M\ni X \mapsto \nabla_X(\rho Z^A) \in \mathbb{C}^{p+1,q+1}.
\end{equation*}
By the respective conjugates $\nabla_{\beb}Z^A=0$ and $\nabla_{\be}Z^A=W^A_{\be}$ of \eqref{Soldering1} and \eqref{Soldering2} (fixing any background contact form and raising indices using the tractor metric) the tangent map $T_x f_{\rho}$ restricted to contact directions maps onto a complementary subspace to $L_x$ inside $L_x^{\perp}$ and induces a complex linear isomorphism of $H_x$ with $L_x^{\perp}/L_x$; combined with \eqref{Soldering3} we see that $T_x f_{\rho}$ is injective and its image is transverse to $L_x$. Now $f$ is the composition of $f_{\rho}$ with the projectivisation map $\mathbb{C}^{p+1,q+1}\setminus \{0\}\rightarrow \mathbb{P}(\mathbb{C}^{p+1,q+1})$; thus we have that $T_x f$ is injective, and further that $f$ is a local CR diffeomorphism from $M$ to the signature $(p,q)$ model hyperquadric in $\mathbb{P}(\mathbb{C}^{p+1,q+1})$.
\end{proof}
\begin{rem}
Throughout this article we have implicitly identified the CR tractor bundle $\cT M$ with the holomorphic part of its complexification in the standard way. In the above proof we have therefore also implicitly identified the tangent space to $\mathbb{C}^{p+1,q+1}$ at any point with the holomorphic tangent space; the section $\rho Z^A$ should be understood as a section of the holomorphic tractor bundle, the map $f_{\rho}$ being determined by the corresponding section of the real tractor bundle.
\end{rem}
The map constructed in the proof is the usual Cartan developing map for a flat Cartan connection, though constructed using tractors and the projective realisation of the model hyperquadric. The fact that the map constructed is a local CR diffeomorphism relies on the soldering property of the canonical Cartan/tractor connection on $M$, which is captured in the formulae \eqref{Soldering1}, \eqref{Soldering2}, and \eqref{Soldering3}. 

\subsection{CR Tractor Gauss-Codazzi-Ricci Equations}\label{TractorGCR}

Our tractor based treatment of transversal CR embeddings in Sections \ref{Submanifolds-and-Tractors} and \ref{HigherCodimension} has shown us exactly what data should be prescribed on a CR manifold in a CR version of the Bonnet theorem: Consider a transversal CR embedding $\Si^{2m+1}\hookrightarrow M^{2n+1}$ between nondegenerate CR manifolds. Then along $\Si$ the ambient standard tractor bundle splits as an orthogonal direct sum $\left(\cT\Si\otimes\cR(-1,0)\right)\oplus \cN$ with the ratio bundle $\cR(-1,0)$ being the dual of an $(m+2)^{th}$ root of the top exterior power $\Lambda^d\cN$ of the normal tractor bundle $\cN$. It is also easy to see that the (pulled back) ambient tractor connection decomposes along $\Si$ as
\begin{equation}\label{TractorGauss-Weingarten}
\nabla = \left(\begin{array}{cc} 
D\otimes\nabla^{\cR}+\sfS  &  -\bL^{\dagger}   \\
 \bL  &   \nabla^{\cN}  \\ 
\end{array}\right) 
\qquad \mathrm{on} \quad \begin{array}{c} 
\cT\Si \otimes \cR(-1,0)  \\
\oplus \\
\cN\\
\end{array}
\end{equation}
where $D\otimes\nabla^{\cR}$ denotes the coupled connection on $\cT\Si \otimes \cR(-1,0)$, with $D$ the submanifold tractor connection and $\nabla^{\cR}$ the connection induced on $\cR(-1,0)$ by the normal tractor connection $\nabla^{\cN}$. The objects $\sfS$ and $\bL$ are as defined in Sections \ref{RelatingTracConnsSec} and \ref{TractorGaussFormula}, and $\bL^{\dagger}(X)$ is the Hermitian adjoint of $\bL(X)$ with respect to the ambient tractor metric for any $X\in\mathfrak{X}(\Si)$. The bundle $\cN$ carries a Hermitian metric $h^{\cN}$ induced by the ambient tractor metric. We refer to the triple $(\cN,\nabla^{\cN},h^{\cN})$ along with $(\cR(-1,0),\nabla^{\cR})$ and the invariants $\sfS$ and $\bL$ as the (extrinsic) \emph{induced data} coming from the CR embedding.

The above observations also establish Proposition \ref{NormalTractorsAndTractorSFFResult}.
\begin{rem}
The Hermitian adjoint $\bL^{\dagger}$ of $\bL$ appears because of \eqref{TractorSFFonNormal}. 
 Note that $\bL^{\dagger}_{iB}{^C}=\overline{\bL_i{^{\bar{C}}}_{\bar{B}}}$ so that in particular $\bL^{\dagger}_{\mub B}{^C}=\overline{\bL_{\mu}{^{\bar{C}}}_{\bar{B}}}$ and $\bL^{\dagger}_{\mu B}{^C}=0$. Note also that for any $X\in \mathfrak{X}(\Si)$
\begin{equation}
\left(\begin{array}{cc} 
\sfS(X)  &  -\bL^{\dagger}(X)   \\
 \bL(X)  &   0 \\ 
\end{array}\right) 
\end{equation}
is a skew-Hermitian endomorphism of $\cT M|_{\Si}$ since each of the connections appearing in \eqref{TractorGauss-Weingarten} preserves the appropriate Hermitian bundle metric.
\end{rem}

We can also easily see what the integrability conditions should be on this abstract data: Observe that the curvature of the connection \eqref{TractorGauss-Weingarten} acting on sections of $\cT M|_{\Si}$ is given by
\begin{align*}
\left(\begin{array}{cc} 
D\otimes\nabla^{\cR}+\sfS  &  -\bL^{\dagger}   \\
 \bL  &   \nabla^{\cN}  \\ 
\end{array}\right) \wedge&
\left(\begin{array}{cc} 
D\otimes\nabla^{\cR}+\sfS  &  -\bL^{\dagger}   \\
 \bL  &   \nabla^{\cN}  \\ 
\end{array}\right) \\
 &= \left(\begin{array}{cc} 
\ka^{D\otimes\nabla^{\cR}+\sfS} -\bL^{\dagger}\wedge\bL  &  -\d\bL^{\dagger} - \sfS\wedge\bL^{\dagger}  \\
 \d\bL + \bL\wedge\sfS &   \ka^{\cN} - \bL\wedge\bL^{\dagger} \\ 
\end{array}\right)
\end{align*}
where $\ka^{D\otimes\nabla^{\cR}+\sfS}$ is the curvature of $D\otimes\nabla^{\cR}+\sfS$, $\d\bL$ and $\d\bL^{\dagger}$ are the respective covariant exterior derivatives of $\bL$ and $\bL^{\dagger}$ with respect to $D\otimes\nabla^{\cR}\otimes\nabla^{\cN}$, and $\ka^{\cN}$ is the curvature of $\nabla^{\cN}$. The above display expresses the pullback of the ambient curvature by the embedding in terms of the induced data of the CR embedding. Writing these relations component-wise leads to the CR tractor Gauss, Codazzi, and Ricci equations; denoting the pullback of the ambient curvature simply by $\ka$ these are, respectively,
\begin{equation}\label{CRtractorGauss}
\Pi \circ \ka \circ \Pi = \ka^{D\otimes\nabla^{\cR}+\sfS} -\bL^{\dagger}\wedge\bL,
\end{equation}
\begin{equation}\label{CRtractorCodazzi}
\mathrm{N} \circ \ka \circ \Pi = \d\bL + \bL\wedge\sfS, 
\end{equation}
and
\begin{equation}\label{CRtractorRicci}
\mathrm{N} \circ \ka \circ \mathrm{N} = \ka^{\cN} - \bL\wedge\bL^{\dagger}
\end{equation}
where $\Pi$ and $\mathrm{N}$ denote the complementary `tangential' and `normal' projections acting on a section $v=(v^{\top},v^{\perp})$ of $\cT M|_{\Si}$. Of course 
\begin{equation*}
\ka^{D\otimes\nabla^{\cR}+\sfS} = \ka^{\Si} - \ka^{\cR(1,0)} + \d \sfS + \sfS\wedge\sfS
\end{equation*}
where $\ka^{\cR(1,0)}$ denotes the curvature of $\cR(1,0)$ acting as a bundle endomorphism via multiplication, $\ka^{\Si}$ is the submanifold tractor curvature, and $\d \sfS$ is the covariant exterior derivative of $\sfS$ with respect to the submanifold tractor connection. Note that the equation $\Pi \circ \ka \circ \mathrm{N} = -\d\bL^{\dagger} - \sfS\wedge\bL^{\dagger}$ is determined by \eqref{CRtractorCodazzi}. Note also that the tractor Ricci equation \eqref{CRtractorRicci} determines the normal tractor curvature $\ka^{\cN}$ in terms of the ambient curvature and the tractor second fundamental form.
\begin{rem}
One can easily write the terms appearing in the tractor Gauss, Codazzi, and Ricci equations more explicitly using abstract indices. For instance we have
\begin{equation*}
(\bL^{\dagger}\wedge\bL)_{ijK}{^L} = 2\bL^{\dagger}_{[i|E}{^L}\bL^{\phantom{\dagger}}_{|j]K}{^E} = 2\bL_{[i|}{^L}_E\bL_{|j]K}{^E}
\end{equation*}
where we use $\bL_{jK}{^E}=\bL_{jC}{^E}\Pi^C_K$ and $\bL^{\dagger}_{iE}{^L}=\bL^{\dagger}_{iE}{^D}\Pi_D^L$ since we are identifying $\cN^{\perp}\subset\cT M|_{\Si}$ with $\cT\Si \otimes \cR(-1,0)$.
\end{rem}

\subsection{The CR Bonnet Theorem}\label{BonnetTheoremSub}

With the notion of induced data on the submanifold from a CR embedding given in the previous section we can now give the following theorem:
\begin{thm}\label{BonnetTheorem}
Let $(\Sigma^{2m+1},H,J)$ be a signature $(p,q)$ CR manifold and suppose we have a complex rank $d$ vector bundle $\mathcal{N}$ on $\Sigma$ equipped with a signature $(p',q')$ Hermitian bundle metric $h^{\mathcal{N}}$ and metric connection $\nabla^{\mathcal{N}}$. Fix an $(m+2)^{th}$ root $\mathcal{R}$ of $\Lambda^d\mathcal{N}$, and let $\nabla^{\mathcal{R}}$ denote the connection induced by $\nabla^{\mathcal{N}}$. Suppose we have a $\mathcal{N}\otimes\mathcal{T}^*\Sigma\otimes\mathcal{R}$ valued 1-form $\bL$ which annihilates the canonical tractor of $\Si$ and an $\mathcal{A}^0\Sigma$ valued 1-form $\mathsf{S}$ on $\Sigma$ such that the connection
\begin{equation*}
\nabla := \left(\begin{array}{cc} 
D\otimes\nabla^{\mathcal{R}}+ \mathsf{S}  &  -\mathbb{L}^{\dagger}   \\
 \mathbb{L}  &   \nabla^{\mathcal{N}}  \\ 
\end{array}\right) 
\qquad \mathrm{on} \quad \begin{array}{c} 
\mathcal{T}\Sigma \otimes \mathcal{R}^*  \\
\oplus \\
\mathcal{N}\\
\end{array}
\end{equation*}
is flat (where $D$ is the submanifold tractor connection), then (locally) there exists a transversal CR embedding of $\Sigma$ into the model $(p+p',q+q')$ hyperquadric $\mathcal{H}$, unique up to automorphisms of the target, realising the specified extrinsic data as the induced data.
\end{thm}
\begin{proof}
Since the complex line bundle $\Lambda^d\mathcal{N}$ is normed by $h^{\cN}$, the bundle $\cR$ is also normed. This means that the tractor metric $h^{\cT \Si}$ induces a Hermitian bundle metric on $\cT\Si\otimes\cR^*$, which we again denote by $h^{\cT \Si}$. We therefore have a Hermitian bundle metric $h=h^{\cT \Si}+h^{\cN}$ on the bundle $\left(\cT\Si\otimes\cR^*\right)\oplus \cN$. Since $\sfS$ is adjoint tractor valued (i.e. skew-Hermitian endomorphism of $\cT \Si$ valued) the connection $D\otimes\nabla^{\mathcal{R}}+ \mathsf{S}$ on $\cT\Si\otimes\cR^*$ preserves $h^{\cT \Si}$. Collectively, the terms involving $\bL$ and $\bL^{\dagger}$ in the displayed definition of $\nabla$ constitute a one form valued in skew-Hermitian endomorphisms of $\left(\cT\Si\otimes\cR^*\right)\oplus \cN$. Combined with the fact that $\nabla^{\cN}$ preserves $h^{\cN}$ this shows that $\nabla$ preserves the Hermitian bundle metric $h$.

The signature $(p+p',q+q')$ model hyperquadric $\mathcal{H}$ can be realised as the space of null complex lines in the projectivisation of $\mathbb{T}=\mathbb{C}^{p+p'+1,q+q'+1}$. Since the connection $\nabla$ on $\left(\cT\Si\otimes\cR^*\right)\oplus \cN$ is flat and preserves $h$ one may locally identify this bundle with the trivial bundle $\Si\times \mathbb{T}$ such that $\nabla$ becomes the trivial flat connection and $h$ becomes the standard signature $(p+p'+1,q+q'+1)$ inner product on $\mathbb{T}$; this trivialisation is uniquely determined up to the action of $\mathrm{SU}(p+p'+1,q+q'+1)$ on $\mathbb{T}$. The canonical null line subbundle $\cE_{\Si}(-1,0)$ of $\cT\Si$ gives rise to a null line subbundle $L=\cE_{\Si}(-1,0)\otimes \cR^*$ of $\Si\times \mathbb{T}$. The null line subbundle $L$ then gives rise to a smooth map into the model $(p+p'+1,q+q'+1)$ hyperquadric given by
\begin{equation}\label{BonnetTheoremMap}
\Si \ni x \mapsto L_x \subset \mathbb{T}=\mathbb{C}^{p+p'+1,q+q'+1}.
\end{equation}
Since the local trivialisation of $\left(\cT\Si\otimes\cR^*\right)\oplus \cN$ is uniquely determined up to the action of $\mathrm{SU}(\mathbb{T})$ the above displayed map from $\Si$ to $\mathcal{H}$ is determined up to automorphisms of $\mathcal{H}$. It remains to show that this map is a transversal CR embedding inducing the specified extrinsic data.

Let us denote the map \eqref{BonnetTheoremMap} by $f:\Si\rightarrow\mathcal{H}\subset \mathbb{P}(\mathbb{T})$. Given a nowhere zero local section $\rho$ of $L=\cE_{\Si}(-1,0)\otimes \cR^*$ we may think of the section $\rho Z^I$ of $\cT\Si\otimes\cR^*$ as a section of $\Si\times \mathbb{T}$ via inclusion; this section gives rise to a lifted map $f_\rho:\Si\rightarrow \mathbb{T}$. The tangent map of $f_{\rho}$ at $x\in \Si$ is given by
\begin{equation*}
T_x\Si \ni X \mapsto \nabla_X (\rho Z^I)\in \mathbb{T}. 
\end{equation*}
From \eqref{Soldering1} and \eqref{Soldering2} we have that $D_{\nub}Z^I=0$ and $D_{\nu}Z^I=W_{\nu}^I$ (fixing some background contact form on $\Si$); using these, the definition of $\nabla$, and the facts that $\sfS_{i J}{^K}Z^J \; \mathrm{mod} \; Z^K =0$ (since $\sfS$ is adjoint valued) and that $\bL$ annihilates the canonical tractor $Z^I$, we see that $T_x f_{\rho}$ restricted to contact directions is injective and induces a complex linear isomorphism of $H_x$ onto a subspace of $L^{\perp}_x/L_x$; combined with \eqref{Soldering3} we see that $T_x f_{\rho}$ is injective and its image is transverse to $L_x$. This implies that the composition $f$ of $f_{\rho}$ with the projectivisation map $\mathbb{T}\setminus \{0\}\rightarrow \mathbb{P}(\mathbb{T})$ is a local CR embedding into the model hyperquadric $\mathcal{H}$. Equation \eqref{Soldering3} further shows that $T_x f_{\rho}(T_x\Si)\not\subset L^{\perp}_x$, so $f$ is transversal.

To see that this embedding induces back the specified extrinsic data we simply need to note that we may identify $\left(\cT\Si\otimes\cR^*\right)\oplus \cN=\Si\times\mathbb{T}$ with $\cT \mathcal{H}|_{\Si}$, identifying $\nabla$ with the flat tractor connection on $\cT \mathcal{H}|_{\Si}$ and $h$ with the tractor metric $h^{\cT \mathcal{H}}$ along $\Si$. Then 
\begin{equation*}
\cT \mathcal{H}|_{\Si} \;\; = \; \begin{array}{c} 
\mathcal{T}\Sigma \otimes \mathcal{R}^*  \\
\oplus \\
\mathcal{N}\\
\end{array}
\end{equation*}
is the usual decomposition of the ambient tractor bundle along the submanifold,
and the definition of $\nabla$ in the statement of the theorem gives the usual decomposition of the ambient tractor connection.
\end{proof}

Our formulation and proof of this CR Bonnet theorem is inspired by the conformal Bonnet theorem formulated and proved in terms of standard conformal tractors by Burstall and Calderbank in \cite{BurstallCalderbank-Conformal}. The condition that the connection $\nabla$ we define be flat is alternatively given in terms of the prescribed data on $(\Si,H_{\Si},J_{\Si})$ by the tractor Gauss, Codazzi, and Ricci equations \eqref{CRtractorGauss}, \eqref{CRtractorCodazzi}, and \eqref{CRtractorRicci} with the left hand sides equal to zero.

\end{document}